\documentclass[11pt, reqno, a4paper]{amsart}
\usepackage{amsmath,amssymb,amsthm}
\usepackage{color}
\usepackage{comment}
\usepackage[utf8]{inputenc}
\theoremstyle{definition}
 \newtheorem{dfn}{Definition}
 \newtheorem{remark}[dfn]{Remark}

\theoremstyle{plain}
 \newtheorem{thm}[dfn]{Theorem}
 \newtheorem{prop}[dfn]{Proposition}
 \newtheorem{lem}[dfn]{Lemma}
 \newtheorem{cor}[dfn]{Corollary}
 \newtheorem{assump}[dfn]{Assumption}

\numberwithin{equation}{section}
\newcommand{\bn}{{\bold n}}

\newcommand{\bm}{{\bold m}}
\newcommand{\bu}{{\bold u}}
\newcommand{\bz}{{\bold z}}
\newcommand{\bv}{{\bold v}}
\newcommand{\bw}{{\bold w}}

\newcommand{\bff}{{\bold f}}

\newcommand{\bG}{{\bold G}}

\newcommand{\bI}{{\bold I}}

\newcommand{\dv}{{\rm div}\,}

\newcommand{\BA}{{\Bbb A}}

\newcommand{\BR}{{\Bbb R}}
\newcommand{\BC}{{\Bbb C}}
\newcommand{\BM}{{\Bbb M}}
\newcommand{\BN}{{\Bbb N}}

\newcommand{\BI}{{\Bbb I}}

\newcommand{\BK}{{\Bbb K}}

\newcommand{\BZ}{{\Bbb Z}}

\newcommand{\CA}{{\mathcal A}}
\newcommand{\CB}{{\mathcal B}}
\newcommand{\CC}{{\mathcal C}}
\newcommand{\CD}{{\mathcal D}}

\newcommand{\CF}{{\mathcal F}}

\newcommand{\CL}{{\mathcal L}}
\newcommand{\CM}{{\mathcal M}}
\newcommand{\CN}{{\mathcal N}}

\newcommand{\CR}{{\mathcal R}}
\newcommand{\CS}{{\mathcal S}}
\newcommand{\CT}{{\mathcal T}}
\newcommand{\CH}{{\mathcal H}}

\newcommand{\CP}{{\mathcal P}}

\newcommand{\CU}{{\mathcal U}}

\newcommand{\bg}{{\bold g}}
\newcommand{\bh}{{\bold h}}
\newcommand{\pd}{\partial}

\newcommand{\supp}{{\rm supp\,}}
\newcommand{\HS}{\BR^N_+}

\renewcommand{\d}{{\rm d}}

\begin{document}
\title[]{$L_1$ approach to  the compressible viscous fluid flows 
in general domains}

\author[]{Jou-Chun Kuo}
\address{(J.-C. Kuo) School of Science and Engineering,
Waseda University \\
3-4-5 Ohkubo Shinjuku-ku, Tokyo, 169-8555, Japan }		
\email{kuojouchun@asagi.waseda.jp} 
\author[]{Yoshihiro Shibata}
\address{(Y. Shibata) 
{ Professor Emeritus of Waseda University \\
3-4-5 Ohkubo Shinjuku-ku, Tokyo, 169-8555, Japan. \\
Adjunct faculty member in the Department of Mechanical Engineering 
and	Materials Science, University of Pittsburgh, USA
}}
\email{yshibata325@gmail.com}

\thanks{{MSC Classification of 2020: Primary: 76N10; Secondary: 35Q30. }
\\
\quad \enskip
The first author is supported by JST SPRING, 
Grant Number JPMJSP2128.  The second author is  partially supported by JSPS KAKENHI 
Grant Number 22H01134}


\date{\today}   

\maketitle


\begin{abstract}
This paper is concerned with the $L_1$ in time $B^s_{q,1}$ in space maximal regularity
for the Stokes equations obtained by linearization procedure of the Navier-Stokes
equations describing the viscous compressible fluid motion. Our main tool of deriving
this maximal regularity is based on the spectral analysis of the corresponding resolvent problem
for the Stokes operators.  An applications of our theorem is to prove the
local well-posedness of the Navier-Stokes equations with non-slip boundary 
conditions in uniform $C^3$ domains, whose boundary is compact. This is an 
extension of results due to Danchin-Tolksdorf \cite{DT22}, where  
the boundedness of the domain is assumed.  
In this paper, we assume that the boundary of the domain is compact, namely,
not only bounded domains but also exterior domains are considered.  Our approach 
of this paper is based  on the spectral analysis of Lam\'e equations, while 
 the method in \cite{DT22} is an extension of 
a result due to Da Prato-Grisvard \cite{DG75}. Our method developed in this
paper has applications to extensive system of parabolic and hyperbolic-parabolic 
equations with non-homogeneous boundary conditions. 
\end{abstract}
\tableofcontents
\section{Introduction} \label{sec:1}

Let $\Omega$ be a domain in the $N$ dimensional Euclidean space $\BR^N$,
whose boundary $\pd\Omega$ is a $C^3$ compact hypersurface.  In partucular,
$\Omega$ is a bounded domain or an exterior domain. 
In this paper, we consider the Navier-Stokes equations describing the 
viscous compressible fluid motion with homogeneous Dirichlet boundary conditions, 
which read as 
\begin{equation}
\label{Eq:Nonlinear}
\left\{\begin{aligned}
\pd_t\varrho + \dv(\varrho\bv)= 0&&\quad&\text{in $\Omega\times(0, T)$}, \\
\varrho(\pd_t\bv+(\bv \cdot \nabla)\bv)-
\mu\Delta\bv-(\mu+\nu)\nabla\dv \bv+\nabla P(\varrho)=0
&&\quad&\text{in $\Omega\times(0, T)$},\\
\bv=0&&\quad&\text{on $\pd\Omega\times(0, T)$},\\
(\varrho,\bv)(0,x)=(\varrho_0,\bv_0)&&\quad&\text{in $\Omega$},
\end{aligned}\right.
\end{equation} 
and its 
linearized system  called here
 the generalized Stokes equations, which reads as 
\begin{equation}
\label{Eq:Linear}
\left\{\begin{aligned}
\pd_t \rho + \eta_0\dv\bv  = F& &\quad &\text{in $\Omega\times(0, T)$}, \\
\eta_0\pd_t\bv - \alpha\Delta\bv - \beta\nabla\dv\bv + \nabla(P'(\eta_0)\rho)
= \bG &&\quad &\text{in $\Omega\times(0, T)$},\\
\bv  = 0 &&\quad &\text{on $\pd\Omega\times(0, \infty)$},\\
(\rho,\bv)(0,x)  = (\rho_0,\bv_0) &&\quad &\text{in $\Omega$}.
\end{aligned}\right.
\end{equation} 
Here, $\rho$ and $\bv=(v_1,\cdots,v_N)$ are unkown functions, 
while the initial datum $(\rho_0, \bu_0)$ is assumed to be given.
In \eqref{Eq:Linear}, the right member $F$ and $\bG$ are also given functions. 
The coefficients $\mu$ and $\nu$ in \eqref{Eq:Nonlinear} are assumed to satisfy
 the ellipticity conditions $\mu>0$ and $\mu+\nu>0$. 
The coefficients $\alpha$ and $\beta$  in \eqref{Eq:Linear} 
are also  assumed to be 
constants such that $\alpha > 0$ and $\alpha+\beta>0$. 
As discussed in \cite[Sec.8]{ES13}, the coefficients $\alpha$ and 
 $\beta$  are 
defined by $\alpha=\mu/\rho_*$ and $\beta=\nu/\rho_*$,  respectively. 
Here, the $\rho_*$ is a positive constant describing the mass density of
the reference body.
In \eqref{Eq:Linear}, the coefficient  $\eta_0$ 
is a given function of the form: $\eta_0(x)=\rho_* + \tilde\eta_0(x)$, 
which appears in the linearized procedure at the initial data $\rho_0(x)$ 
which is very close to $\eta_0$. The reason why we call equations \eqref{Eq:Linear}
generlaized is that the coefficient $\eta_0$ depends on $x\in \Omega$. 
 The pressure of the fluid is given by a smooth function 
$P=P(\rho)$ defined for $\rho \in (0, \infty)$ such that $P'(\rho) > 0$.
Throughout the paper, we assume that there exist two positive numbers
$\rho_1 < \rho_2$ such that there hold
\begin{equation}\label{july.2}
\rho_1 < \rho_*  < \rho_2, \quad \rho_1 < \eta_0(x) < \rho_2, \quad 
\rho_1 < P'(\rho_*) < \rho_2, \quad \rho_1 < P'(\eta_0(x))< \rho_2
\end{equation}
for $x \in \Omega$. 
\subsection{$L_1$ maximal regularity for generalized Stokes equations.}
Our main result for the  linear problem \eqref{Eq:Linear} is the following theorem. 
\begin{thm}\label{thm:main.1} 
\thetag1 If $\eta_0=\rho_*$, then $1 < q < \infty$ and $-1+1/q < s < 1/q$. 
\thetag2 If $\tilde\eta_0\not\equiv0$ and $\tilde\eta_0 \in B^{s+1}_{q,1}(\Omega)$,
 then $N-1 < q < 2N$ and $-1+N/q \leq s < 1/q$.
Assume that the conditions \eqref{july.2} holds. 
Let $T > 0$. Then, there exists a positive constant $\gamma_0$ such that 
for any initial data $(\rho_0, \bu_0) \in \CH^s_{q,1}(\Omega)$  and 
right members $F\in L_1((0, T), B^{s+1}_{q,1}(\Omega))$ and 
$\bG \in L_1((0, T), B^s_{q,1}(\Omega)^N)$, problem \eqref{Eq:Linear} admits
unique solutions $\rho$ and $\bu$ with 
$$\rho \in W^1_1((0, T), B^{s+1}_{q,1}(\Omega)), \quad
\bu \in L_1((0, T), B^{s+2}_{q,1}(\Omega)^N) \cap W^1_1((0, T), B^s_{q,1}(\Omega)^N)
$$
possessing the estimate:
\begin{align*}
&\|(\pd_t\rho, \rho)\|_{L_1((0, T), B^{s+1}_{q,1}(\Omega))}
+ \|\pd_t\bu\|_{L_1((0, T), B^s_{q,1}(\Omega))} 
+ \|\bu\|_{L_1((0, T), B^{s+2}_{q,1}(\Omega))} \\
&\quad \leq e^{\gamma T}(C\|(\rho_0, \bu_0)\|_{\CH^s_{q,1}(\Omega)} 
+ C(\rho_*, \|\tilde\eta_0\|_{B^{s+1}_{q,1}(\Omega)})
\|(F, \bG)\|_{L_1((0, T), \CH^s_{q,1}(\Omega)}).
\end{align*}
for any $\gamma \geq \gamma_0$.  Here, constants $C$ and 
$C(\rho_*, \|\tilde\eta_0\|_{B^{s+1}_{q,1}(\Omega)})$ are independent of $\gamma$
but depending on $\gamma_0$. 
\end{thm}
\begin{remark} In the theorem, 
$B^\nu_{q,p}(\Omega)$ denotes standard Besov spaces, whose definition
will be given in Subsection 2.2 below and $\CH^s_{q,1}(\Omega)
= B^{s+1}_{q,1}(\Omega)\times B^s_{q,1}(\Omega)$. 
Moreover,  $L_1((0, T), X)$ and $W^1_1((0, T), X)$ denote
the standard $X$-valued $L_1$ and $W^1_1$ spaces. 
\end{remark}
\subsection{The local well-posedness of the Navier-Stokes equations.}
\label{subsec.1.2}
To treat equations \eqref{Eq:Nonlinear}, according to Str\"ohmer \cite{St90}, we introduce
Lagrange transformation.  Let $\bv(x, t)$ be the velocity field in 
Eulaer coordinates $x=(x_1, \ldots x_N) \in \Omega$ and 
$x(y, t)$ be a solution of the Caucy problem:
$$\frac{dx}{dt} = \bv(x, t) \quad(t > 0),  \quad x|_{t=0} = y=(y_1, \ldots, y_N).$$
We go over Euler coordinates $x$ to Lagrange coordinates $y$, and then the connection 
between Euler coordinate and Lagrange coordinates can be written as 
\begin{equation}\label{lag.1}x = y + \int^t_0\bu(y, \tau)\,d\tau = X_{\bu}(y, t).
\end{equation}
We see that $\bu(y, t) = \bv(x, t) = \bv(X_{\bu}(y, t), t)$ and $(\pd_t + \bv\cdot\nabla)
\rho(x, t) = \pd_t\eta(y, t)$ with 
$\eta(y, t) = \rho(X_{\bu}(y, t), t)$. \par 
If we find a solution $\bu$ in $L_1((0, T), B^{s+2}_{q,1}(\Omega)) \cap 
W^1_1((0, T), B^s_{q,1}(\Omega))$ with $-1+N/q \leq s < 1/q$ and $N-1<q  < 2N$,
then the map $x= X_{\bu}(y, t)$ is $C^{1+\sigma}$diffeomorphism with some small $\sigma>0$.
Moreover, since the Jacobian matrix of transformation \eqref{lag.1}  is given by
\begin{equation*}\label{lag,.2}
\nabla_yX_{\bu}(y, t) = \BI + \int^t_0\nabla_y\bu(y, \tau)\,d\tau.
\end{equation*}
Thus, if  $\bu$ satisfies 
\begin{equation}
\label{lag.3}
\bigg\lVert\int_0^t\nabla\bu(\tau,\xi)\,d\tau\bigg\rVert_{L_\infty}\le c,
\end{equation} 
for some small constant $c>0$, 
then transformation \eqref{lag.1} 
 gives a $C^{1}$ one to one map. Moreover, using an idea due to
Str\"ohmer \cite{St89, St90},
we see that this map is a bijection from $\Omega$ onto $\Omega$ if 
$\bu|_{\pd\Omega}=0$. \par
Let 
$$\BA_{\bu}(y, t) = (\nabla_yX_{\bu}(y, t))^{-1}
= \sum_{\ell=0}^\infty\Bigl(-\int^t_0\nabla_y\bu(y. \tau)\,d\tau\Bigr)^\ell,
$$ 
and then $\nabla_x = \BA_\bu^\top \nabla_y$, where $A^\top$ denotes the 
transposed $A$.  From this formula, equations \eqref{Eq:Nonlinear} are transformed
into the following system of equations: 
\begin{equation}\label{ns:2}\left\{\begin{aligned}
\pd_t\eta+\eta\dv\bu = F(\eta, \bu)& 
&\quad&\text{in $\Omega\times(0, T)$}, \\
\eta\pd_t\bu - \alpha\Delta \bu  -\beta\nabla\dv\bu 
+  \nabla P(\eta)
 = 
 \bG(\eta, \bu)& &\quad&\text{in $\Omega\times(0, T)$}, \\
\bu|_{\pd\Omega} =0, \quad (\eta, \bu)|_{t=0} = (\rho_0, \bu_0)
& &\quad&\text{in $\Omega$}.
\end{aligned}\right.\end{equation}
Here, we have set 
\begin{align*}
F(\eta, \bu) & =  \rho((\BI-\BA_\bu):\nabla\bu) 
\\
\bG(\eta,  \bu)& = 
(\BI-(\BA_\bu^\top)^{-1})(\rho\pd_t\bu -\alpha\Delta\bu)  
 +\alpha(\BA_\bu^\top)^{-1}\dv((\BA_\bu\BA_\bu^\top-\BI):\nabla\bu)\\
&+\beta\nabla((\BA_\bu^\top-\BI):\nabla\bu).
\end{align*}
By Theorem \ref{thm:main.1},  we have the following local well-posedness of
 equations \eqref{ns:2}. 
\begin{thm}\label{thm:main.2} Let $N-1 < q < \infty$ and $-1+N/q \leq s < 1/q$. 
Let $\rho_*$, $\tilde\eta_0(x)$, and $\eta_0(x)$ be the same as in
Theorem \ref{thm:main.1}. 
Then, there exist constants ${\sigma_0}>0$ and $T>0$ such that 
for any initial data $\rho_0 \in B^{s+1}_{q,1}(\Omega)$
and $\bu_0 \in B^s_{q,1}(\Omega)^N$,   problem \eqref{ns:2} admits 
 unique solutions $\rho$ and $\bu$ satisfying
the regularity conditions:
\begin{equation}\label{solclass}
\eta-\rho_0 \in W^1_1((0, T), B^{s+1}_{q,1}(\Omega)), 
\quad 
\bu \in L_1((0, T), B^{s+2}_{q,1}(\Omega)^N) \cap 
W^1_1((0, T), B^s_{q,1}(\Omega)^N)
\end{equation}
provided that  $\|\rho_0 - \eta_0\|_{B^{s+1}_{q,1}(\Omega)}\leq {\sigma_0}$.
\end{thm}
\begin{proof} We can prove the theorem employing the same 
argument as in Kuo-Shibata \cite{KSprep} 
replacing the half space with $\Omega$, and so we may omit the proof.
\end{proof}
\begin{cor}\label{cor:main.2} Let $N-1 < q < \infty$ and $-1+N/q \leq s < 1/q$. 
Let $\rho_*$, $\tilde\eta_0(x)$, and $\eta_0(x)$ be the same as in
Theorem \ref{thm:main.1}. 
Then, there exist constants ${\sigma_0}>0$ and $T>0$ such that 
for any initial data $\rho_0 \in B^{s+1}_{q,1}(\Omega)$
and $\bu_0 \in B^s_{q,1}(\Omega)^N$,   problem \eqref{Eq:Nonlinear} admits 
 unique solutions $\rho$ and $\bu$ satisfying
the regularity conditions:
\begin{equation}\label{solclass*}\begin{aligned}
\rho-\rho_0 &\in W^1_1((0, T), B^s_{q,1}(\Omega)) \cap L_1((0, T), B^{s+1}_{q,1}(\Omega)),
\\
\bv &\in  L_1((0, T), B^{s+2}_{q,1}(\Omega)^N) \cap 
W^1_1((0, T), B^s_{q,1}(\Omega)^N)
\end{aligned}\end{equation}
provided that $\|\eta_0-\rho_0\|_{B^{s+1}_{q,1}(\Omega)} \leq \sigma_0$.
\end{cor}
\begin{proof}  From \eqref{solclass}, we see that
$\bu \in L_1 ((0, T), \mathrm{BC}^1 (\overline{\HS})^d)$,
because $B^{N/q}_{q,1}(\Omega)$ is continuously imbedded into $L_\infty(\Omega)$.
As already mentioned, using a similar argument as in \cite{St89, St90}, we see that
$x=X_{\bu} (y, , t)$ is a $C^1$-diffeomorphism from $\Omega$ onto $\Omega$ 
for every $t \in [0, T)$ if \eqref{lag.3} holds.   
\par
For any function $F \in B^s_{q, 1} (\Omega)$, $1 < q < \infty$, 
$- \min (d \slash q, d \slash q') < s \le d \slash q$,
it follows from the chain rule (and the transformation 
rule for integrals) that
\begin{equation*}
\label{norm-equivalence}
\lVert F \circ X_{\bu}^{- 1} \rVert_{\CB^s_{q, 1} (\Omega)}
\le C \lVert F \rVert_{\CB^s_{q, 1} (\Omega)}
\end{equation*}
with a constant $C > 0$. This fact may be proved along the
same way as in the discussion given in Section 8.3 in \cite{DHMT}.
Thus, using Theorem \ref{thm:main.2}, we see that the original equations \eqref{Eq:Nonlinear}
admit solutions $\rho$ and $\bv$ possessing the estimate \eqref{solclass*}.
\end{proof}
\begin{remark}
R. Danchin and R. Tolksdorf \cite{DT22} proved the local and global well-posedness 
of equations \eqref{Eq:Nonlinear} 
in the $L_1$ in time and $B^{N/q}_{q,1} \times B^{N/q -1}_{q,1}$ in space 
maximal regularity framework for some $q \in (2, min(4, 2N/(N-2))$  
under the assumption that  the fluid domain $\Omega$ is bounded.  This assumption
is necessary to use 
 the Da Prato - Grisvard theory \cite{DG75}. Moreover,   
 they consider only the case where $s=N/q-1$ for their local well-posedness. 
Thus, Corollary \eqref{cor:main.2} is an extension of the result of the local
wellposedness by Danchin and Tolksdorf \cite{DT22}. \par
Our method to obtain the $L_1$ maximal regularity is completely different from
\cite{DG75, DT22}.  What is necessary for us  to obtain $L_1$ integrability
is spectral analysis.   It can be seen from 
Propositions \ref{prop:L1} and \ref{prop:L2}
in Sect. 3 below. Thus, 
the spectral properties of solutions to
equatons \eqref{Eq:Linear} play essential role and are  derived from the spectral properties 
of solutions to the Lam\'e equations, which read as
 \begin{equation}\label{Eq:Lame}\begin{aligned}
\eta_0\lambda\bv - \alpha\Delta \bv - \beta\nabla\dv\bv = f&&\quad\text{in 
$\Omega$}, \quad \bv|_{\pd\Omega}=0.
\end{aligned}\end{equation}
 Sect. 4  is devoted to driving the spectral properties 
of solutions to \eqref{Eq:Lame} . 
  \par
Since the global well-posedness for small initial data has been proved by \cite{DT22}
in the bounded domain case, we do not study the same problem in this paper. 
Concerning the global well-posedness for small initial data in
exterior domains,  we are interested in extending the result due to 
the second author \cite{S22}  in the $L_p$-$L_q$ framework ($1 < p, q < \infty$) 
to the $L_1$ in time maximal regularity framework. But, this is a future work. 
\end{remark}
\begin{remark}
Our essential assumption for domains 
is that the boundary is compact.  If we can prove that 
$$\sum_{\ell \in\BN} \|\varphi_j \bu\|_{B^s_{q,1}(\Omega)}^q 
\leq C\|\bu\|_{B^s_{q,1}(\Omega)}^q
$$
for some partition of unity $\{\varphi_j\}_{j\in \BN}$ in $\Omega$, we can treat the 
case where the boundary is non-compact.  This inequality holds for 
$L_q(\Omega)$. 
\end{remark}
\subsection{Short History}
Mathematical studies on the compressible Navier-Stokes equations started with 
the uniqueness results in a bounded domain by Graffi \cite{G53}, whose result is extended by Serrin \cite{S59}
in the sense that there is no assumption on the equation of state of the fluid. 
In the studies \cite{G53} and \cite{S59}, the fluid occupies a bounded domain surrounded by a smooth boundary.
A local in time existence theorem in H\"older continuous spaces was first proved by
Nash \cite{N62}, Itaya~\cite{I71,I76a}, and Vol'pert and Hudjaev \cite{VH}
 independently, for the whole space case.
As for the boundary value problem case, Tani~\cite{T77} proved a local in time existence theorem 
in a similar setting provided that a (bounded or unbounded) domain $\Omega$ has a smooth boundary.
In Sobolev-Slobodetskii spaces, the local existence was shown by Solonnikov \cite{Sol},
see also the work due to Danchin \cite{D00, D10} for an improvement of Solonnikov's result.
Matsumura and Nishida \cite{MN80, MN83} made a breakthrough
in proving the global well-posedness for small initial data using the energy method. 
This result was extended to the optimal resularity of initial data in the $L_2$ space 
 by Kawashita \cite{K}.  Kobayashi and Shibata \cite{KS1} improved the decay properties
of solutions in the exterior domains combining the energy method and $L_p$-$L_q$ decay 
properties of solutions to the linearized equations, where the condition:
$1 < p \leq 2 \leq q \leq \infty$ is assumed.  In the  no restrictions of
exponents case, 
so called the diffusion wave properties has been studied by Hopf-Zumbrun \cite{HZ}
and Liu and Wang \cite{LW}.  Kobayashi and Shibata \cite{KS1, KS2}  improved results
due to  \cite{HZ, LW}.  
On the basis of a different approach, Mucha and Zajaczkowski \cite{MZ} applied 
$L_p$-energy estimates to show the global existence theorem in the $L_p$ framework.

In the half space case, the decay properties were studied by Kagei-Kobayashi \cite{KK02, 
KK05}. 
The global well-posedness results were extensively studied  in the energy spaces of exterior
domains by \cite{SE18, ST04, VZ86, WT} 
and in the critical space of the whole space by \cite{AP07, CMZ,  CD10, DX, H11, H, LZ, O}. 
Valli \cite{Vali83} and Tsuda \cite{Tsuda16}
studied time periodic solutions in the $L_2$ framework for the bounded domains 
and for $\BR^N$, respectively.


Str\"ohmer \cite{St90} introduced  Lagrangian coordinates to rewrite the system
of equations \eqref{Eq:Nonlinear}. 
Thanks to this reformulation (see Subsec. 1.3), the convection term in the density equation, 
namely $\varrho\cdot\nabla\bv$, 
may be dropped off, so that the transformed system 
becomes the evolution equation of parabolic type, so called the Stokes system, 
and he used the semigroup theory to treat the Stokes system
in the $L_2$ framework.   Developing this research, 
the second author and Enomoto \cite{SE18} 
and the second author \cite{S22} used the $L_p$-$L_q$ maximal regularity for 
the Stokes system and they proved local well-posedness for any initial 
data and the global well-posedness for small initial data, where 
the class of initial data are $(\varrho_0-\rho_*,\bv) \in
B^{2(1-1/p)+1}_{q,p}\times B^{2(1-1/p)}_{q,p}$. \par

The $L_1$ in time maximal regularity approach to the Navier-Stokes 
equations was started by Danchin and Mucha \cite{DM15} 
  for the incompressible 
viscous fluid flows with non-slip conditions.
  Recently,  the global wellposedness
for the small initial data for 
the free boundary problem of the Navier-Stokes equations for the viscous
incompressible
fluid flow was investigated
by \cite{DHMT}, \cite{OS}, and \cite{SW1} in the half-space by using
the $L_1$ in time and $\dot B^s_{q,1}$ in space maximal regularity. 
As we already mentioned, 
for the Navier-Stokes equations describing the viscous compressible fluid motion 
\eqref{Eq:Nonlinear},
the $L_1$ in time  and $B^s_{q,1}$ in space maximal regularity
approach was first  investigated by Danchin-Torksdorf \cite{DT22} 
under the assumption that 
 the fluid domain is bounded, 
which is required to prove
their extension version of Da Prato-Grisvard theory \cite{DG75}. 
In this paper, we establish the  $L_1$ in time and 
$B^s_{q,1}$ space maximal regularity theorems for  equations
\eqref{Eq:Linear}, cf. Theorem \ref{thm:main.1},  and the local well-posedness of 
non-linear problem \eqref{Eq:Nonlinear} in exterior domains, cf.. Theorem \ref{thm:main.2}. 
Our method to prove $L_1$ integrability is given in Section 3,  which has been  
 investigated by \cite{S23} based on the spectral analysis.
 Our method can be used widely to 
show the $L_1$ maximal regularity for parabolic or hyperbolic-parabolic systme of 
equations with non-homogeneous boundar conditions. For example, the second author and 
Watanabe \cite{SW1} proved 
the $L_1$ maximal regularity for the Stokes equations with free boundary conditions
 by using the spectral analysis of solutions to the generalized resolvent problem and 
Proposition \ref{prop:L1} in Sect. 3 below. \par
\subsection{Why is the $L_1$ approach important ?}  If we use  Lagrange transformation
following  Str\"ohmer \cite{St89, St90}, then we have to require that 
the Jacobian of Lagrange transformation 
$I + \int_0^t\nabla\bu(y, \tau)\,d\tau$ is invertible, 
where $\bu(y, \tau)$ stands for the velocity field of a fluid particle at time $t$
which was located in $y$ at initial time $t=0$. 
Hence,  it is always crucial to get a control of 
$\int_0^t \nabla\bu(y, \tau)\, d\tau$ in a suitable norm.  
 In particular, it is necessary to 
find a small constant $c>0$ such that \eqref{lag.3} holds, 
which ensures that Lagrangian transformation is invertible. 

Moreover, in view of time trace,  if
the velocity field $\bu$ belongs to the maximal regularity class
$L_p((0, T), W^2_q(\Omega)^N) \cap W^1_p((0, T), L_q(\Omega)^N)$, then 
$\bu|_{t=0} \in B^{2(1-1/p)}_{q,p}(\Omega)^N$.  Thus, $p=1$ gives the minimal
regularity for the initial data.
From these points of view, it is worth while investigating
 the $L_1$ maximal regularity theorem
with $B^s_{q,1}(\Omega)$ in space with $-1+N/q \leq s < 1/q$ and 
$N-1 < q < 2N$.  These constraints for $q$ and $s$ are unavoidable 
and essentially depends on estimates of the product of 
functions using Besov norms obtained by Abidi and Paicu \cite{AP07}. In fact, 
if $\eta_0$ is a constant, then we can relax the condition that 
$1 < q < \infty$ and $-1+1/q < s < 1/q$ for the linear theory. 
\par 
\section{Preparations for latter sections}

\subsection{Symbols used throughout the paper}\label{sec:2}
Let us fix the symbols used in this paper. Let $\BR$, $\BN$, and $\BC$ be the set of all real, natural, 
complex numbers, respectively, while let $\BZ$ be the set of all integers. Moreover, 
$\mathbb K$ stands for either $\BR$ or $\BC$. 
Set $\BN_0 := \BN \cup \{0\}$. For multi-index $\kappa=(\kappa_1, \ldots, \kappa_n)
\in \BN_0^n$ 
and $x=(x_1, \ldots, x_n) \in \BR^n$, $\pd^\alpha = \pd_x^\alpha = 
\pd^{|\alpha|}/\pd x_1^{\alpha_1}\cdots\pd x_N^{\alpha_N}$ stands for 
standard partial derivatives of order $\alpha$, 
where $|\alpha| = \alpha_1+\cdots+\alpha_N$.  For the dual variable 
$\xi=(\xi_1, \ldots, \xi_n)\in \BR^n$, $D^\kappa_\xi = \pd^{|\kappa|}/\pd^{\kappa_1}\xi_1
\cdots \pd^{\kappa_n}\xi_n$. For differentiations, we also use symbols
$\nabla f= \{\pd^\kappa f \mid |\kappa|=1\}, \bar\nabla f = \{\pd^\kappa f \mid
|\kappa| \leq 1\}$, $\nabla^2f = \{\pd^\kappa f \mid |\kappa|=2\}$, 
$\bar\nabla^2f = \{\pd^\kappa f \mid |\kappa| \leq 2\}$. \par 
For $\epsilon \in (0, \pi/2)$ and $\lambda_0>0$, we define parabolic
sectors $\Sigma_\epsilon$ and $\Sigma_{\epsilon, \lambda_0}$ by
\begin{equation*}\label{sector}
\Sigma_\epsilon = \{\lambda \in \BC\setminus\{0\} \mid |\arg\lambda| 
\leq \pi-\epsilon\}, \quad
\Sigma_{\epsilon, \lambda_0} = \{\lambda \in \Sigma_\epsilon \mid 
|\lambda| \geq \lambda_0\}.
\end{equation*}
Let $\HS$ and $\pd\HS$ denote the half space and its boundary defined by
$$\HS = \{x=(x_1, \ldots, x_N) \in \BR^N \mid x_N > 0\},\quad
\pd\HS = \{x=(x_1, \ldots, x_N) \in \BR^N \mid x_N=0\}.$$
For $N\in \BN$ and a Banach space $X$ on $\BK$, let $\CS(\BR^N;X)$ be the Schwartz class of 
$X$-valued rapidly decreasing functions on $\BR^N$. We denote $\mathcal{S'}(\BR^N;X)$ by the space of $X$-valued tempered distributions, which means the set of all 
continuous linear mappings from $\CS(\BR^N)$ to  $X$.  For $N\in \BN$, 
we define the Fourier transform $f \mapsto \CF [f]$ from $\CS(\BR^N;X)$ onto itself and its inverse  as
\begin{equation*}
\CF [f](\xi):=\int_{\BR^{N}}f(x) e^{-ix\cdot\xi }\,dx, \qquad 
\CF_{\xi}^{-1} [g] (x) :=\frac{1}{(2\pi)^N}\int_{\BR^{N}}g(\xi) e^{ix\cdot\xi }\,d\xi,
\end{equation*}
respectively. In addition, we define the partial Fourier transform
$\CF'[f(\,\cdot\,, x_N)] = \hat f(\xi', x_N)$ 
and partial inverse Fourier transform $\CF^{-1}_{\xi'}$ by
\begin{align*}
\CF'[f(\,\cdot\,, x_N)] (\xi') &:= \hat f(\xi', x_N) = \int_{\BR^{N-1}} f(x', x_N) e^{-ix'\cdot\xi'} \,dx',\\
\CF^{-1}_{\xi'}[g (\,\cdot\,, x_N)](x') &:= \frac{1}{(2\pi)^{N-1}}\int_{\BR^{N-1}}
g(\xi', x_N) e^{ix'\cdot\xi'}\,d\xi',
\end{align*}
where we have set $x'=(x_1,\cdots,x_{N-1}) \in \BR^{N-1}$ and $\xi' = (\xi_1, \cdots, \xi_{N-1})\in\BR^{N-1}$, and the Laplace transform $\CL[f](\lambda)$ and inverse
Laplace transform $\CL^{-1}[g](t)$ by
$$\CL[f](\lambda) = \int_{\BR} e^{-\lambda t}f(\cdot, t)\,dt, \quad
\CL^{-1}[g](t) = \frac{1}{2\pi i}\int_{\BR} e^{\lambda t} g(\lambda)\,d\tau
\quad (\lambda = \gamma + i\tau).
$$
\par
For a  domain $D$ and a Banach space $X$ on $\BK$,  $L_p(D, X)$ and $W^m_p(D, X)$ 
stand for respective standard $X$valued Lebesgue spaces and  
 Sobolev spaces, while $\|\cdot\|_{L_p(D,X)}$ and $\|\cdot\|_{W^m_p(D,X)}$
denote their norms. When $X = \BR^N$, we omit $X=\BR^N$, namely, we write
$L_p(D)$, $W^m_p(D)$, $\|\cdot\|_{L_p(D)}$ and $\|\cdot\|_{W^m_p(D)}$. 
For a domain $D$ in $\BR^N$ and 
$N \ge 2$, we set $(\bff, \bg)_D = \int_D \bff (x) \cdot \bg (x) \,dx$ 
for $N$-vector functions $\bff$ and $\bg$ on $D$, 
where we will write $(\bff, \bg) = (\bff, \bg)_{D}$ for short if there is no confusion.
\par
For Banach spaces  $X$ and $Y$  on $\BK$,  $\CL(X, Y)$ denotes the set of all bounded
linear operators from $X$ into $Y$, and we write $\CL(X)=\CL(X,X)$.  
Let $X\times Y$ denotes the product of $X$ and $Y$,
that is $X\times Y = \{(x, y) \mid x \in X, \enskip y \in Y\}$, 
while  $\|(x, y)\|_{X\times Y}=\|x\|_X + \|y\|_Y$ denotes its norm,
where $\|\cdot\|_Z$ denotes the norm of $Z$ ( $Z \in \{X, Y\}$).
To denote $n$ product space of $X$,  we write $X^n
=\{x =(x_1, \ldots, x_n) \mid x_i \in X\, (i=1, \ldots, n)\}$, while its norm is denoted by 
$\|x\|_X= \sum_{i=1}^n\|x_i\|_X$. 
Let ${\rm Hol}\, (U, X)$ denote the set of 
all $X$ valued holomorphic functions defined on  a complex domain $U$. 
$X \hookrightarrow Y$ means that $X$ is continuously imbedded into $Y$,
 that is $X \subset Y$ and $\|x\|_Y \leq C\|x\|_X$ with some constant $C$. 
\par
For any interpolation couple $(X, Y)$ of Banach spaces $X$ and $Y$ on $\BK$, 
the operations $(X, Y) \to (X, Y)_{\theta, p}$ and 
$(X, Y) \to (X, Y)_{[\theta]}$ are 
called the real interpolation functor for each $\theta \in (0, 1)$ 
and $p \in [1, \infty]$ and the complex interpolation functor 
for each $\theta \in (0,1)$, respectively.
By $C > 0$ we will often denote a generic constant 
that does not depend on the quantities at stake.  And, by $C_{a, b, \cdots}$ we denote 
generic constants depending on the quantities $a$, $b$, $c, \cdots$.  $C$ and 
$C_{a, b, c, \cdots}$ may change from line to line. 

\subsection{Definition of Besov spaces and some properties} \label{subsec:2.1}

To define  Besov space $B^s_{q,r}$,  we  introduce Littlewood-Paley decomposition.
Let $\phi \in \CS (\BR^N)$ with $\supp \phi = \{ \xi \in \BR^N \mid 1 
\slash 2 \le \lvert \xi \rvert \le 2 \}$ such that
$\sum_{k \in \BZ} \phi (2^{- k} \xi) = 1$ for all $\xi \in \BR^N \setminus \{0\}$. 
Then, define
\begin{equation}\label{little:1}
\phi_k := \CF^{- 1}_\xi [\phi (2^{- k} \xi)] \quad (k \in \BZ), \qquad \psi =
1 - \sum_{k \in \BN} \phi (2^{- k} \xi).
\end{equation}
For $1 \le p, q \le \infty$ and $s \in \BR$ we denote
\begin{equation}\label{pld:2}
\lVert f \rVert_{B^s_{p, q} (\BR^N)} := 
\left\{\begin{aligned}
& \lVert \psi * f \rVert_{L_p (\BR^N)} + \bigg(\sum_{k \in \BN} \Big(2^{s k} \lVert \phi_k * f \rVert_{L_p (\BR^N)} \Big)^q \bigg)^{1 \slash q}
& \quad & \text{if $1\le q < \infty$}, \\
& \lVert \psi * f \rVert_{L_p (\BR^N)} + \sup_{k \in \BN} \Big(2^{s k} \lVert \phi_k * f \rVert_{L_p (\BR^N)} \Big)
& \quad & \text{if $q=\infty$}.
\end{aligned}\right.
\end{equation}
Here, $f \, * \, g$ means the convolution between $f$ and $g$.
Then Besov spaces $B^s_{p, q} (\BR^N)$ are defined as the sets of all 
$f \in \CS' (\BR^N)$ such that $\lVert f \rVert_{B^s_{p, q} (\BR^N)} < \infty$. 
In particular, 
$$B^s_{q, \infty-}(\BR^N) = \{g \in B^s_{q, \infty}(\BR^N) \mid 
\lim_{k\to\infty} 2^{sk}\|\phi_k*f\|_{L_q(\BR^N)} = 0\}.
$$
When $1 \leq r \leq \infty-$, we define $r'$ by $1' = \infty-$, ${\infty-}' = 1$ and $r'
= r/(r-1)$ for $1 < r < \infty$. 
\par
For any domain $D$ in $\BR^N$,  $B^s_{q,r}(D)$ is defined by the restriction of 
$B^s_{q,r}(\BR^N)$, that is 
\begin{gather*}B^s_{q,r}(D) = \{f \in \CD'(D) \mid \text{there exists a $g\in B^s_{q,r}(\BR^N)$
such that $g|_{D} = f$}\}, \\
\|f\|_{B^s_{q,r}(D)} = \inf \{\|g\|_{B^s_{q,r}(\BR^N)} \mid 
g \in B^s_{q,r}(\BR^N), \enskip g|_{D} = f\}.
\end{gather*}
Here, $\CD'(\Omega)$ denotes the set of all distributions on $D$ and $g|_{D}$ denotes the 
restriction of $g$ to $D$.  
\par 
It is well-known that $B^s_{p,q}(D)$ may be \textit{characterized} by means of real interpolation.
In fact, for $-\infty<s_0<s_1<\infty$, $1<p<\infty$, $1\le q\le\infty$, and $0<\theta<1$, it follows that
\begin{equation*}
\label{interpolation}
B^{\theta s_0+(1-\theta)s_1}_{p,q}(D)=\left(H^{s_0}_p(D),H^{s_1}_p(D)\right)_{\theta,q},
\end{equation*}
cf. \cite[Theorem 8]{M74},  \cite[Theorem 2.4.2]{Tbook78}. Here, the real interpolation functors
are denoted by $(\cdot, \cdot)_{\theta, q}$.

\subsection{Estimates of  products  and 
 composite  functions using Besov norms} \label{subsec:2.2}
We use the following lemma concerning the estimate of product 
using the Besov norms. 
\begin{lem}\label{lem:prod} Let $D$ be a uniform $C^3$ domain whose boundary is a
compact hypersurce. 
Let  $N-1 < q  <2N$, $1 \leq r \leq \infty$ and $-1+N/q \leq  s < 1/q$.  Then, 
for any $u \in B^s_{q,r}(D)$ and $v \in B^{N/q}_{q,\infty}(D) \cap L_\infty(D))$, there holds
\begin{equation}\label{prod.1}
\|uv\|_{B^s_{q,r}(D)} \leq C_{D,s,q,r}\|u\|_{B^s_{q,r}(D)}\|v\|_{B^{N/q}_{q, \infty} \cap L_\infty(D)}.
\end{equation}
\end{lem}
\begin{proof} By the Abidi-Paicu estimate \cite{AP07} and the Haspot estimate
\cite{H11}, when $2< q < \infty$ and 
$-N/q < s < N/q$ or when $1 \leq q < 2$ and $-N/q' < s < N/q$,  
the estimate \eqref{prod.1}
holds. When $2 \leq q < \infty$, the condition: $-N/q < -1+N/q$ implies that 
$q <2N$.  When $1 \leq q < 2$, the condition: $-N/q' \leq -1 + N/q$ implies that 
$N \geq 1$.  The condition: $N-1 < q$ follows from the condition:
$-1 + N/q < 1/q$. The proof is completed. 
\end{proof}
The following lemma is concerned with the estimate of composite functions using
 Besov norms , cf. 
\cite[Proposition 2.4]{H11} and \cite[Theorem 2.87]{BCD}.
\begin{lem}\label{lem:Hasp} Let $1 < q < \infty$. 
Let $I$ be an open interval of $\BR$.  Let $\omega>0$ and let $\tilde\omega$ be the smallest
integer such that $\tilde\omega \geq \omega$. Let $F:I \to \BR$ satisfy $F(0) = 0$ and 
$F' \in BC^{\tilde\omega}(I, \BR)$.  Assume that $v\in B^\omega_{q,r}$
has valued in $J \subset\subset I$.  Then, 
$F(v) \in B^\omega_{q,1}$ and there exists a constant $C$ depending only on 
$\nu$, $I$, $J$, and $N$, such that 
$$\|F(v)\|_{B^\omega_{q,1}} \leq C(1 + \|v\|_{L_\infty})^{\tilde\omega}
\|F'\|_{BC^{\tilde\omega}(I,\BR)}
\|v\|_{B^\omega_{q,1}}.$$
\end{lem}
\subsection{Fourier multiplier theorems in $\BR^N$} \label{sec:2.3}
To estimate solution formulas in $\BR^N$, we use the following 
Fourier multiplier theorem of Mihlin - H\"ormander type \cite{Mih, Hor}.
  Let 
$m(\xi)$ be a $C^\infty(\BR^N)$ function such that for any multi-index $\kappa
\in \BN_0^{N}$ there exists a constant $C_\alpha$ such that 
$$|D_\xi^\kappa m(\xi)| \leq C_\alpha|\xi|^{-|\kappa|}.$$
We call $m$ a multiplier symbol  of order $0$.  Set 
$[m] = \max_{|\kappa| \leq N}C_{\kappa}$.  For any multiplier  of order $0$,
we define an operator $T_m$ by 
$$T_mf = \CF^{-1}[m \CF[f]].$$  
We call $T_m$ the Fourier multiplier with symbol $m$.  We know the following
Fourier multiplier theorem of 
Mihlin-H\"ormander type.
\begin{prop}\label{FMT} Let $1<q<\infty$, $1 \leq r \leq \infty$ and $s\in\BR$. 
Let $T_m$ be a Fourier multiplier with symbol $m$.  Then, for any 
$f \in B^s_{q,r}(\BR^n)$, there holds
$$\|T_mf\|_{B^s_{q,r}(\BR^N)} \leq C_q[m]\|f\|_{B^s_{q,r}(\BR^N)}$$
with some constant $C_q$ depending solely on $q$.
\end{prop}
\begin{proof}
Let $\phi_k$ and $\psi$ be functions introduced in \eqref{little:1} to define the Littlewood-
Paley decomposition. 
Let $m(\xi)$ be a multiplier symbol of order $0$, and then 
$\phi_km$ and $\psi m$ are also multiplier symbols of order $0$.
By  the standard Fourier multiplier
theorem of Mihlin-H\"ormander type, we have
\begin{align*}
\|\phi_k*(T_mf) \|_{L_q(\BR^N)} &\leq C[m]\|\phi_k*f\|_{L_q(\BR^N)}, \\
\|\psi *( T_mf)\|_{L_p(\BR^N)} &\leq C[m]\|\psi*f\|_{L_q(\BR^N)}.
\end{align*}
Thus, by the definitions of the Besov norms \eqref{pld:2}, we have
$$\|T_mf\|_{B^s_{q,r}(\BR^N)} \leq C[m]\|f\|_{B^s_{q,r}(\BR^N)}.$$
This completes the proof of Proposition \ref{FMT}. 
\end{proof}
\subsection{Symbol classes and estimates of the integral operators in $\HS$} \label{subsec:2.4}

Let $\Sigma_{\epsilon, \lambda_0}$ be a sector defined by 
\begin{equation*}\label{sector}
\Sigma_{\epsilon, \lambda_0} = \{\lambda \in \BC\setminus\{0\} \mid 
|\arg\lambda| \leq \pi-\epsilon, \quad |\lambda| \geq \lambda_0\}
\end{equation*}
for $\epsilon \in (0, \pi/2)$ and $\lambda_0 > 0$, cf. Subsec. \ref{sec:2}.
We introduce symbol classes used to represent solution formulas
in $\HS$.  Let $m(\lambda, \xi')$ be a function defined on $\Lambda_{\epsilon, \lambda_0}\times
(\BR^{N-1}\setminus\{0\})$ such that for each $\xi' \in \BR^{N-1}$ $m(\lambda, \xi')$ is  
holomorphic with respect to $\lambda \in \Lambda_{\epsilon, \lambda_0}$ and for 
each $\lambda \in \Lambda_{\epsilon, \lambda_0}$ $m(\lambda, \xi')
\in C^\infty(\BR^{N-1}\setminus
\{0\})$.  Let $\ell \in \BZ$. We say that  $m(\lambda, \xi')$ is an order $\ell$  symbol 
if for any $\kappa' \in \BN_0^{N-1}$
and $\lambda \in \Lambda_{\epsilon}$ there exists a constant $C_{\kappa'}$ 
depending on $\kappa'$, $\epsilon$, $\lambda_0$ and $\ell$ such that 
\begin{equation*}\label{multi.1}
|D_{\xi}^{\kappa'} m(\lambda, \xi')| \leq C_{\kappa'}(|\lambda|^{1/2}+|\xi'|)^{\ell-|\kappa'|}.
\end{equation*}
Let 
$$\|m\| = \max_{|\kappa'| \leq N} C_{\kappa'}.$$
We can show the following two propositions using the same argument as in the proof of
Lemma 4.4 in Enomoto-Shibata \cite{ES13}.
\begin{prop} \label{prop:2}
Let $1 < q < \infty$, $\epsilon \in (0, \pi/2)$, $\lambda_0>0$, and 
$\lambda\in\Lambda_{\epsilon, \lambda_0}$. Let  $m_0(\lambda, \xi') \in \BM_0$. Set 
$$M(x_N) = \frac{e^{-Bx_N}-e^{-Ax_N}}{B-A}.$$
Define the integral operators $L_i$, $i=1, 2$, by the formula:
\allowdisplaybreaks
\begin{align*}
L_1(\lambda)f &= \int^\infty_0\CF^{-1}_{\xi'}\left[m_0(\lambda, \xi')
Be^{-B(x_N+y_N)}\CF'[f](\xi', y_N)\right](x')\,dy_N, \\
L_2(\lambda)f &= \int^\infty_0\CF^{-1}_{\xi'}\left[m_0(\lambda, \xi')
B^2M(x_N + y_N)\CF'[f](\xi', y_N)\right](x')\,dy_N, \\
L_3(\lambda)f &= \int^\infty_0\CF^{-1}_{\xi'}\left[m_0(\lambda, \xi')
B^2\pd_\lambda(Be^{-B(x_N+y_N)})\CF'[f](\xi', y_N)\right](x')\,dy_N, \\
L_4(\lambda)f &= \int^\infty_0\CF^{-1}_{\xi'}\left[m_0(\lambda, \xi')
B^2\pd_\lambda(B^2M(x_N + y_N))\CF'[f](\xi', y_N)\right](x')\,dy_N,
\end{align*}
respectively. 
Then for every $f\in L_q(\HS)$, it holds
\begin{equation*}
\|L_i(\lambda)f\|_{L_q(\HS)} \leq C_q\|m_0\|\|f\|_{L_q(\HS)}\quad (i=1,2,3,4).
\end{equation*}
\end{prop}

\section{$L_1$ integrability of Laplace inverse transformation}\label{sec.3}

In this section, we consider the $L_1$ integrability of solutions to equations
\eqref{Eq:Linear}, which is treated as a perturbation 
of Lam\'e equations with Dirichlet conditions.   The solution to the 
time dependent problem is represented by the Laplace transform of
the solutions to the corresponding resolvent problem.  
Thus, in this section, we consider the Laplace inverse transform of 
operators holomorphically depending on the spectral parameter 
$\lambda \in \Sigma_{\epsilon, \lambda_0}$ with 
$0 < \epsilon < \pi/2$ and $\lambda_0>0$, and we shall give spectral 
properties which guarantees the $L_1$ integrability of the Laplace inverse transform.
\begin{dfn}\label{dfn:L1int} Let $D$ be a domain in $\BR^N$ 
Let $\lambda_0>0$ and $0 < \epsilon < \pi/2$.
Let $1 < q < \infty$, $1 \leq r \leq \infty-$,  and $-1+1/q < s < 1/q$.  Let $\sigma > 0$ 
be a small number such that $-1+1/q < s-\sigma < s+\sigma < 1/q$.
Let $\nu \in \{s-\sigma, s, s+\sigma\}$.
Let $\CN \in {\rm Hol}\,(\Sigma_{\epsilon, \lambda_0}, 
\CL(B^\nu_{q,r}(D), B^{\nu+2}_{q,r}(D))$.   We say that
$\CN$ has  $(s, \sigma, q, r)$  properties in $D$ if for any
$\lambda \in \Sigma_{\epsilon, \lambda_0}$ there hold
\begin{equation}\label{L1prop}\begin{aligned}
\|(\lambda, \lambda^{1/2}\bar\nabla,
\bar\nabla^2)\pd_\lambda^\ell \CN(\lambda)g\|_{B^\nu_{q,r}(D)}
& \leq C|\lambda|^{-\ell}\|g\|_{B^\nu_{q,r}(D)} \quad(\ell=0,1), \\
\|(\lambda^{1/2}\bar\nabla, \bar\nabla^2)\CN(\lambda)g
\|_{B^s_{q,r}(D)}& \leq C|\lambda|^{-\frac{\sigma}{2}}
\|g\|_{B^{s+\sigma}_{q,r}(D)} \\
\|(1, \lambda^{-1/2}\bar\nabla)\CN(\lambda)g
\|_{B^s_{q,r}(D)} & \leq C|\lambda|^{-(1-\frac{\sigma}{2})}
\|g\|_{B^{s-\sigma}_{q,r}(D)}, \\
\|(\lambda, \lambda^{1/2}\bar\nabla, \bar\nabla^2)\pd_\lambda\CN(\lambda)g
\|_{B^s_{q,r}(D)}& \leq C|\lambda|^{-(1-\frac{\sigma}{2})}
\|g\|_{B^{s-\sigma}_{q,r}(D)}
\end{aligned}\end{equation}
provided that $g\in B^{s+\sigma}_{q,r}(D)$. 
\end{dfn}
\begin{remark} \thetag1 Since $s-\sigma < s < s+\sigma$, that 
$g\in B^{s+\sigma}_{q,1}(\Omega)$ implies that 
$g\in B^{\nu}_{q,1}(\Omega)$ for $\nu=s$ and $\nu=s-\sigma$.\\
\thetag2 To prove the $L_1$ integrability of the Laplace inverse transform of
$\CN(\lambda)$, it is enough to consider the $r=1$ case. But, 
as spectral properties of operators, we consider the case where $1 \leq r \leq \infty-$. The reason
why we use $\infty-$ instead of $\infty$ is that the density argument does not
hold in case $r=\infty$. 
\end{remark}
We consider the $L_1$ integrability of the Laplace transform of $\CN$.  Let 
\begin{equation*}\label{laplace:Lame}
N(t)g = \CL^{-1}[\CN(\lambda)g](t).
\end{equation*}
\begin{prop}\label{prop:L1}Let $\epsilon \in (0, \pi/2)$ and $D$ be a domain 
in $\BR^N$.  Let 
$1 < q < \infty$, $-1+1/q < s < 1/q$, and $\lambda_0>0$. 
 Let $\sigma>0$ be a number such that 
$-1+1/q < s-\sigma < s+\sigma < 1/q$.  
Assume that $C^\infty_0(D)$ is dense in $B^\nu_{q,1}(D)$
for $\nu \in \{s-\sigma, s, s+\sigma\}$. Let 
 $\CN(\lambda)
\in {\rm Hol}\,(\Sigma_{\epsilon, \lambda_0}, \CL(B^s_{q,1}(D), 
B^{s+2}_{q,1}(D))$ be an operator having $(s, \sigma, q, 1)$ properties in $D$.
Then,  
$N(t)g=0$ for $t < 0$ and 
$e^{-\gamma t}N(t)g \in L_1(\BR, B^{s+2}_{q,1}(D))$ possessing the estimate
\begin{equation}\label{L1}
\int^\infty_0e^{-\gamma t}\|N(t)g\|_{B^{s+2}_{q,1}(D)}\,dt 
\leq C\|g\|_{B^s_{q,1}(D)}
\end{equation}
for any $g \in B^s_{q,1}(D)$ and $\gamma \geq \lambda_0$.  Here, the constant $C$
depends on $\lambda_0$ but is independent of $\gamma\geq\lambda_0$.
\end{prop}
\begin{remark} The condition that $C^\infty_0(D)$ is dense in $B^\nu_{q,r}(D)$
 holds for $\nu \in \{s-\sigma, s, s+\sigma\}$, $-1 +1/q < \nu <1/q $ and $1 \leq  r \leq \infty-$
at least 
in case of  $\BR^N$, $\HS$,  bent half-spaces
and $C^2$ domains. 
\end{remark}
\begin{proof}
Since $C^\infty_0(D)$ is dense in $B^{s+\sigma}_{q,1}(D)$ and 
$B^s_{q,1}(D)$, we may assume that $g \in C^\infty_0(D)^N$ 
below. 
First, we shall show that 
\begin{equation}\label{L1.1}
N(t)g = 0 \quad\text{for $t < 0$}.
\end{equation}
To prove \eqref{L1.1}, we  represent $N(t)$ 
by using the contour integral in the complex plane
$\BC$. Let $C_R$ be a path deifined by 
\begin{align*}
C_R = \{\lambda \in \BC \mid \lambda =Re^{i\theta}, 
\quad -\frac{\pi}{2} \leq \theta \leq \frac{\pi}{2}\}.
\end{align*} 
Let $\gamma > \lambda_0$. 
By  Cauchy theorem in the theory of one complex variable, we have
\begin{equation}\label{L1.2}
0 = \int_{-R}^{R} e^{(\gamma + i\tau)t}
\CN(\gamma+i\tau )\bg\,d\tau
+ \int_{C_R+\gamma} e^{\lambda t}\CN(\lambda)\bg\,d\lambda.
\end{equation}
Using \eqref{L1prop}, we know that 
$$\|\CN(\lambda)g\|_{B^s_{q,1}(D)}
\leq C|\lambda|^{-1}\|g\|_{B^s_{q,1}(D)}.
$$
Thus, for $t < 0$ we have
\begin{align*}
&\Bigl\|\int_{C_R+\gamma} e^{\lambda t}\CN(\lambda)g
\,d\lambda\Bigr\|_{B^s_{q,1}(D)} \\
&\quad \leq Ce^{\gamma t}\int^{\pi/2}_{-\pi/2}e^{tR\cos\theta}
|\gamma+Re^{i\theta}|^{-1}\, Rd\theta 
\|g\|_{B^{s}_{q,1}(D)}\, d\theta \\
&\quad \leq Ce^{\gamma t}\int^{\pi/2}_0 e^{-|t|R\cos\theta}\,d\theta
\|g\|_{B^{s}_{q,1}(D)}.
\end{align*}
Since $|^{-|t|R\cos\theta}| \leq 1$, by Lebesgue's dominated convergence theorem, we have
$$\lim_{R\to\infty}\int^{\pi/2}_0 e^{-|t|R\cos\theta}\,d\theta
= \int^{\pi/2}_0\lim_{R\to\infty} e^{-|t|R\cos\theta}\,d\theta = 0.
$$
Therefore, letting $R\to\infty$ in \eqref{L1.2}, we have
$$0 = \int_{\BR}e^{(\gamma + i\tau)t}\CN(\gamma+i\tau )g
\,d\tau
= N(t)g,
$$
which proves \eqref{L1.1}. \par 
We next consider the case where $t>0$.  Let $\Gamma_\pm$ be the contours defined by
$$\Gamma_\pm = \{\lambda = re^{\pm \pi(\pi-\epsilon)} \mid r \in (0, \infty)\}.$$
We shall show that 
\begin{align}\label{L1.3}
\|\bar\nabla^2N(t)g\|_{B^s_{q,1}(D)} &\leq Ce^{\gamma t}
t^{-(1-\frac{\sigma}{2})}
\|g\|_{B^{s+\sigma}_{q,1}(D)}, \\
\label{L1.4}
\|\bar\nabla^2N(t)g\|_{B^s_{q,1}(D)} &\leq Ce^{\gamma t}
t^{-(1+\frac{\sigma}{2})}
\|g\|_{B^{s-\sigma}_{q,1}(D)}.
\end{align}
Noticing that $|e^{\lambda t}| = e^{t{\rm Re}\,\lambda} = e^{tr\cos(\pi-\epsilon)}
= e^{-tr\cos\epsilon}$ and $|\gamma + re^{\pm(\pi-\epsilon}|
= (\gamma^2-2\gamma\cos\epsilon r+ r^2)^{1/2} \geq (1-\cos\epsilon)^{1/2}r$
for $\lambda \in \Gamma_+\cup \Gamma_-+\gamma$
and using \eqref{L1prop}, we have 
\begin{align*}
\|\bar\nabla^2N(t)g\|_{B^s_{q,1}(D)} 
&\leq \Bigl\|\frac{1}{2\pi}\int_{\Gamma_+\cup\Gamma_-+\gamma}
e^{\lambda t}\,\bar\nabla^2 \CN(\lambda)g\,d\lambda\Bigr\|_{B^s_{q,1}}\\
&\leq Ce^{\gamma t}\int^\infty_0 e^{-tr\cos\epsilon}
((1-\cos\epsilon)^{1/2}r)^{-\frac{\sigma}{2}}\,dr
\|g\|_{B^{s+\sigma}_{q,1}(D)} \\
&=Ce^{\gamma t}t^{-1+\frac{\sigma}{2}}
\int^\infty_0 e^{-\tau\cos\epsilon}((1-\cos\epsilon)^{1/2}\tau)^{-\frac{\sigma}{2}}\,d\tau
\|g\|_{B^{s+\sigma}_{q,1}(D)}.
\end{align*}
Here, we have used the change of variable: $tr = \tau$. \par
We use integration by parts to represent 
$\bar\nabla^2N(t)g$
by
$$\bar\nabla^2N(t)g = \frac{-1}{2\pi it}
\int_{\Gamma_+\cup \Gamma_-+\gamma}e^{\lambda t}\pd_\lambda (\bar\nabla^2 \CN(\lambda)
g)\,d\lambda
$$
and applying \eqref{L1prop}
we have 
\begin{align*}
\|\bar\nabla^2N(t)g\|_{B^s_{q,1}(D)} 
&\leq \Bigl\|\frac{1}{2\pi i t}\int_{\Gamma_+\cup\Gamma_-+\gamma}e^{\lambda t}
\pd_\lambda(\bar\nabla^2 \CN(\lambda)g)\,d\lambda\Bigr\|_{B^s_{q,1}} \\
&\leq Ce^{\gamma t}
t^{-1}\int^\infty_0 e^{-tr\cos\epsilon}((1-\cos\epsilon)^{1/2}r)^{-(1-\frac{\sigma}{2})},dr
\|g\|_{B^{s-\sigma}_{q,1}(D)} \\
&=Ce^{\gamma t}
t^{-1-\frac{\sigma}{2}}
\int^\infty_0 e^{-\tau\cos\epsilon}
((1-\cos\epsilon)^{1/2}\tau)^{-(1-\frac{\sigma}{2})},d\tau
\|g\|_{B^{s-\sigma}_{q,1}(D)}.
\end{align*}
Therefore, we have \eqref{L1.3} and \eqref{L1.4}. \par
We shall prove \eqref{L1} by using \eqref{L1.3} and \eqref{L1.4}. 
We write
\begin{align*}
\int^\infty_0 e^{-\gamma t}\|\bar\nabla^2 N(t)g
\|_{B^{s}_{q,1}(D)}\,\d t 
&= \sum_{j \in \BZ} \int^{2^{(j+1)}}_{2^j}e^{-\gamma t}
\|\bar\nabla^2N(t)g\|_{B^{s}_{q,1}(D)}\, \d t\\
&\leq \sum_{j \in \BZ} \int^{2^{(j+1)}}_{2^j}
\sup_{t \in (2^j, 2^{j+1})}(e^{-\gamma t}
\|\bar\nabla^2 N(t)g\|_{B^{s}_{q,1}(D)})\,\d t
\\
& \quad =\sum_{j \in \BZ} 2^j
\sup_{t \in (2^j, 2^{j+1})}(e^{-\gamma t}\|\bar\nabla^2N(t)g
\|_{B^{s}_{q,1}(D)}).
\end{align*}
Setting $a_j = \sup_{t \in (2^j, 2^{j+1})}e^{-\gamma t}
\|\bar\nabla^2N(t)g\|_{B^{s}_{q,1}(D)}$, 
we have
$$
\int^\infty_0 e^{-\gamma t}\|\bar\nabla^2N(t)g\|_{B^{s}_{q,1}(D)}\,\d t
\leq 2((2^ja_j))_{\ell_1} = 2((a_j)_{j \in \BZ})_{\ell^1_1}. 
$$
Here and in the following, 
$\ell^s_q$ denotes the set of all sequences $(2^{js}a_j)_{j \in \BZ}$ such that 
\begin{align*}
\|((a_j)_{j \in \BZ})\|_{\ell^s_q} &= \Bigl\{\sum_{j \in \BZ} 
(2^{js}|a_j|)^q \Bigr\}^{1/q} < \infty \quad\text{for $1 \leq q < \infty$},  \\
\|((a_j)_{j \in \BZ})\|_{\ell^s_\infty} &= \sup_{j \in \BZ} 
2^{js}|a_j| < \infty \quad \text{for  $q = \infty$}.
\end{align*}
By \eqref{L1.3} and \eqref{L1.4},  we have
$$\sup_{j \in \BZ} 2^{j(1-\frac{\sigma}{2})}a_j \leq C\|g\|_{B^{s+\sigma}_{q,1}(D)}, 
\quad\sup_{j \in \BZ} 2^{j(1+\frac{\sigma}{2})}a_j \leq C\|g\|_{B^{s-\sigma}_{q,1}(D)}
$$
Namely, we have
$$\|(a_j)\|_{\ell_\infty^{1-\frac{\sigma}{2}}} \leq C\|g\|_{B^{s+\sigma}_{q,1}(D)},
\quad \|(a_j)\|_{\ell_\infty^{1+ \frac{\sigma}{2}}} \leq C\|g\|_{B^{s-\sigma}_{q,1}(D)}.
$$
According to \cite[5.6.1.Theorem]{BL},  
we know that $\ell^1_1 = (\ell^{1-\frac{\sigma}{2}}_\infty,
\ell^{1+\frac{\sigma}{2}}_\infty)_{1/2, 1}$, 
where
$(\cdot, \cdot)_{\theta, q}$ denotes the real interpolation functor, 
and therefore we have 
$$
\int^\infty_0 e^{-\gamma t}\|N(t)g\|_{B^{s+2}_{q,1}(D)}\,\d t
\leq C
\|g\|_{(B^{s+\sigma}_{q,1}(D), B^{s-\sigma}_{q,1}(D))_{1/2, 1}} 
=C
\|g\|_{B^{s}_{q,1}(D)}
$$
for any $g \in C^\infty_0(D)$. But, $C^\infty_0(D)$ is dense in
$B^s_{q,1}(D)$, so the estimate \eqref{L1} holds for any 
$g\in B^s_{q,1}(D)$.  This  completes the proof of 
Proposition \ref{prop:L1}.
\end{proof}
To treat the perturbation term, we introduce one more definition.
\begin{dfn}\label{dfn.2}
 Let $1 < q < \infty$, $\lambda_0>0$, and $0 < \epsilon < \pi/2$. 
Let $X$ and $Y$ be two Banach spaces and $\CM(\lambda) \in {\rm Hom}\,(\Sigma_{\epsilon,
\lambda_0}, \CL(X, Y))$.  We say that $\CM(\lambda)$ has a generalized
resolvent properties for $(X, Y)$ if there hold
\begin{equation*}\label{L1. 5}
\|\pd_\lambda^\ell \CM(\lambda)f\|_Y \leq C|\lambda|^{-\ell-1}
\|f\|_X
\quad\text{for $f \in X$ and $\ell=0,1$}.
\end{equation*}
\end{dfn}
\begin{remark} If $\CM(\lambda)$ is a usual resolvent operator 
$\CM(\lambda) = (\lambda \bI - A)^{-1}$ of closed linear opreator $A$ defined in
dense subspace $D(A)$ of $X$ for $\lambda \in \Sigma_{\epsilon, \lambda_0}$,
then $\pd_\lambda (\lambda\bI-A)^{-1} = -(\lambda\bI-A)^{-2}$, and so
$(\lambda \bI-A)^{-1}$ has generalized resolvent properties for $(X, X)$.
\end{remark}
Let $M(t)$ be the Laplace invese transform of $\CM(\lambda)$ defined by
$$M(t)f = \CL^{-1}[\CM(\lambda)f] = \int_{\BR}e^{(\gamma  + i\tau)t}
\CM(\gamma + i\tau)f\,d\tau.
$$
Then, we have the following proposition about the $L_1$ integrability of 
$M(t)$.
\begin{prop} \label{prop:L2}
Let $1 < q < \infty$, $\lambda_0>0$, and $0 < \epsilon < \pi/2$. 
Let $X$ and $Y$  be two Banach spaces and $\CM(\lambda) \in {\rm Hom}\,(\Sigma_{\epsilon,
\lambda_0}, \CL(X, Y))$. If $\CM(\lambda)$ has  generalized resolvent properties
for $(X, Y)$, 
then, for $f \in X$ and $\gamma > \lambda_0$, it holds that 
$$\int_{\BR}e^{-\gamma t}\|M(t)f\|_Y \, dt \leq C\|f\|_X.$$
\end{prop}
\begin{proof} For $\lambda \in \Sigma_{\epsilon, \lambda_0}$, we have
\begin{align*}
\|\CM(\lambda)f\|_Y \leq C|\lambda|^{-1}\|f\|_X
&\leq C\lambda_0^{-(1-\frac{\sigma}{2})}|\lambda|^{-\frac{\sigma}{2}}\|f\|_X,
\\
\|\pd_\lambda \CM(\lambda)f\|_Y \leq C|\lambda|^{-2}\|f\|_X
&\leq C\lambda_0^{-(1+\frac{\sigma}{2})}|\lambda|^{-(1-\frac{\sigma}{2})}\|f\|_X
\end{align*}
for any $\lambda \in \Sigma_{\epsilon, \lambda_0}$.  Thus, employing the 
same argument as in the proof of Proposition \ref{prop:L1}, we can prove
Proposition \ref{prop:L2}.  This completes the proof. 
\end{proof}
In view of Propositions \ref{prop:L1} and \ref{prop:L2}, to prove the $L_1$ integrability
of solutions to the evolution equations, it is a key to prove the existence of 
solutions operators having  $(s, \sigma, q, 1)$ properties and generalized resolvent 
properties to the corresponding  resolvent problems. Thus, the main parts of this paper are devoted
to driving solution operators having such properties.

Now, we shall give a theorem used to prove that an operator has 
$(s, \sigma, q,r)$ properties. For this purpose, 
we consider two operator valued holomorphic 
functions $Q_i(\lambda)$ ($i=1,2$) defined on $\Sigma_\epsilon$ acting on
$f \in C^\infty_0(\HS)$. We denote the dual operator of $Q_i(\lambda)$ by $Q_i(\lambda)^*$
which satisfies the equality:
$$|(Q_i(\lambda)f, \varphi)_D| = |(f, Q_i(\lambda)^*\varphi)_D| \quad(i=1,2)
$$
for any $f$ and $\varphi \in C^\infty_0(D)$.  Here, $(f, g) = \int_D f(x)g(x)\,dx$.  
And, we assume that $C^\infty_0(D)$ is dense in $B^s_{q,r}(D)$.
Let $Q_i(\lambda)$ satisfy the following assumptions.
\begin{assump}\label{assump.1} Let $1 < q < \infty$ and $q' = q/(q-1)$.  \\
 For any $f \in C^\infty_0(D)$ and 
$\lambda\in \Lambda_{\epsilon, \lambda_0}$, the following estimates hold:
\allowdisplaybreaks
\begin{align}
\|Q_1(\lambda)f\|_{W^i_q(D)}&\leq C\|f\|_{W^i_q(D)}, \label{as.1}\\
\|Q_1(\lambda)f\|_{L_q(D)}&\leq C|\lambda|^{-1/2}\|f\|_{W^1_q(D)}, \label{as.3}\\
\|Q_1(\lambda)^*f\|_{W^i_{q'}(D)}&\leq C\|f\|_{W^i_{q'}(D)}, \label{as.5} \\
\|Q_1(\lambda)^*f\|_{L_{q'}(D)}&\leq C|\lambda|^{-1/2}\|f\|_{W^1_{q'}(D)}, \label{as.6} \\
\|Q_2(\lambda)f\|_{W^i_q(D)}&\leq C|\lambda|^{-1}\|f\|_{W^i_q(D)}, \label{as.7} \\
\|Q_2(\lambda)f\|_{W^1_q(D)}&\leq C|\lambda|^{-1/2}\|f\|_{L_q(D)}, \label{as.9} \\
\|Q_2(\lambda)^*f\|_{W^i_{q'}(D)}&\leq C|\lambda|^{-1}
\|f\|_{W^i_{q'}(D)}, \label{as.11}\\
\|Q_2(\lambda)^*f\|_{W^1_{q'}(D)}&\leq C|\lambda|^{-1/2}\|f\|_{L_{q'}(D)}. \label{as.12} 
\end{align}
for $i=0,1$, where we have written $W^0_r(D) = L_r(D)$ for simplicity. 
\end{assump}
The following theorem will be used to prove that solution operators of 
Lam\'e equations have $(s, \sigma, q, r)$ properties, which has
 been proved in  \cite{S23}, \cite{SW1}, and \cite{Kuo23}. 
\begin{thm}\label{thm:5.2} Let $1 < q < \infty$, $1 \leq r \leq \infty$, 
$-1+1/q < s < 1/q$. Let $\sigma > 0$ be a number such that 
$-1+1/q < s-\sigma < s+\sigma < 1/q$. Let $Q_i(\lambda)$ $(i=1,2)$ be operator valued holomorphic 
functions defined on $\Lambda_{\epsilon, \lambda_0}$ acting on $C^\infty_0(D)$ functions.
 Then, for any $\lambda \in \Lambda_{\epsilon, \lambda_0}$ and $f \in C^\infty_0(D)$, 
the following assertions hold.\\
\thetag1 If $Q_1(\lambda)$ satisfies \eqref{as.1} and \eqref{as.5},
then there holds
$$\|Q_1(\lambda)f \|_{B^s_{q,r}(D)}\leq C\|f\|_{B^s_{q,r}(D)}.$$
If $Q_1(\lambda)$ satisfies \eqref{as.3} and 
\eqref{as.6} in addition, then there holds
$$\|Q_1(\lambda)f\|_{B^s_{q,r}(D)}  
\leq C|\lambda|^{-\frac{\sigma}{2}}\|f\|_{B^{s+\sigma}_{q.r}(D)}.
$$
\thetag2 If $Q_2(\lambda)$ satisfies \eqref{as.7} and \eqref{as.11},  
then there holds
$$\|Q_2(\lambda)f\|_{B^s_{q,r}(D)} \leq C|\lambda|^{-1}\|f\|_{B^s_{q,r}(D)}.$$
If $Q_2(\lambda)$ satisfies \eqref{as.9} and \eqref{as.12} in addition, then there holds
$$\|Q_2(\lambda)f\|_{B^s_{q,r}(D)}  \leq C|\lambda|^{-(1-\frac{\sigma}{2})}
\|f\|_{B^{s-\sigma}_{q,r}(D)}.$$
\end{thm}


\section{On the spectral analysis of  Lam\'e equations}
We shall prove that a solution operator of
Lam\'e equations \eqref{Eq:Lame} has $(s, \sigma, q, r)$ properties.  Our proof is 
divided into the whole space case, the half-space case, the bent half space case,
and the general domain case, which is the standard procedure.   
We start with 
\subsection{The whole space case.}\label{sec:4.1}

In this subsection, we consider the Lam\'e equations: 
\begin{equation}\label{L.1}
\lambda \bu - \alpha\Delta \bu - \beta \nabla\dv\bu = \bg \quad\text{in $\BR^N$}
\end{equation}
for $\lambda \in \Sigma_\epsilon$ with $\epsilon \in (0, \pi/2)$. 
We shall prove 
\begin{thm}\label{thm:L.1} Let $1 < q < \infty$, $1 \leq r \leq \infty$, $-1+1/q < s < 1/q$,
$0 < \epsilon < \pi/2$, and $\lambda_0>0$.  Let $\sigma$ be a small positive number
such that $-1+1/q < s-\sigma < s+\sigma < 1/q$.  Let $\nu \in \{s-\sigma, s, s+\sigma\}$.
Then, there exists an operator $\CS(\lambda) \in {\rm Hol}\,(
\Sigma_\epsilon, \CL(B^\nu_{q,r}(\BR^N)^N, B^{\nu+2}_{q,r}(\BR^N)^N))$ having
$(s, \sigma, q, r)$ properties in $\BR^N$ such that 
for any $\bg \in B^s_{q,r}(\BR^N)^N$, $\bu = \CS(\lambda)\bg$ is a unique solution of
equations \eqref{L.1}.
\end{thm}
\begin{proof}
Applying the divergence to equations \eqref{L.1} gives 
$$\lambda \dv \bu - (\alpha + \beta)\Delta \dv\bu = \dv\bg \quad\text{in $\BR^N$}.
$$
Using the Fourier transform $\CF$ and its inverse transform $\CF^{-1}$, we have
$$\dv\bu = \CF^{-1}\Bigl[\frac{i\xi\cdot\CF[\bg](\xi)}{\lambda+(\alpha+\beta)|\xi|^2}\Bigr].$$
Inserting this formula into \eqref{L.1}, we have
\begin{equation}\label{sol.3.1}\begin{aligned}
\CS(\lambda)\bg&= \CF^{-1}[\frac{\CF[\bg](\xi) + \beta \CF[\dv\bu](\xi)}{\lambda + \alpha|\xi|^2}\Bigr]\\
& = \CF^{-1}\Bigl[\frac{\CF[\bg](\xi)}{\lambda + \alpha|\xi|^2}\Bigr]
+ \beta\CF^{-1}\Bigl[\frac{i\xi i\xi\cdot\CF[\bg](\xi)}{(\lambda + \alpha|\xi|^2)
(\lambda + (\alpha+\beta)|\xi|^2)}\Bigr].
\end{aligned}\end{equation}
As we know well, there exist positive constants $c_1$ and $c_2$ depending on $\alpha$,
$\beta$ and $\epsilon$ such that for any $\lambda \in \Sigma_\epsilon$ there hold:
\begin{equation*}\label{L.2}\begin{aligned}
c_1(|\lambda|^{1/2}+|\xi|)^2  \leq {\rm Re}\, (\lambda 
+\alpha|\xi|^2) &\leq |\lambda + \alpha|\xi|^2| 
\leq c_2(|\lambda|^{1/2}+|\xi|)^2, \\
c_1(|\lambda|^{1/2}+|\xi|)^2 \leq
{\rm Re}\, (\lambda 
+(\alpha+\beta)|\xi|^2) &\leq |\lambda + (\alpha+\beta)|\xi|^2| 
\leq c_2(|\lambda|^{1/2}+|\xi|)^2.
\end{aligned}\end{equation*}
Thus, applying the Fourier multiplier theorem of Mikhlin-H\"ormander type, we have 
\begin{equation}\label{whole.est1}
\|(\lambda, \lambda^{1/2}\bar\nabla, \bar\nabla^2)\CS(\lambda)\bg\|_{B^\nu_{q,r}(\BR^B)}  
\leq C(1+\lambda_0^{-1/2} + \lambda_0^{-1})\|\bg\|_{B^\nu_{q,r}(\BR^N)}.
\end{equation}
\par
Let $0 < \sigma < 1$.  For $\bg \in B^{s+\sigma}_{q,r}(\BR^N)$, we write
\begin{align*}
\lambda^{1/2}\lambda^{\sigma/2}\bar\nabla \CS(\lambda)\bg
&= \CF^{-1}\Bigl[\frac{\lambda^{\frac12+\frac{\sigma}{2}}(1, i\xi)(1+|\xi|^2)^{\sigma/2}\CF[\bg](\xi)}
{(\lambda + \alpha|\xi|^2)(1+|\xi|^2)^{\sigma/2}}\Bigr] \\
&+ \beta\CF^{-1}\Bigl[\frac{\lambda^{\frac12+\frac{\sigma}{2}}(i\xi)i\xi i\xi\cdot
((1+|\xi|^2)^{\sigma/2}\CF[\bg](\xi))}{(\lambda + \alpha|\xi|^2)
(\lambda + (\alpha+\beta)|\xi|^2)(1+|\xi|^2)^{\sigma/2}}\Bigr], 
\end{align*}
Applying the Fourier multiplier theorem of Mikhilin-H\"ormander type, we have
$$\|\lambda^{1/2}\lambda^{\sigma/2}\nabla \CS(\lambda)\bg\|_{B^s_{q,r}(\BR^N)}
\leq C(1+\lambda_0^{-1/2})\|\bg\|_{B^{s+\sigma}_{q,r}(\BR^N)}.
$$
Analogously, we have
$$
\|\lambda^{\frac{\sigma}{2}}\bar\nabla^2\bu\|_{B^s_{q,r}(\BR^N)} \leq 
C(1 + \lambda_0^{-1/2} + \lambda_0^{-1})\|\bg\|_{B^{s+\sigma}_{q,r}(\BR^N)}.
$$
Moreover, since $\bg \in B^{s+\sigma}_{q,r}(\BR^N) \subset B^{s-\sigma}_{q,r}(\BR^N)$,
changing $(1+|\xi|^2)^{\sigma/2}$ by $(1+|\xi|^2)^{-\sigma/2}$, 
we have
\begin{equation}\label{whole.est2}\|(1, \lambda^{-1/2}\bar\nabla)\CS(\lambda)\bg\|_{B^s_{q,r}(\BR^N)}
\leq C(1+\lambda_0^{-1/2})|\lambda|^{-(1-\frac{\sigma}{2})}
\|\bg\|_{B^{s-\sigma}_{q,r}(\BR^N)}.
\end{equation}
Concerning $\pd_\lambda \CS(\lambda)\bg$, differentiating equations \eqref{L.1} 
and using the uniqueness of solutions, we see that 
$\pd_\lambda \CS(\lambda)\bg = -\CS(\lambda)\CS(\lambda)\bg$. Using \eqref{whole.est1}
and \eqref{whole.est2}, we immediately have
\begin{align*}
\|(\lambda, \lambda^{1/2}\bar\nabla, \bar\nabla^2)\pd_\lambda \CS(\lambda)\bg
\|_{B^s_{q,r}(\BR^N)} & \leq C|\lambda|^{-1}\|\bg\|_{B^s_{q,r}(\BR^N)}, \\
\|(\lambda, \lambda^{1/2}\bar\nabla, \bar\nabla^2)\pd_\lambda \CS(\lambda)\bg
\|_{B^s_{q,r}(\BR^N)} & \leq C|\lambda|^{-(1-\frac{\sigma}{2})
}\|\bg\|_{B^{s-\sigma}_{q,r}(\BR^N)}.
\end{align*}
This completes the proof of Theorem \ref{thm:L.1}.
\end{proof}

\subsection{The half-space case.}\label{sec.4.2}
In this section, we consider the Lam\'e equations in the half space, which read as 
\begin{equation}\label{hL.1}
\lambda \bu - \alpha \Delta\bu - \beta\nabla\dv\bu = \bg \quad\text{in $\HS$},
\quad \bu=0 \quad\text{on $\pd\HS$}.
\end{equation}
Notice that $\pd\HS=\{x=(x_1, \ldots, x_N) \in \BR^N \mid x_N=0\}$. 
We shall prove 
\begin{thm}\label{thm:hL.1} Let $1 < q < \infty$, $1 \leq r \leq \infty-$, 
 $-1+1/q < s < 1/q$, $\epsilon \in (0, \pi/2)$ and $\lambda_0>0$.
 Let $\sigma$ be a small positive number such that 
$-1+1/q < s-\sigma < s < s+\sigma < 1/q$.
Let $\nu \in \{s-\sigma, s, s+\sigma\}$. 
Then, there exists an  operator $\CS_h(\lambda) \in {\rm Hol}\,(
\Sigma_\epsilon, \CL(B^\nu_{q,r}(\HS)^N, B^{\nu+2}_{q,r}(\HS)^N))$ 
having $(s, \sigma, q, r)$ properties in $\HS$ such that 
for any $\lambda \in \Sigma_{\epsilon, \lambda_0}$ 
and $\bg \in B^\nu_{q,r}(\HS)^N$, $\bu = \CS_h(\lambda)\bg$  is a unique solution of
\eqref{hL.1}.
\end{thm}
In what follows, we shall prove Theorem \ref{thm:hL.1}. 
Since we know  solution operators
in $\BR^N$, we  consider the compensation equations:
\begin{equation}\label{july.1}
\lambda \bu-\alpha\Delta \bu - \beta\nabla\dv\bu = 0\quad\text{in $\HS$}, 
\quad u_j|_{x_N=0} = h_j|_{x_N=0}, \quad u_N|_{x_N=0}=0
\end{equation}
for $j=1, \ldots, N-1$.  Let $\bh' = (h_1, \ldots, h_{N-1})$.
To obtain a solution formula of \eqref{july.1}, we apply the 
partial Fourier transform $\CF'$ with respect to the tangential variables 
$x'=(x_1, \ldots, x_{N-1})$ and its inverse transform $\CF^{-1}_{\xi'}$ with respect
to the dual variables $\xi' = ( \xi_1, \ldots, \xi_{N-1}) \in \BR^{N-1}$, and then we have
the system of ordinary differential equations:
\begin{alignat*}2
(\lambda +\alpha |\xi'|^2 -\alpha D_N^2)\CF'[u_j] -\beta i\xi_j(i\xi'\cdot\CF'[\bu'] + D_N\CF'[u_N]) = 0&&
\quad&(x_N>0), \\
(\lambda +\alpha |\xi'|^2 -\alpha D_N^2)\CF'[u_N] -\beta D_N(i\xi'\cdot\CF'[\bu'] + D_N\CF'[u_N]) = 0&&
\quad&(x_N>0), \\
\CF'[u_j]|_{x_N=0} = \CF'[h_j](\xi', 0), \quad \CF'[u_N]|_{x_N=0} = 0&
\end{alignat*}
for $j=1, \ldots, N-1$. Here, we have written 
$i\xi'\cdot \CF'[\bu'] = \sum_{j=1}^{N-1}i\xi_j\CF'[u_j]$. 
 Multiplying the first equation with $i\xi_j$, differentiating the second equation
and summing up the resultant equations, we have
$$ (\lambda + (\alpha+\beta)|\xi'|^2 - (\alpha+\beta)D_N^2)(i\xi'\cdot\CF'[\bu'] + D_N\CF'[u_N]) = 0.$$
Applying this formula to the equations above implies that 
$$(\lambda + \alpha|\xi'|^2- D_N^2)(\lambda + (\alpha+\beta)|\xi'|^2- D_N^2)\CF[u_j] = 0 
\quad(j=1, \ldots, N).
$$
Thus, $A = \sqrt{(\alpha+\beta)^{-1}\lambda + |\xi'|^2}$ and $B = \sqrt{\alpha^{-1}\lambda + 
|\xi'|^2}$ are two characteristic roots, where we choose the branches such that 
${\rm Re} A > 0$ and ${\rm Re}\, B > 0$. Set 
$$\CF'[u_j] = m_je^{-Bx_N} + n_j(e^{-Ax_N} - e^{-Bx_N}).$$
Substituting these formulas into the equations, we have
\begin{align*}
&\alpha(B^2-A^2)n_j - \beta i\xi_j(i\xi'\cdot \bn' - An_N) = 0, 
\quad \beta i\xi_j(i\xi'\cdot\bm'- i\xi'\cdot\bn'- m_NB+n_NB) = 0,  \\
&\alpha(B^2-A^2)n_N + \beta A(i\xi'\cdot \bn' - An_N) = 0, \quad
\beta B(i\xi'\cdot\bm'- i\xi'\cdot\bn'- m_NB+n_NB) = 0, \\
&m_i=\hat h_j,  \quad m_N=0,
\end{align*}
for $j=1, \ldots, N-1$, 
where we have set  $i\xi'\cdot\bm' = \sum_{j=1}^{N-1}i\xi_j m_j$, 
$i\xi'\cdot \bn' = \sum_{j=1}^{N-1}i\xi_jn_j$ and $\hat h_j = \CF'[h_j]|_{x_N=0}$.
Thus, we have
\begin{align*}
n_j = \frac{\beta i\xi_j}{\alpha(B^2-A^2)}(i\xi'\cdot\bn'-An_N), \quad 
n_N = -\frac{\beta A}{\alpha(B^2-A^2)}(i\xi'\cdot\bn'-An_N),  \quad
i\xi'\cdot\bn'-n_NB = i\xi'\cdot \hat \bh',
\end{align*}
where we have set $i\xi'\cdot \hat\bh' = \sum_{j=1}^{N-1}i\xi_j \hat h_j$. Moreover, 
we have
$\alpha(B^2-A^2)i\xi'\cdot\bn+ \beta|\xi'|^2(i\xi'\cdot \bn'-An_N) = 0$, which implies that 
$$i\xi'\cdot\bn' = \frac{\beta A|\xi'|^2}{\alpha(B^2-A^2)+\beta|\xi'|^2} n_N.$$
Thus, 
$$\Bigl(\frac{\beta A|\xi'|^2}{\alpha(B^2-A^2)+\beta|\xi'|^2}-B\Bigr)n_N = i\xi'\cdot\hat \bh',
$$
which implies that 
$$i\xi'\cdot\hat\bh' = \frac{-(\beta|\xi'|^2+\alpha B(A+B))(B-A)}
{(B-A)(\alpha(A+B)B+\beta|\xi'|^2)}n_N.$$
From this, we  have 
$$n_N = -\frac{\alpha(B^2-A^2) + \beta|\xi'|^2}{(B-A)
(\alpha B(A+B) + \beta|\xi'|^2)}i\xi'\cdot\hat \bh'.
$$
Also, 
$$i\xi'\cdot\bn' = \frac{\beta A|\xi|^2}{\alpha(B^2-A^2)+\beta|\xi'|^2}n_N
= -\frac{\beta A|\xi'|^2}{(B-A)(\alpha B(A+B) + \beta|\xi'|^2)}i\xi'\cdot\hat \bh'.$$
Thus, 
$$i\xi'\cdot\bn' - An_N = \frac{-\beta A|\xi'|^2 + A(\alpha(B^2-A^2)+\beta|\xi'|^2)}
{(B-A)(\alpha B(A+B) + \beta|\xi'|^2)}
= \frac{\alpha A(B^2-A^2)}{(B-A)(\alpha B(A+B) + \beta|\xi'|^2)}i\xi'\cdot\hat\bh'.$$
Using this formula, we have
$$
n_j = \frac{\beta i\xi_j A}{(B-A)(\alpha B(A+B) + \beta|\xi'|^2)}i\xi'\cdot\hat \bh', \quad
n_N = -\frac{\beta  A^2}{(B-A)(\alpha B(A+B) + \beta|\xi'|^2)}i\xi'\cdot\hat \bh'.
$$
To obtain
$$\alpha B(A+B) + \beta|\xi'|^2 = A((\alpha + \beta)A + \alpha B),$$
we use the formulas:
$$(\alpha+\beta)A^2= \lambda + (\alpha+\beta)|\xi'|^2, 
\quad \alpha B^2 = \lambda + \alpha |\xi'|^2.$$
Finally, we arrive at 
$$
n_j = \frac{\beta i\xi_j }{(B-A)((\alpha+\beta)A+\alpha B)}i\xi'\cdot\hat \bh', \quad
n_N = -\frac{\beta  A}{(B-A)((\alpha+\beta)A + \alpha B)}i\xi'\cdot\hat \bh'.
$$
Set 
$$M(x_N) = \frac{e^{-Bx_N}-e^{-Ax_N}}{B-A}, \quad L(\lambda, \xi') 
= (\alpha+\beta)A + \alpha B.$$
The $L(\lambda, \xi')$ is called the Lopatinski determinant of the system of equations \eqref{comp.1}.
We may have 
\begin{equation*}\label{hsol.1}\begin{aligned}
\CF'[u_j](\xi', x_N) &= \hat 
h_je^{-Bx_N} + M(x_N)\frac{\beta i\xi_j }{L(\lambda, \xi')}i\xi'\cdot\hat \bh',\\
\CF'[u_N](\xi', x_N) & = -M(x_N)\frac{\beta A}{L(\lambda, \xi')}i\xi'\cdot\hat \bh.
\end{aligned}\end{equation*}
Noting that $D_NM(x_N) = -e^{-Bx_N} - AM(x_N)$ and using Volevich's trick, we write $\CF'[u_j]$ and $\CF'[u_N]$ as follows:
\begin{align*}
\CF'[u_j](\xi', x_N)& = \int^\infty_0Be^{-B(x_N+y_N)}\CF'[h_j](\xi', y_N)\,dy_N
-\int^\infty_0 e^{-B(x_N+y_N)}\CF'[D_Nh_j](\xi', y_N)\,dy_N \\
&+ \int^\infty_0(e^{-B(x_N+y_N)} +AM(x_N+y_N))\frac{\beta i\xi_j}{L(\lambda, \xi')}
i\xi'\cdot\CF'[\bh'](\xi', y_N)\,dy_N \\
&-\int^\infty_0 M(x_N+y_N)\frac{\beta i\xi_j}{L(\lambda, \xi')}i\xi'\cdot\CF'[D_N\bh'](\xi', y_N)\,dy_N, \\
\CF'[u_N](\xi', x_N) & = 
- \int^\infty_0(e^{-B(x_N+y_N)} +AM(x_N+y_N))\frac{\beta A}{L(\lambda, \xi')}
i\xi'\cdot\CF'[\bh'](\xi', y_N)\,dy_N \\
&+\int^\infty_0 M(x_N+y_N)\frac{\beta A}{L(\lambda, \xi')}i\xi'\cdot\CF'[D_N\bh'](\xi', y_N)\,dy_N.
\end{align*}
Moreover, using the formula: $1 = (\alpha^{-1}\lambda + |\xi'|^2)B^{-2}$ and writing
$\Delta' h_j=\sum_{k=1}^{N-1}D_k^2 h_j$ and $\dv'\bh' = \sum_{k=1}^{N-1}D_kh_k$, 
 we rewrite the formulas above as follows:
\allowdisplaybreaks 
\begin{align*}
\CF'[u_j](\xi', x_N) & = \int^\infty_0Be^{-B(x_N+y_N)}\frac{1}{B^2}\CF'[(\lambda-\Delta')h_j](\xi', y_N)\,dy_N
\\ 
&-\int^\infty_0Be^{-B(x_N+y_N)}\frac{\alpha^{-1}\lambda^{1/2}}{B^3}\CF'[\lambda^{1/2}D_Nh_j](\xi', y_N)\,dy_N 
\\
& + \sum_{\ell=1}^{N-1}\int^\infty_0Be^{-B(x_N+y_N)}\frac{i\xi_\ell}{B^3}
\CF'[D_\ell D_Nh_j](\xi', y_N)\,dy_N, \\
&+ \int^\infty_0Be^{-B(x_N+y_N)}\frac{\alpha^{-1}\lambda^{1/2}}{B^3}
\frac{\beta i\xi_j}{L(\lambda, \xi')}\CF'[\lambda^{1/2}\dv'\bh'](\xi', y_N)\,dy_N\\
&-\sum_{\ell=1}^{N-1}\int^\infty_0Be^{-B(x_N+y_N)}\frac{i\xi_\ell}{B^3}\frac{\beta i\xi_j}{L(\lambda, \xi')}
\CF'[D_\ell\dv'\bh'](\xi', y_N)\,dy_N\\
& + \int^\infty_0 B^2M(x_N+y_N)\frac{\beta A}{B^2L(\lambda, \xi')}\CF'[D_j\dv\bh'](\xi', y_N)\,dy_N\\
& + \int^\infty_0 B^2M(x_N+y_N)\frac{\beta A}{B^2L(\lambda, \xi')}\CF'[D_j\dv'\bh'](\xi', y_N)\,dy_N
\\
& -\int^\infty_0B^2M(x_N+y_N)\frac{\beta i\xi_j}{B^2L(\lambda, \xi')}\CF'[D_N\dv'\bh'](\xi', y_N)\,dy_N;
\\
\CF'[u_N](\xi', y_N) & = 
- \int^\infty_0Be^{-B(x_N+y_N)}\frac{\alpha^{-1}\lambda^{1/2}}{B^3}
\frac{\beta A}{L(\lambda, \xi')}\CF'[\lambda^{1/2}\dv'\bh'](\xi', y_N)\,dy_N\\
&+\sum_{\ell=1}^{N-1}\int^\infty_0Be^{-B(x_N+y_N)}\frac{i\xi_\ell}{B^3}\frac{\beta A}{L(\lambda, \xi')}
\CF'[D_\ell\dv'\bh'](\xi', y_N)\,dy_N\\
& +\int^\infty_0 B^2M(x_N+y_N)\frac{\beta A^2\alpha^{-1/2}\lambda^{1/2}}{B^4L(\lambda, \xi')}\CF'[\lambda^{1/2}\dv\bh'](\xi', y_N)\,dy_N\\
& - \sum_{\ell=1}^{N-1}\int^\infty_0 B^2M(x_N+y_N)\frac{\beta A^2i\xi_\ell}{B^4L(\lambda, \xi')}\CF'[D_j\dv'\bh'](\xi', y_N)\,dy_N
\\
& +\int^\infty_0B^2M(x_N+y_N)\frac{\beta A}{B^2L(\lambda, \xi')}\CF'[D_N\dv'\bh'](\xi', y_N)\,dy_N.
\end{align*}
There exist two positive constants $c_1 < c_2$ such that 
\begin{equation*}\label{lop.1}
c_1(|\lambda|^{1/2}+|\xi'|) \leq {\rm Re}((\alpha + \beta)A + \alpha B)
\leq |L(\lambda, \xi')| \leq c_2(|\lambda|^{1/2}+|\xi'|).
\end{equation*}
In particular, $L(\lambda, \xi')^{-1}$ is the order $-1$ symbol. 
Let $\CD_\lambda \bh' = (\lambda\bh', \lambda^{1/2}\nabla\bh', \nabla^2\bh')$.  Then, there exist
two matrices of order $-2$ symbols  $\CM_1(\lambda, \xi')$ and $\CM_2(\lambda, \xi')$
such that $\bu(x)$ can be written as 
\begin{equation*}\label{hsol.2}\begin{aligned}
\bu(x) & =\int^\infty_0\CF^{-1}_{\xi'}[Be^{-B(x_N+y_N)} \CM_1(\lambda, \xi') \CF'[\CD_\lambda\bh'](\xi', y_N)]
\,dy_N \\
& + \int^\infty_0\CF^{-1}_{\xi'}[B^2M(x_N+y_N) \CM_2(\lambda, \xi') \CF'[\CD_\lambda\bh'](\xi', y_N)]
\,dy_N.
\end{aligned}\end{equation*}

Let $H_1 = (H_{11}, \ldots, H_{1N-1})$, $H_2 = (H_{2jk} \mid j =1\ldots, N, k=1, \ldots, N-1)$ and 
$H_3 = (H_{3jk\ell} \mid j, k=1, \ldots, N, \ell=1, \ldots, N-1)$ and $H_{1j}$, $H_{2jk}$ and $H_{3jk\ell}$
are  corresponding variables to $\lambda h_j$, $\lambda^{1/2}D_jh_k$ and $D_jD_kh_\ell$, respectively.
Set $H =(H_1, H_2, H_3)$, which is an $(N-1)(1+N+N^2)$ vector. 
 Define an operator $\CT_h(\lambda)$ by
\begin{align*}
\CT_h(\lambda)H
& = \int^\infty_0\CF^{-1}_{\xi'}[Be^{-B(x_N+y_N)} \CM_1(\lambda, \xi') \CF'[H](\xi', y_N)]
\,dy_N \\
& + \int^\infty_0\CF^{-1}_{\xi'}[B^2M(x_N+y_N) \CM_2(\lambda, \xi') \CF'[H](\xi', y_N)]
\,dy_N.
\end{align*}
Notice that 
\begin{equation*}\label{sol.3.1}
\CT_h(\lambda)\CD_\lambda\bh'= \bu
\end{equation*}
We shall show the following theorem concerning $\CT_h(\lambda)$.
\begin{thm}\label{thm:3.1} Let $1 < q < \infty$, $1 \leq r \leq  \infty-$ and $-1+1/q <s < 1/q$. 
Set $m(N) = (N-1)(1 + N + N^2)$.  Then, for any
$\lambda \in \Lambda_{\epsilon}$ and $H \in B^s_{q,r}(\HS)^{m(N)}$, 
there holds 
\begin{align}\label{5.1.1}
 \|(\lambda, \lambda^{1/2}\bar\nabla, \bar\nabla^2) \CT_h(\lambda)H\|_{B^s_{q,r}(\HS)}
 & \leq C\|H\|_{B^s_{q,r}(\HS)}.
\end{align}

Moreover, let $\sigma>0$ be a number such that 
$-1+1/q < s-\sigma < s+\sigma < 1/q$.  Then, for any
$\lambda \in \Lambda_{\epsilon}$ and $H \in C^\infty_0(\HS)^{m(N)}$, 
there hold 
\begin{align}\label{5.2.1}
\|(\lambda^{1/2}\bar\nabla, \bar\nabla^2) \CT_h(\lambda)H\|_{B^s_{q,r}(\HS)} 
& \leq C|\lambda|^{-\frac{\sigma}{2}}\|H\|_{B^{s+\sigma}_{q,r}(\HS)},
\\
\|(1, \lambda^{-1/2}\bar\nabla)\CT_h(\lambda)H\|_{B^s_{q,r}(\HS)}
& \leq 
C|\lambda|^{-(1-\frac{\sigma}{2})}\|H\|_{B^{s-\sigma}_{q,r}(\HS)}
\label{5.2.3} 
\end{align}
\end{thm}
\begin{proof}
In what follows, we shall estimate $\CT_h(\lambda)$ using 
Theorem \ref{thm:5.2} in Sect. \ref{sec.3}. 
Notice that 
$\|f\|_{H^1_q(\HS)} = \|\bar\nabla f\|_{L_q(\HS)}$ and $\|f\|_{H^2_q(\HS)} = 
\|\bar\nabla^2f\|_{L_q(\HS)}$.  In what follows, 
we may assume that $H \in C^\infty_0(\HS)^{m(N)}$, because 
$C^\infty_0(\HS)$ is dense in $B^s_{q,r}(\HS)$ for $1 < q < \infty$, $1 \leq r \leq \infty-$ and 
$-1+1/q < s < 1/q$ (cf. Proposition 2.24, Lemma 2.32, and
Corollaries 2.26 and 2.34 in \cite{Gaudin}). 
Using the formulas:
$$\pd_N^\ell M(x_N ) = (-1)^\ell(A^\ell M(x_N) + \frac{A^\ell-B^\ell}{A-B}e^{-Bx_N})\quad(\ell \geq 1),\\
$$
and setting 
\begin{align*}
\CM_1^{(0)}(\lambda) &= \CM_1(\lambda), \quad \CM_1^{(\ell)}(\lambda)
= (-B)^\ell \CM_1(\lambda) + (-1)^\ell \frac{A^\ell-B^{\ell}}{A-B} \CM_2(\lambda) \quad(\ell \geq 1), 
\\
\CM_2^{(0)}(\lambda) &= \CM_2(\lambda), \quad
\CM_2^{(\ell)}(\lambda) = (-1)^\ell A^\ell \CM_2(\lambda) \quad(\ell \geq 2).
\end{align*}
for the notational simplicity, 
we write
\begin{align}
\pd_N^\ell \CT_h(\lambda)H = 
\int^\infty_0\CF^{-1}_{\xi'}&\Bigl[(\CM_{1}^{(\ell)}(\lambda)\CF'[H](\xi', y_N)Be^{-B(x_N+y_N)} 
\nonumber \\
&+ 
\CM_{2}^{(\ell)}(\lambda)\CF'[H](\xi', y_N)B^2M(x_N+y_N)\Bigr](x')\,dy_N.
\label{dif.1}
\end{align}
Using these symbols, we can write
\begin{align*}
\lambda^{k}\pd_{x'}^{\kappa'}\pd_N^\ell \CT_h(\lambda)H = 
\int^\infty_0\CF^{-1}_{\xi'}\Bigl[(&\lambda^{k}(i\xi')^{\kappa'}\CM_{1}^{(\ell)}(\lambda)
\CF'[H](\xi', y_N)Be^{-B(x_N+y_N)} 
\\
+ 
&\lambda^{k}(i\xi')^{\kappa'}\CM_{2}^{(\ell)}(\lambda)\CF'[H](\xi', y_N)B^2M(x_N+y_N)\Bigr](x')\,dy_N.
\end{align*}
If $2k+|\kappa'| + \ell\leq 2$, then $\lambda^{k}(i\xi')^{\kappa'}\CM_{1}^{(\ell)}(\lambda)
\in \BM_0$ and $\lambda^{k}(i\xi')^{\kappa'}\CM_{2}^{(\ell)}(\lambda) \in \BM_0$.  Thus, 
by Proposition \ref{prop:2} we have
\begin{equation}\label{5.3}
\|(\lambda, \lambda^{1/2}\bar\nabla, \bar\nabla^2)\CT_h(\lambda)H\|_{L_q(\HS)} 
\leq C\|H\|_{L_q(\HS)}.
\end{equation}
To obtain the estimate in $W^1_q(\HS)$, noting that $H \in C^\infty_0(\HS)^{m(N)}$,
using the formulas:
$$
\pd_N (-B)^{-1} e^{-B(x_N+y_N)} =e^{-B(x_N+y_N)}, \quad 
\pd_N (A^{-1}M(x_N+y_N) - (AB)^{-1}e^{-B(x_N+y_N)}) = M(x_N+y_N) 
$$
and setting 
$$\tilde\CM^{(\ell-1)}_1(\lambda) = (B^{-1}\CM^{\ell}_1(\lambda)+ A^{-1}\CM^{\ell}_2(\lambda)),
\quad \tilde \CM^{(\ell-1)}_2(\lambda) = -A^{-1}\CM^{\ell}_2(\lambda),
$$
by integration
by parts, we have 
\begin{align*}
\pd_N^\ell \CT_h(\lambda)H = 
\int^\infty_0\CF^{-1}_{\xi'}\Bigl[
&\tilde\CM^{(\ell-1)}_1(\lambda)\CF'[\pd_NH](\xi', y_N)Be^{-B(x_N+y_N)} \nonumber \\
&+ \tilde\CM^{(\ell-1)}_2(\lambda)\CF'[\pd_NH](\xi', y_N)B^2M(x_N+y_N)\Bigr](x')\,dy_N.
\label{dif.3}
\end{align*}
Thus, we have
\begin{align*}
\lambda^{k}\pd_{x'}^{\kappa'}\pd_N^\ell \CT_h(\lambda)H = 
\int^\infty_0\CF^{-1}_{\xi'}\Bigl[
&\lambda^{k}(i\xi')^{\kappa'}(\tilde\CM^{(\ell-1)}_1(\lambda)
\CF'[\pd_NH](\xi', y_N)Be^{-B(x_N+y_N)} \\
&-\lambda^{k}(i\xi')^{\kappa'}\tilde\CM_2^{(\ell-1)}(\lambda)
\CF'[\pd_NH](\xi', y_N)B^2M(x_N+y_N)\Bigr](x')\,dy_N.
\end{align*}
If $2k+|\kappa'| + \ell \leq 3$, both  
$\lambda^{k}(i\xi')^{\kappa'}\tilde\CM^{(\ell-1)}_1(\lambda)$ and 
$\lambda^{k}(i\xi')^{\kappa'}\tilde\CM^{(\ell-1)}_2(\lambda)$ are 
order $0$ symbols, and so by Proposition \ref{prop:2}, we have
\begin{equation}\label{5.4}\begin{aligned}
\|(\lambda, \lambda^{1/2}\bar\nabla, \bar\nabla^2) \CT_h(\lambda)H\|_{W^1_q(\HS)} 
\leq C\|H\|_{W^1_q(\HS)}, \\
\|(\lambda^{1/2}\bar\nabla, \bar\nabla^2)\CT_h(\lambda)H\|_{L_q(\HS)}
 \leq C|\lambda|^{-1/2}\|H\|_{W^1_q(\HS)}.
\end{aligned}\end{equation}
\par
We next consider $\CT_h^*(\lambda)$, which is defined by exchanging $\CF'$ and $\CF^{-1}_{\xi'}$
in the formula of $\CT_h(\lambda)$.  Namely, 
\begin{align*}
\CT_h^*(\lambda)H = 
\int^\infty_0\CF'\Bigl[&\CM_{1}(\lambda)\CF^{-1}_{\xi'}[H](\xi', y_N)Be^{-B(x_N+y_N)} \\
&+ 
\CM_{2}(\lambda)\CF^{-1}_{\xi'}[H](\xi', y_N)B^2M(x_N+y_N)\Bigr](x')\,dy_N.
\end{align*}
Then, employing the same argument as in the proof of  \eqref{5.3} and \eqref{5.4}, we have
\begin{equation}\label{5.5}\begin{aligned}
\|\lambda, \lambda^{1/2}\bar\nabla, \bar\nabla^2) \CT_h^*(\lambda)H\|_{L_{q'}(\HS)} 
& \leq C\|H\|_{L_{q'}(\HS)}, \\
\|(\lambda, \lambda^{1/2}\bar\nabla, 
\bar\nabla^2)\CT_h^*(\lambda)H\|_{W^1_{q'}(\HS)} & \leq C\|H\|_{W^1_{q'}(\HS)}, \\
\|(\lambda^{1/2}\bar\nabla, \bar\nabla^2)
 \CT_h^*(\lambda)H\|_{L_{q'}(\HS)} &\leq C|\lambda|^{-1/2}\|H\|_{W^1_{q'}(\HS)}.
\end{aligned}\end{equation}
Since $H \in C^\infty_0(\HS)$, we see that 
$(\lambda, \lambda^{1/2}\bar\nabla, \bar\nabla^2)\CT_h(\lambda)H 
= (\CT_1(\lambda)(\lambda, \lambda^{1/2}\bar\nabla, \bar\nabla^2)H)$,
which implies that $((\lambda, \lambda^{1/2}\bar\nabla, \bar\nabla^2)\CT_1(\lambda))^{*}
=(\lambda, \lambda^{1/2}\bar\nabla, \bar\nabla^2)\CT_h(\lambda)^*$. 
In view of \eqref{5.3}, \eqref{5.4}, and \eqref{5.5},  the assertion \thetag1 of Theorem \ref{thm:5.2} 
implies that  \eqref{5.1.1} and \eqref{5.2.1} hold. \par
Let $X_r \in \{L_r(\HS), W^1_r(\HS)\}$ for $r = q, q'$.  
To prove \eqref{5.2.3}, from \eqref{5.3}, \eqref{5.4}, and \eqref{5.5}, we observe that there hold:
\begin{align*}
\|T_h(\lambda)H\|_{X_q} + |\lambda^{-1/2}\bar\nabla T_h(\lambda)H\|_{X_q}
&\leq C\|\lambda|^{-1}\|H\|_{X_q}, \\
\|T_h(\lambda)H\|_{W^1_q(\HS)} + \|\lambda^{-1/2}\bar\nabla T_h(\lambda)H\|_{W^1_q(\HS)}
&\leq C|\lambda|^{-1/2}\|H\|_{L_q(\HS)}, \\
\|T_h^*(\lambda)H\|_{X_{q'}} + \|\lambda^{-1/2}\bar\nabla T_h^*(\lambda)H\|_{X_{q'}}
&\leq C|\lambda|^{-1}\|H\|_{X_{q'}}, \\
\|T_h^*(\lambda)H\|_{W^1_{q'}(\HS)} + \|\lambda^{-1/2}\bar\nabla T_h^*(\lambda)H\|_{W^1_{q'}(\HS)}
&\leq C|\lambda|^{-1/2}\|H\|_{L_{q'}(\HS)}.
\end{align*}
Thus, by Theorem \ref{thm:5.2}, we have \eqref{5.2.3}. 
This completes the proof of Theorem \ref{thm:3.1}.
\end{proof}



\begin{proof}[\bf Proof of Theorem \ref{thm:hL.1}] 
Since $C^\infty_0(\HS)$ is dense in 
$B^\nu(\HS)$ for $-1+1/q < \nu < 1/q$, we may assume that 
$\bg = (g_1, \ldots, g_N)\in C^\infty_0(\HS)^N$. For any $f$ defined in $\HS$, let 
$f_e$ and $f_o$ be its even and odd extensions, which are define by 
$$f_e(x) = \begin{cases} f(x) &\quad(x_N > 0), \\ f(x', -x_N) &\quad(x_N < 0), \end{cases}
\quad
f_o(x) = \begin{cases} f(x) &\quad(x_N > 0), \\- f(x', -x_N) &\quad(x_N < 0). \end{cases}
$$
We consider the extension $\bg_e = (g_{1e}, \ldots, g_{N-1e}, g_{No})$ of $\bg$. 
Since $\bg \in C^\infty_0(\HS)^N$, so $\bg_e \in C^\infty_0(\BR^N)^N$.  Let $\CS(\lambda)$ be 
the solution operator of equations \eqref{L.1}, which is given in Theorem \ref{thm:L.1}. 
Let $\bu_1 = \CS(\lambda)\bg_e$, and then from \eqref{sol.3.1}, we see that $u_{1N}|_{x_N=0}=0$.
Here, $u_{1N}$ denotes the $N$-th component of $\bu_1$. \par
Let $\CT_h(\lambda)$ be the solution operator of the compensative  equations 
\eqref{july.1} given in Theorem \ref{thm:3.1}. Let $(\CS(\lambda)\bg_e)_i$ denote
the $i$-th component of $\CS(\lambda)\bg_e$ and set $(\CS(\lambda)\bg_e)' = ((\CS(\lambda)\bg_e)_1,
\ldots, (\CS(\lambda)\bg_e)_{N-1})$.  
Let $\bu_2 = \CT_\lambda(\lambda)\CD_\lambda(\CS(\lambda)\bg_e)'$, and then
$\bu=\bu_1-\bu_2$ is a solution of equations \eqref{hL.1}. 
\par
Let 
$$\CS_h(\lambda)\bg = \CS(\lambda)\bg_e - \CT_h(\lambda)\CD_\lambda(\CS(\lambda)\bg_e)'.$$
Notice that $\CS_h(\lambda)\bg = \bu$ is a solution of equations \eqref{hL.1}.
Our task is to prove that $\CS_h(\lambda)$ has the $(s, \sigma, q,r)$ properties.  
Since we know that  for $\CS(\lambda)$ has the $(s,\sigma,q,r)$ properties 
from Theorem \ref{thm:L.1}.   In what follows, we use Theorems \ref{thm:L.1}
to estimate $\CD_\lambda(\CS(\lambda)\bg_e)'$ and Theorem \ref{thm:3.1} to estimate
$\CT_h(\lambda)$.  Noting that $\|\bg_e\|_{B^\nu_{q,r}(\BR^N)} 
\leq C\|\bg\|_{B^\nu_{q,r}(\HS)}$, we observe that
\begin{align*}
&\|(\lambda, \lambda^{1/2}\bar\nabla, \bar\nabla^2)\CT_h(\lambda)\CD_\lambda(\CS(\lambda)\bg_e)'
\|_{B^\nu_{q,r}(\HS)} 
\leq C\|\CD_\lambda(\CS(\lambda)\bg_e)'\|_{B^\nu_{q,r}(\HS)} \\
&\quad \leq C\|(\lambda, \lambda^{1/2}\bar\nabla, \bar\nabla^2)\CS(\lambda)\bg_e\|_{B^\nu_{q,r}(\BR^N)}
\leq C\|\bg\|_{B^\nu_{q,r}(\HS)}, \\
&\|(\lambda, \lambda^{1/2}\bar\nabla, \bar\nabla^2)\CT_h(\lambda)\CD_\lambda(\CS(\lambda)\bg_e)'
\|_{B^\nu_{q,r}(\HS)} 
\leq C|\lambda|^{-\frac{\sigma}{2}}\|\CD_\lambda(\CS(\lambda)\bg_e)'\|_{B^{s+\sigma}_{q,r}(\HS)} \\
&\quad \leq C|\lambda|^{-\frac{\sigma}{2}}\|(\lambda, \lambda^{1/2}\bar\nabla, \bar\nabla^2)\CS(\lambda)
\bg_e\|_{B^{s+\sigma}_{q,r}(\HS)}
\leq C|\lambda|^{-\frac{\sigma}{2}}\|\bg\|_{B^{s+\sigma}_{q,r}(\HS)}, \\
&\|(1, \lambda^{-1/2}\bar\nabla)\CT_h(\lambda)\CD_\lambda(\CS(\lambda)\bg_e)'
\|_{B^\nu_{q,r}(\HS)} \leq C|\lambda|^{-(1-\frac{\sigma}{2})}
\|\CD_\lambda(\CS(\lambda)\bg_e)'\|_{B^{s-\sigma}_{q,r}(\HS)} \\
&\quad \leq C|\lambda|^{-(1-\frac{\sigma}{2})}\|(\lambda, \lambda^{1/2}\bar\nabla, \bar\nabla^2)\CS(\lambda)
\bg_e\|_{B^{s-\sigma}_{q,r}(\BR^N)}
\leq C|\lambda|^{-(1-\frac{\sigma}{2})}\|\bg\|_{B^{s-\sigma}_{q,r}(\HS)}, 
\end{align*}
Therefore, we see that $\CS_h(\lambda)$ has the estimates stated in \eqref{L1prop}. 
Moreover, using the relation: $\pd_\lambda \CS_h(\lambda) = -\CS_h(\lambda)\CS_h(\lambda)$,
we see that $\pd_\lambda \CS_h(\lambda)$ has the estimates stated in \eqref{L1prop}. 
Namely, we see that $\CS_h(\lambda)$ has $(s, \sigma, q,r)$ properties. 
This completes the proof of Theorem \ref{thm:hL.1}.
\end{proof}

\subsection{The bent half space case.} \label{sec.4.3}

Let $x_0 \in \pd\Omega$. As was seen in  \cite[Appendix]{ES13}
or in \cite[Subsec. 3.2.1]{S20}, 
there exist a constant $d>0$, a 
diffeomorphism of  $C^3$ class $\Phi: \BR^N \to \BR^N$, $x \mapsto y=\Phi(x)$ 
and its inverse map $\Phi^{-1}: \BR^N \to \BR^N$, $y \mapsto x=\Phi^{-1}(x)$  such that 
$\Phi(0)=x_0$, 
 $B_d(x_0) \cap \Omega \subset \Phi(\HS)$, and $B_d(x_0) \cap \pd\Omega
\subset \Phi(\BR^N_0)$  and 
\begin{equation*}\label{coef.10}
\nabla \Phi = \CA + \CB(x), \quad \nabla \Phi^{-1}(y) = \CA_- + \CB_-(y)
\end{equation*}
where  $\CA$ and $\CA_-$ are $N\times N$ orthogonal matrices of constant coefficients
such that $\CA\CA_- = \CA_-\CA = I$ and $\CB(x)$ and $\CB_-(y)$ are $N\times N$
matrices of $C^2$ functions. 
Here and in the following, we write $B_d(x_0) = \{y \in \BR^N \mid |y-x_0| < d\}$. \par
From the construction of diffeomorphisms $\Phi$ and $\Phi^{-1}$
(cf. \cite[Appendix]{ES13}
or in \cite[Subsec. 3.2.1]{S20}), 
we may assume that for any  constant  $M_1 > 0$
we can choose  $0 < d < 1$ small enough in such a way that 
\begin{equation}\label{coef.11}
\|(\CB, \CB_-)\|_{L_\infty(\BR^N)}  \leq M_1. 
\end{equation}
Furthermore, we may assume that there exist constants 
$D$ and $M_2$ such that  
\begin{equation}\label{coef.12}\begin{aligned}
\|\nabla(\CB, \CB_-)\|_{L_\infty(\BR^N)}&\leq D\\
\|\nabla^2(\CB, \CB_-)\|_{L_\infty(\BR^N)}& \leq M_2.
\end{aligned}\end{equation}
Here, $D$ is independent of choice of 
$M_1$ and $d$, but $M_2$ depends on $M_1^{-1}$ and $d$.  
We may assume that $M_1 < 1 \leq D \leq M_2$.
\par
Let 
\begin{equation}\label{domain.1}
\Omega_+ = \Phi(\HS),  \quad \Gamma_+ = \Phi(\pd\HS).
\end{equation}
$\Omega_+$ is called a bent space. 
In this section, we consider 
 Lam\'e equations in $\Omega_+$, which reads as 
\begin{equation}\label{resol.1}
\lambda \bv - \alpha\Delta\bv - \beta\nabla\dv \bv  = \bg\quad
\text{in $\Omega_+$}, \quad 
\bv|_{\Gamma_+} = 0.
\end{equation}
We shall show the following theorem.
\begin{thm}\label{thm:bent1} Let $x_0 \in \pd\Omega$. 
Let $\Phi$ and $\Phi^{-1}$ be a $C^3$ diffeomorphism on $\BR^N$ and 
 its inverse, respectively.    Let $\Omega_+$ and $\Gamma_+$ be the bent space
and its boundary defined in \eqref{domain.1}.
Let $1 < q < \infty$, $1 \leq r \leq \infty-$, and $-1+1/q < s < 1/q$. Let $\sigma$ be a small 
positive number such that $-1+1/q < s-\sigma < s+\sigma < 1/q$.  
Let $\nu \in \{s-\sigma, s, s+\sigma\}$.
Then,
there exist a small constant $d>0$, a large constant $\lambda_1 > 0$ and an operator
$\CS_p(\lambda) \in {\rm Hol}\, (\Sigma_{\epsilon, \lambda_1},
\CL(B^\nu_{q,r}(\Omega_+)^N, B^{\nu+2}_{q,r}(\Omega_+)^N))$ having 
$(s, \sigma, q, r)$ properties in $\Omega_+$ such that 
for any $\lambda \in \Sigma_{\epsilon, \lambda_1}$ and $\bg \in B^\nu_{q,r}(\Omega_+)$,
 $\bv=\CS_p(\lambda)\bg$ is a unique solution of equations \eqref{resol.1}.
\end{thm}
\begin{proof}
First, we shall reduce problem \eqref{resol.1}
to that in the half-space $\HS$.
Let $a_{kj}$ and $b_{kj}(x)$ be the $(k, j)$th 
components of $\CA_-$ and $\CB_-(\Phi(x))$, 
and then we have
\begin{equation}\label{change.1}
\frac{\pd}{\pd y_j} = \sum_{k=1}^N(a_{kj} + b_{kj}(x))\frac{\pd}{\pd x_k}\quad(j=1, \ldots, N).
\end{equation}
Notice that 
\begin{equation}\label{change.2}
\sum_{j=1}^N a_{jk}a_{j\ell} =\sum_{j=1}^N  a_{kj}a_{\ell j} = \delta_{k\ell}.
\end{equation}
Let $\tilde \bv(x) = \bv(y)$.  We write $\tilde \bv(x) = (\tilde v_1(x), \ldots, \tilde v_N(x))^\top$ and 
$\bv(y)=(v_1(y), \ldots, v_N(y))^\top$, where $A^\top$ denotes the transposed $A$ for any
vector or matrix $A$.  
By \eqref{change.1} we have
\begin{equation}\label{dv.1}
\dv_y \bv(y) = \sum_{\ell=1}^N \frac{\pd v_\ell}{\pd y_\ell} =
\sum_{\ell, m=1}^N (a_{m\ell} +b_{m\ell}(x))\frac{\pd \tilde v_\ell}{\pd x_m}.
\end{equation}
Moreover, we set $\tilde v_\ell = \sum_{k=1}^N a_{k\ell}w_k$, 
and $\bw = (w_1, \ldots, w_N)^\top$. 
From  \eqref{change.2} and \eqref{dv.1} it follows that 
\begin{equation*}\label{dv.3}
\dv_y \bv_y = 
\sum_{\ell, m, k=1}^N (a_{m\ell} +b_{m\ell}(x))a_{k\ell}\frac{\pd w_k}{\pd x_m}
= \dv \bw + \sum_{ m, k=1}^N(\sum_{\ell=1}^N a_{k\ell} b_{m\ell}(x))\frac{\pd w_k}{\pd x_m}.
\end{equation*}
\par
For  equations \eqref{resol.1}, we observe that
\begin{align*}
&\Delta v_i  = \sum_{j=1}^N\frac{\pd^2 v_i}{\pd y_j^2} \\
& =\sum_{j, k,\ell=1}^N (a_{kj} + b_{kj}(x))\frac{\pd}{\pd x_k}((a_{\ell j} 
+ b_{\ell j}(x))\frac{\pd \tilde v_i}{\pd x_\ell})
\\
& = \sum_{j,k,\ell=1}^Na_{kj}a_{\ell j}\frac{\pd^2 \tilde v_i}{\pd x_k\pd x_\ell}
+ \sum_{j,k, \ell=1}^Nb_{kj}(x)(a_{\ell j} + b_{\ell j}(x))\frac{\pd^2\tilde v_i}{\pd x_k\pd x_\ell}
+ \sum_{j, k, \ell=1}^N(a_{kj}+b_{kj}(x))\frac{\pd b_{\ell j}}{\pd x_k}\frac{\pd \tilde v_i}{\pd x_\ell}\\
& = \Delta \tilde v_i + \sum_{j,k,\ell=1}^N(a_{kj}b_{\ell j}(x)
+ a_{\ell j}b_{kj}(x) + b_{kj}(x)b_{\ell j}(x))\frac{\pd^2 \tilde v_i} {\pd x_k\pd x_\ell}
+ \sum_{j, k, \ell=1}^N(a_{kj} + b_{lj}(x))\frac{\pd b_{\ell j}}{\pd x_k}
\frac{\pd \tilde v_i}{\pd x_\ell}\\
& = \sum_{n=1}^N a_{ni}(\Delta w_n + \sum_{j,k,\ell=1}^N(a_{kj}b_{\ell j}(x)
+ a_{\ell j}b_{kj}(x) + b_{kj}(x)b_{\ell j}(x))\frac{\pd^2 w_n} {\pd x_k\pd x_\ell}\\
&\hskip7cm+ \sum_{j, k, \ell=1}^N(a_{kj} + b_{lj}(x))\frac{\pd b_{\ell j}}{\pd x_k}
\frac{\pd w_n}{\pd x_\ell}; \\
&\frac{\pd}{\pd y_i}\dv\bv 
= \sum_{j=1}^N(a_{ji} + b_{ji}(x))\frac{\pd}{\pd x_j}(\dv\bw 
+ \sum_{k, n=1}^N(\sum_{\ell=1}^N a_{k\ell}b_{n\ell}(x))\frac{\pd w_k}{\pd x_n}).
\end{align*}
Thus, we have
\allowdisplaybreaks
\begin{align*}
g_i & = \lambda \sum_{n=1}^N a_{ni}w_n \\
& - \alpha \sum_{n=1}^N a_{ni}\bigl\{ \Delta w_n + \sum_{j,k,\ell=1}^N(a_{kj}b_{\ell j}(x)
+ a_{\ell j}b_{kj}(x) + b_{kj}(x)b_{\ell j}(x))\frac{\pd^2 w_n} {\pd x_k\pd x_\ell} \\
&\hskip3cm+ \sum_{j, k, \ell=1}^N(a_{kj} + b_{lj}(x))\frac{\pd b_{\ell j}}{\pd x_k}
\frac{\pd w_n}{\pd x_\ell}\bigr\} \\
&-\beta \sum_{j=1}^N(a_{ji} + b_{ji}(x))\frac{\pd}{\pd x_j}\bigl\{\dv\bw 
+ \sum_{k, n=1}^N(\sum_{\ell=1}^N a_{k\ell}b_{n\ell}(x))\frac{\pd w_k}{\pd x_n}\bigr\}.
\end{align*}
Noticing $\sum_{i=1}^N a_{ni}a_{mi} = \delta_{nm}$ and $\sum_{i=1}^Na_{ji}a_{mi} = \delta_{jm}$, 
where $\delta_{ij}$ denote the Koronecker delta symbols such that
$\delta_{ii}=1$ and $\delta_{ij}=0$ for $i\not=j$, we have 
\allowdisplaybreaks
\begin{align*}
&\sum_{i=1}^Na_{mi}g_i(\Phi(x)) = \lambda w_m \\
&-\alpha\bigl\{ \Delta w_m + \sum_{j,k,\ell=1}^N(a_{kj}b_{\ell j}(x)
+ a_{\ell j}b_{kj}(x) + b_{kj}(x)b_{\ell j}(x))\frac{\pd^2 w_m} {\pd x_k\pd x_\ell} \\
&\quad + \sum_{j, k, \ell=1}^N(a_{kj} + b_{lj}(x))\frac{\pd b_{\ell j}}{\pd x_k}
\frac{\pd w_m}{\pd x_\ell}\bigr\} \\
& -\beta \frac{\pd}{\pd x_m}(\dv \bw + \sum_{k,n=1}^N
(\sum_{\ell=1}^Na_{k\ell}b_{n\ell}(x))\frac{\pd w_k}{\pd x_n}) \\
&-\beta \sum_{j=1}^N(\sum_{i=1}^Na_{mi}b_{ji}(x))\frac{\pd}{\pd x_j}
(\dv \bw + \sum_{k,n=1}^N
(\sum_{\ell=1}^Na_{k\ell}b_{n\ell}(x))\frac{\pd w_k}{\pd x_n}).
\end{align*}
Let 
\allowdisplaybreaks
\begin{align*}
\tilde g_m &= \sum_{i=1}^N a_{mi}g_i(\Phi(x)), \quad
\tilde\bg(x) = (\tilde g_1(x), \ldots, \tilde g_m(x))^\top, \\
\CR_{2j}\bw & =\alpha  \sum_{j,k,\ell=1}^N(a_{kj}b_{\ell j}(x)
+ a_{\ell j}b_{kj}(x) + b_{kj}(x)b_{\ell j}(x))\frac{\pd^2 w_m} {\pd x_k\pd x_\ell} \\
& + \beta\sum_{k,n=1}^N
(\sum_{\ell=1}^Na_{k\ell}b_{n\ell}(x))\frac{\pd^2 w_k}{\pd x_m\pd x_n}) \\
&-\beta \sum_{j=1}^N\bigl\{\sum_{i=1}^Na_{mi}b_{ji}(x))
(\frac{\pd}{\pd x_j}\dv \bw + \sum_{k,n=1}^N
(\sum_{\ell=1}^Na_{k\ell}b_{n\ell}(x))\frac{\pd^2 w_k}{\pd x_j\pd x_n}\bigr\}, \\
\CR_{1j}\bw& = \alpha \sum_{j, k, \ell=1}^N(a_{kj} + b_{lj}(x))\frac{\pd b_{\ell j}}{\pd x_k}
\frac{\pd w_n}{\pd x_\ell} + \beta \sum_{k,n=1}^N 
(\sum_{\ell=1}^N a_{k\ell}\frac{\pd b_{n\ell}}{\pd x_m})\frac{\pd w_k}{\pd x_n}) \\ 
&+ \beta(\sum_{i, j=1}^Na_{mi}b_{ji}(x))\sum_{k,n=1}^N
(\sum_{\ell=1}^Na_{k\ell}\frac{\pd b_{n\ell}}{\pd x_j})\frac{\pd w_k}{\pd x_n})\\
\CR_2\bw &= (\CR_{21}\bw, \ldots, \CR_{2N}\bw), \quad 
\CR_1\bw = (\CR_{11}\bw, \ldots, \CR_{1N}\bw).
\end{align*}
Then, we have
\begin{equation}\label{benteq.1}
\lambda \bw - \alpha \Delta\bw - \beta \nabla\dv\bw
+ \CR_2\bw + \CR_1\bw  = \tilde \bg \quad\text{in $\HS$}, 
\quad \bw|_{x_N=0}=0.
\end{equation}
This is  reduced Lam\'e equations in the half-space. \par
We now solve equations \eqref{benteq.1} by using Theorem \ref{thm:hL.1}.  Let 
$\CS_h(\lambda)$ be the solution operator of equations \eqref{hL.1}
given in Theorem \ref{thm:hL.1}.  Set $\bw = \CS_h(\lambda)\tilde \bg$ and 
insert it into equations \eqref{hL.1}.  Then, setting 
$$\CR_h(\lambda)\tilde\bg = \CR_2\bw + \CR_1\bw = \CR_2\CS_h(\lambda)\tilde\bg
+ \CR_1\CS_h(\lambda)\tilde\bg,
$$
we have
\begin{equation*}\label{p.eq1}
\lambda \bw - \alpha \Delta \bw -\beta \nabla\dv \bw
+ \CR_2\bw + \CR_1\bw = (\bI +\CR_h(\lambda))\tilde\bg
\quad\text{in $\HS$}, \quad \bw|_{x_N=0} = 0. 
\end{equation*}
To estimate $\CR_h(\lambda)\tilde\bg$, we use the following lemma.
\begin{lem}\label{lem:APH} Let $1 < q < \infty$, $1 \leq r \leq \infty$, and $-1+1/q < s < 1/q$.
Let $p_2$ be an exponent such that $N < p_2 < \min(q, q')N$.  Then, we have
\begin{equation}\label{p.prod1}
\|uv\|_{B^s_{q,r}(\HS)} \leq C\|u\|_{B^s_{q,r}(\HS)}\|v\|_{B^{N/p_2}_{p_2,r}(\HS) \cap L_\infty(\HS)}.
\end{equation}
\end{lem}
\begin{proof}
By using an extension map from $\HS$ into $\BR^N$, it is sufficient to prove the lemma
in the case where the domain is $\BR^N$ instead of $\HS$.  Below, we omit $\BR^N$. 
We shall use the Abidi-Paicu theory \cite[Cor.2.5]{AP07} or
 the Haspot theory \cite[Prop. 2.3]{H11}.
According to the Abidi-Paicu-Haspot theory, we have
$$\|uv\|_{B^{s_1+s_2-N(\frac{1}{p_1}+\frac{1}{p_2}-\frac{1}{q})}_{q,r} }
\leq C\|u\|_{B^{s_1}_{q,r}}
\|v\|_{B^{s_2}_{p_2, r}\cap L_\infty}$$
provided that $1/q \leq 1/p_1+1/\lambda_1 \leq 1$, $1/q \leq 1/p_2 + 1/\lambda_2 \leq 1$, 
$1/q \leq 1/p_1 + 1/p_2$, $p_1 \leq \lambda_2$, $p_2 \leq \lambda_1$, 
$s_1+s_2+N\inf(0, 1-1/p_1-1/p_2)>0$, $s_1 + N/\lambda_2 < N/p_1$ and $s_2 + N/\lambda_1 \leq N/p_2$. 
We choose $p_1=q$, $s_1=s$ and $s_2 = N(1/p_1+1/p_2-1/q)= N/p_2$. 
In particular, $s_1+s_2-N(\frac{1}{p_1}+\frac{1}{p_2}-\frac{1}{q})=s$. Let $\lambda_1=\infty$, and then
$1/q \leq 1/q + 0 \leq 1$, $p_2 \leq \lambda_1$.  We choose $\lambda_2$ in such a way that
$1/\lambda_2 = 1/q -1/p_2$ when $1/q \geq 1/p_2$ and $\lambda_2=\infty$ when
$1/q < 1/p_2$. When $1/q < 1/p_2$, we have $s_1 + N/\lambda_2 = s < 1/q < N/q$ 
When $1/q \geq 1/p_2$, we have 
$s_1 + N/\lambda_2 = s + N(1/q-1/p_2) < N/q$ , namely we choose $p_2$ such that  
$s - N/p_2 <0$. Since $s < 1/q$, we choose $p_2$ such that $1/q \leq N/p_2$, 
that is $p_2 \leq qN$.  Thus, so far we choose $p_2$ in such a way that $N < p_2 < qN$.
Since $\lambda_1=\infty$, the condition $p_2 \leq \lambda_1$ is satisfied. When $1/q \geq 1/p_2$, 
$\lambda_2^{-1} = 1/q -1/p_2 < 1/q$, and so $q < \lambda_2$.
When $1/q < 1/p_2$, $\lambda_2=\infty$, and so $q \leq \lambda_2$.
When $1-1/q-1/p_2 \geq 0$, that is $p_2 \geq q'$, $s_1+s_2+N\inf(0, 1/p_1-1/p_2) = s+N/p_2 > 0$.
Since $s > -1+1/q = -1/q'$, we have $-N/p_2 < -1/q'$ provided that $p_2 \leq Nq'$. 
When $1-1/q-1/p_2 < 0$, that is  $p_2 < q'$, 
$s_1+s_2+N\inf(0, 1-1/p_1-1/p_2) = s+N/p_2+N/q'-N/p_2 = s+N/q' > 0$
because $s > -1/q'$.  Summing up, if $N < p_2 < \min(q,q')N$, then the Abidi-Paicu-Haspot 
conditions are all satisfied.  Thus, we have \eqref{p.prod1}.  This completes the proof of 
Lemma \ref{lem:APH}.
\end{proof}
\begin{lem}\label{lem:dif1} Let $1 < q < \infty$, $1 \leq r \leq \infty-$ and $-1+1/q < s < 1/q$.
Then, for $f \in B^s_{q,r}(\HS)$ and $g \in W^1_\infty(\HS)$, there holds 
\begin{equation}\label{dif.1}
\|fg\|_{B^s_{q,r}(\HS)} \leq C_s\|f\|_{B^s_{q,r}(\HS)}\|g\|_{L_\infty(\HS)}^{1-|s|}
\|g\|_{W^1_\infty(\HS)}^{|s|}.
\end{equation}
provided that $s\not=0$ and 
\begin{equation}\label{dif.2}
\|fg\|_{B^0_{q,r}(\HS)} \leq C_\epsilon\|f\|_{B^0_{q,r}(\HS)}\|g\|_{L_\infty(\HS)}^{1-\epsilon}
\|g\|_{W^1_\infty(\HS)}^{\epsilon}.
\end{equation}
with any small $\epsilon > 0$. Here, 
 $C_s$ and $C_\epsilon$ denote constants being independent of $f$ and $g$.
\end{lem}
\begin{proof} First, we consider the case where $0 < s < 1/q$.  Since $C^\infty_0(\HS)$ is 
dense in $B^s_{q,r}(\HS)$, we may assume that $f \in C^\infty_0(\HS)$.  We know that 
\begin{equation}\label{real.1}
(L_q(\HS), W^1_q(\HS))_{s, r} = B^s_{q,r}(\HS).
\end{equation}
Here, $(\cdot, \cdot)_{s, r}$ denotes the real interpolation functor. 
We see easily that 
$$\|fg\|_{L_q(\HS)} \leq \|f\|_{L_q(\HS)} \|g\|_{L_\infty(\HS)}, \quad
\|fg\|_{W^1_q(\HS)} \leq \|f\|_{W^1_q(\HS)} \|g\|_{W^1_\infty(\HS)}.$$	
Since $(\cdot, \cdot)_{s,r}$ is an exact interpolation functor of exponent $s$ (cf. \cite[p.41,
in the proof of Theorem 3.1.2]{BL}), we have
$$\|fg\|_{B^s_{q,r}(\HS)} \leq C\|g\|_{L_\infty(\HS)}^{1-s}
\|g\|_{W^1_\infty(\HS)}^s\|f\|_{B^s_{q,r}(\HS)}.
$$
This shows \eqref{dif.1} for $0 < s < 1/q$. \par
Next we consider the case where $-1+1/q < s < 0$.  For any $\varphi \in C^\infty_0(\HS)$, 
we have
\begin{align*}
|(fg, \varphi)_{\HS} &= |(f, g\varphi)_{\HS}| \leq \|f\|_{B^s_{q,r}(\HS)}
\|g\varphi\|_{B^{-s}_{q',r'}(\HS)} \\
& \leq C\|g\|_{L_\infty(\HS)}^{1-|s|}
\|g\|_{W^1_\infty(\HS)}^{|s|} \|f\|_{B^s_{q,r}(\HS)}\|\varphi\|_{B^{-s}_{q', r'}(\HS)}.
\end{align*}
Since $C^\infty_0(\HS)$ is dense in $B^{-s}_{q', r'}(\HS)$, we have
$$\|fg\|_{B^s_{q,r}(\HS)} \leq C\|g\|_{L_\infty(\HS)}^{1-|s|}
\|g\|_{W^1_\infty(\HS)}^{|s|} \|f\|_{B^s_{q,r}(\HS)}.$$
Since $B^0_{q,r}(\HS) = (B^\epsilon_{q,r}(\HS), B^{-\epsilon}_{q,r}(\HS))_{1/2, r}$ 
for any $\epsilon > 0$,  we have \eqref{dif.2}.
This completes the proof of Lemma \ref{lem:dif1}.
\end{proof}
{\bf Continuation of the proof of Theorem \ref{thm:bent1}.}  
Let $\lambda \in \Sigma_{\epsilon, \lambda_0}$. For $\nu \in \{s-\sigma, s, s+\sigma\}$, 
using Theorem \ref{thm:hL.1}, 
we have
\begin{equation*}\label{pert.1}
\|(\lambda, \lambda^{1/2}\bar\nabla, \bar\nabla^2)\bw\|_{B^\nu_{q,r}(\HS)}
= \|(\lambda, \lambda^{1/2}\bar\nabla, \bar\nabla^2)\CS_h(\lambda)\tilde\bg\|_{B^\nu_{q,r}(\HS)}
\leq C_{\lambda_0}\|\tilde\bg\|_{B^\nu_{q,r}(\HS)}.
\end{equation*}
Since we assume that $-1+1/q < s < 1/q$, we see that $|s| \leq \max(1/q,  1/q')$.
Let $\kappa = \max(1/q,  1/q') < 1$. Using Lemma \ref{lem:dif1} and \eqref{coef.11}
and \eqref{coef.12} and recalling that $\bw = \CS_h(\lambda)\tilde\bg$,  we have
\begin{align*}
\|\CR_h(\lambda)\tilde\bg\|_{B^\nu_{q,r}(\HS)} 
&\leq C(M_1^{1-\kappa}D^\kappa\|\nabla^2 \bw\|_{B^\nu_{q,r}(\HS)}
+ D^{1-\kappa}M_2^\kappa \|\nabla\bw\|_{B^\nu_{q,r}(\HS)}) \\
&\leq C(M_1^{1-\kappa}D^\kappa + |\lambda|^{-1/2}D^{1-\kappa}M_2^\kappa )
\|\tilde\bg\|_{B^\nu_{q,r}(\HS)}.
\end{align*}
Recall that $D$ is independent of $M_1$ and $M_2$. Choosing $M_1>0$ so small and 
$\lambda_1>0$ so large in such a way that 
\begin{gather}
CM_1^{1-\kappa}D^\kappa < 1/4, \label{small.a.1} \\
C\lambda_1^{-1/2}D^{1-\kappa}M_2^\kappa <1/4 \label{small.a.2}
\end{gather}
 we have
$$\|\CR_h(\lambda)\tilde\bg\|_{B^\nu_{q,r}(\HS)} \leq (1/2)\|\tilde\bg\|_{B^\nu_{q,r}(\HS)}
$$
for any $\lambda \in \Sigma_{\epsilon, \lambda_1}$.  Thus,  the inverse operator
$(\bI+\CR_h(\lambda))^{-1}$ exists in $\CL(B^\nu_{q,r})$ and 
$\|(\bI-\CR_h(\lambda))^{-1}\tilde\bg\|_{B^\nu_{q,r}(\HS)} \leq 2\|\tilde\bg\|_{B^s_{q,r}(\HS)}.$
Let $\bu = \CS_h(\lambda)(\bI+\CR_h(\lambda))\tilde\bg$, and then $\bu 
\in B^{\nu+2}_{q,r}(\HS)$ and $\bu$ satisfies equations \eqref{benteq.1}, that is
\begin{equation}\label{pert.eq1}
\lambda \bu - \alpha \Delta\bu - \beta\nabla\dv\bu + \CR_2\bu + \CR_1\bu =\tilde\bg
\quad\text{in $\HS$}, \quad \bu|_{x_N=0} = 0.
\end{equation}
By  Theorem \ref{thm:hL.1}, we have
\begin{equation*}\label{pert.2}
\|(\lambda, \lambda^{1/2}\bar\nabla, \bar\nabla^2)\bu\|_{B^\nu_{q,r}(\HS)}
\leq C\|(\bI+\CR_h(\lambda))^{-1}\tilde\bg\|_{B^\nu_{q,r}(\HS)} 
\leq 2C\|\tilde\bg\|_{B^\nu_{q,r}(\HS)}.
\end{equation*}
Moreover, for $\tilde\bg \in C^\infty_0(\HS)$, noting that $(\bI + \CR_h(\lambda))\tilde\bg
\in B^{s\pm\sigma}_{q,r}(\HS)$, we have
\begin{align}\label{pert.3}
\|(\lambda, \lambda^{1/2}\bar\nabla, \bar\nabla^2)\bu\|_{B^\nu_{q,r}(\HS)}
\leq C|\lambda|^{-\frac{\sigma}{2}}\|(\bI+\CR_h(\lambda))^{-1}\tilde\bg\|_{B^{s+\sigma}_{q,r}(\HS)} 
\leq 2C|\lambda|^{-\frac{\sigma}{2}}\|\tilde\bg\|_{B^{s+\sigma}_{q,r}(\HS)}, \\
\label{pert.4}
\|(1, \lambda^{-1/2}\bar\nabla)\bu\|_{B^\nu_{q,r}(\HS)}
\leq C|\lambda|^{-(1-\frac{\sigma}{2})}\|(\bI+\CR_h(\lambda))^{-1}\tilde\bg\|_{B^{s-\sigma}_{q,r}(\HS)} 
\leq 2C|\lambda|^{-(1-\frac{\sigma}{2})}\|\tilde\bg\|_{B^{s-\sigma}_{q,r}(\HS)}.
\end{align}
\par
We now define $\bv$ by $\bv = \bu\circ\Phi^{-1}$. From the definition of equations \eqref{pert.eq1}
we see that $\bu$ satisfies equations \eqref{resol.1}. \par
To estimate $\bv$, we shall use the following lemma.
\begin{lem}\label{comp.1} Let $1 < q < \infty$, $1 \leq r \leq \infty-$, and $-1 +1/q < s < 1/q$.
Let $\sigma$ be a small positive number such that $-1+1/q < s-\sigma < s+\sigma < 1/q$. 
Let $\nu \in \{s-\sigma, s, s+\sigma\}$. Then,
we have
\begin{align*}
\|(\lambda, \lambda^{1/2}\bar\nabla, \bar\nabla^2)\bv\|_{B^s_{q,r}(\Omega_+)} 
&\leq C\|(\lambda, \lambda^{1/2}\bar\nabla, \bar\nabla^2)\bu\|_{B^s_{q,r}(\HS)},
\\
\|\tilde\bg\|_{B^\nu_{q,r}(\HS)} &\leq C\|\bg\|_{B^\nu_{q,r}(\Omega_+)}.
\end{align*}
Here, $C$ denote a constant depending on $D$ in \eqref{coef.12}.
\end{lem}
\begin{proof} It is sufficient to prove that for $f \in B^\nu_{q,r}(\HS)$, 
\begin{equation}\label{dif.4}
\|f\circ\Phi^{-1}\|_{B^\nu_{q,r}(\Omega_+)} \leq C\|f\|_{B^\nu_{q,r}(\HS)}. 
\end{equation}
In fact, $\Phi$ is a diffeomorphism of $C^3$ class, and so we can also show that 
$$\|g\circ\Phi\|_{B^\nu_{q,r}(\HS)} \leq C\|g\|_{B^\nu_{q,r}(\Omega_+)} 
\quad\text{for any $g \in B^\nu_{q,r}(\Omega_+)$}.$$ \par
Since $C^\infty_0(\HS)$ is dense in $B^\nu_{q,r}(\HS)$, 
we may assume that $f \in C^\infty_0(\HS)$.
For $0 < \nu < 1/q$, we shall use \eqref{real.1}. We have
\begin{align*}
\|f\circ\Phi^{-1}\|_{L_q(\Omega_+)} &= \Bigl(\int_{\Omega}|f(x)|^q|\det \nabla\Phi(x)|\,dx\Bigr)^{1/q}
\leq \|\det(\nabla\Phi)\|_{L_\infty(\BR^N)}^{1/q}\|f\|_{L_q(\Omega)}, \\
\|\nabla (f\circ\Phi^{-1})\|_{L_q(\Omega_+)} &\leq \|\nabla\Phi^{-1}\|_{L_\infty(\BR^N)}\
\|(\nabla f)\circ\Phi^{-1}\|_{L_q(\Omega_+)}\\
&\leq C(\|\nabla\Phi^{-1}\|_{L_\infty(\BR^N)}\|\det (\nabla\Phi)\|_{L_\infty(\BR^N)}^{1/q}
\|\nabla f\|_{L_q(\HS)}.
\end{align*}
Thus, by \eqref{real.1}, we have \eqref{dif.4}, where $C$ is a constant depending on
$\|\nabla \Phi\|_{L_\infty(\BR^N)}$ and $\|\nabla \Phi^{-1}\|_{L_\infty(\BR^N)}$. \par
Let $-1+1/q < \nu < 0$.  For any $\varphi \in C^\infty_0(\Omega_+)$, we have
\begin{align*}
|(f\circ\Phi^{-1}, \varphi)_{\Omega_+}| = |(f, (\varphi\circ\Phi) (\det(\nabla \Phi)))_{\HS}|
\leq \|f\|_{B^{\nu}_{q,r}(\HS)}\|(\varphi\circ\Phi)\det(\nabla \Phi)\|_{B^{-\nu}_{q',r'}(\HS)}.
\end{align*}
In the similar manner to the proof of Lemma \ref{lem:dif1}, we see that 
$$\|(\varphi\circ\Phi)(\det(\nabla \Phi))\|_{B^{-\nu}_{q',r'}(\HS)}.
 \leq C\|\varphi\circ\Phi \|_{B^{-\nu}_{q', r'}(\HS)}
$$
with some constant $C$ depending on $\|\nabla \Phi\|_{L_\infty(\BR^N)}$ and 
$\|\nabla^2 \Phi\|_{L_\infty(\BR^N)}$. Applying \eqref{dif.4} yields 
$\|\varphi\circ\Phi \|_{B^{-\nu}_{q', r'}(\HS)} \leq C\|\varphi\|_{B^{-\nu}_{q', r'}(\Omega_+)}$. 
For any $\varphi \in C^\infty_0(\Omega_+)$, we have
$$|(f\circ\Phi^{-1}, \varphi)_{\Omega_+}|  \leq C\|f\|_{B^\nu_{q,r}(\HS)}\|\varphi\|_{B^{-\nu}_{q',r'}(\Omega_+)}.
$$
Since $C^\infty_0(\Omega_+)$ is dense in $B^{-\nu}_{q', r'}(\Omega_+)$, this shows \eqref{dif.4}
for $-1+1/q < \nu < 0$.  \par
When $\nu=0$, we use the relation: 
$B^0_{q,r}(\HS) = (B^{-\epsilon}_{q,r}(\HS), B^\epsilon_{q,r}(\HS))_{1/2, r}$
for any small $\epsilon > 0$.  Thus from the results for $\nu\not=0$, it follows that 
\eqref{dif.4} holds for $\nu=0$.  This completes the proof of Lemma \ref{comp.1}.
\end{proof}
{\bf Continuation of the proof of Theorem \ref{thm:bent1}.}  Obviously, using 
\eqref{pert.3}, \eqref{pert.4}, and Lemma \ref{comp.1}, we have 
\begin{align*}
\|(\lambda, \lambda^{1/2}\bar\nabla, \bar\nabla^2)\bv\|_{B^s_{q,r}(\Omega_+)}
&\leq C|\lambda|^{-\frac{\sigma}{2}}\|\bg\|_{B^{s+\sigma}_{q,r}(\Omega_+)}, \\
\|(1, \lambda^{-1/2}\bar\nabla)\bv\|_{B^s_{q,r}(\Omega_+)}
&\leq C|\lambda|^{-(1-\frac{\sigma}{2})}\|\bg\|_{B^{s-\sigma}_{q,r}(\Omega_+)}.
\end{align*}
Moreover, using the properties that $\pd_\lambda \CS_h(\lambda) = -\CS_h(\lambda)
\CS_h(\lambda)$, we see that $\CS_h(\lambda)$ has $(s, \sigma, q, r)$ properties
in $\Omega_+$. 
Recall that  $d>0$ has been chosen so small that the inequality \eqref{small.a.1} holds
and that $\lambda_1>0$ has been chosen so large that the inequality 
\eqref{small.a.2} holds. Thus, $d>0$ and $\lambda_1$ depend on $D$. 
Moreover, the constants appearing in the proof of Theorem \ref{thm:bent1} depend
on $D$ and $\CS_h(\lambda)$. 
But, $\CS_h(\lambda)$ is fixed, and so the constants appearing in the proof of
Theorem \ref{thm:bent1} depend only on $D$.  This completes the proof of 
Theorem \ref{thm:bent1}
\end{proof}

\subsection{On the spectral analysis of  generalized Lam\'e equations in 
$\Omega$}\label{sec.4.4}

In this subsection, we consider the following equations:
\begin{equation}\label{eq.7.1}\begin{aligned}
\eta_0\lambda \bz - \alpha\Delta \bz - \beta\nabla\dv\bz= \bg&
&\quad&\text{in $\Omega$}, \quad \bz|_{\pd\Omega}=0.
\end{aligned}\end{equation}
Here, $\eta_0=\rho_*$ or $\eta_0(x) = \rho_* + \tilde\eta_0(x)$, where 
$\rho_*$ is a positive constant, and $\tilde\eta_0(x) \in B^{N/q+1}_{q,1}(\Omega)$
is a given function.  We assume that $\rho_*$ and 
$\eta_0$ satisfies the condition \eqref{july.2}.  
 In this section, we shall show the following theorem.
\begin{thm}\label{thm:7.1} Let $\epsilon \in (0, \pi/2)$ and $1 \leq r \leq \infty-$. 
\thetag1 Assume that $\eta_0 = \rho_*$.  Let $1 < q < \infty$, $-1+1/q < s <1/q$, and 
$\sigma>0$ such that $-1+1/q < \sigma < 1/q$.
\thetag2 Assume that $\tilde\eta_0\not\equiv0$.  
Let $N-1 < q < 2N$, $1 \leq r \leq \infty-$,  $-1+N/q \leq s < 1/q$,
and $\sigma>0$ such that 
$s+\sigma < 1/q$ and $\sigma < 2N/q-1$.  
Then, there exists a large constant $\lambda_2$
and an operator $\CU_\Omega(\lambda) \in {\rm Hol}\, (\Sigma_{\epsilon, \lambda_2},
\CL(B^\nu_{q,1}(\Omega)^N, B^{\nu+2}_{q,1}(\Omega)^N)$  having 
$(s, \sigma, q, r)$ properties in $\Omega$ 
 such that for any $\lambda \in \Sigma_{\epsilon, \lambda_4}$ and 
$\bg \in B^\nu_{q,1}(\Omega)^N$, $\bz = \CU_\Omega(\lambda)\bg$ is a 
unique solution of equations \eqref{eq.7.1}. 
\end{thm}
\begin{proof}
We only consider the case where $\Omega$ is an exterior domain and
$\tilde\eta_0\not\equiv0$.  Other cases can be proved analogously.  
Below, let $\nu \in \{s-\sigma, s, s+\sigma\}$. 
First, we consider the far field.  Let $\CS$ be the operator given in Theorem \ref{thm:L.1}.
Let $R>0$ be a large positive number such that
$(B_R)^c \subset \Omega$.  Replacing $\lambda$ with $\rho_*\lambda$, we see that
$\bw_R = \CS(\rho_*\lambda)(\tilde\psi_R\bg) \in B^{\nu+2}_{q,r}(\Omega)^N$ 
satisfies equations 
\begin{equation}\label{pert.8.1.1}\begin{aligned}
\rho_*\lambda \bw_R - \alpha\Delta \bw_R -\beta\nabla\dv\bw_R
= \tilde\psi_R\bg&&\quad&\text{in $\BR^N$}
\end{aligned}\end{equation}
for $\lambda \in \Sigma_{\epsilon, \lambda_3/\rho_1}$ and $\bg \in B^\nu_{q,r}(\Omega)^N$. 
Let
$$A_R =\rho_* + \psi_R(x)(\eta_0(x)-\rho_*) = \rho_* + \psi_R(x)\tilde\eta_0(x). $$
From \eqref{pert.8.1.1} it follows that 
\begin{equation*}\label{2.25}\begin{aligned}
A_R\lambda \bw_R - \alpha\Delta\bw_R -\beta\nabla\dv\bw_R 
&= \tilde\psi_R\bg - S_R(\lambda)(\tilde\psi_R\bg) &\quad&\text{in $\BR^N$},
\end{aligned}\end{equation*}
where $S_R(\lambda)$ is defined by 
$$S_R(\lambda)\bh
 = -\psi_R\tilde\eta_0\lambda\CS(\lambda)\bh
$$
for $\bh \in B^\nu_{q,1}(\BR^N)^N$. 
By Lemma \ref{lem:prod} we have
$$\|S_R(\lambda)\bh)\|_{B^s_{q,r}(\BR^N)} 
\leq C\|\psi_R\tilde\eta_0\|_{B^{N/q}_{q,1}(\BR^N)}
\|\lambda \CS(\lambda)\bh\|_{B^s_{q,r}(\BR^N)}.
$$
By Lemma 12 in \cite{KSprep},  for any $\delta > 0$ there exists an $R_0>1$ such that 
$$\|\psi_R \tilde\eta_0\|_{B^{N/q}_{q,1}(\Omega)} < \delta
$$
for any $R > R_0$. Using  Theorem \ref{thm:L.1}  we have
$$\|S_R(\lambda)\bh\|_{B^\nu_{q,r}(\BR^N)} 
\leq C\delta\|\bh\|_{B^\nu_{q,r}(\BR^N)}
$$
for $\nu \in \{s-\sigma, s, s+\sigma\}$. We choose $\delta > 0$ in such a way that 
$C\delta \leq 1/2$, we have 
\begin{equation*}\label{bound.1}\|S_R(\lambda)\bh\|_{B^\nu_{q,r}(\BR^N)} 
\leq (1/2)\|\bh\|_{B^\nu_{q,r}(\BR^N)}
\end{equation*}
for  and $R > R_0$. \par
We define $\CS_{R, \infty}$ by 
$\CS_{R, \infty}(\lambda) = \sum_{\ell=0}^\infty S_R(\lambda)^\ell$,
and then we have 
\begin{align*}
\|\CS_{R, \infty}(\lambda)\bh\|_{B^\nu_{q,r}(\BR^N)} 
&\leq \sum_{\ell=0}^\infty(1/2)^\ell\|\bh\|_{B^\nu_{q,r}(\BR^N)}
= 2\|\bh\|_{B^\nu_{q,r}(\BR^N)}.
\end{align*}
Let $\bv_R = \CS(\rho_*\lambda)S_{R, \infty}(\lambda)\tilde\psi_R\bg$.   Then, $\bv_R$
satisfies the equations:
\begin{equation}\label{pert.8.1.2}\begin{aligned}
A_R\lambda \bv_R - \alpha\Delta\bv_R- \beta\nabla \dv\bv_R
&= G_R&\quad&\text{in $\Omega$}, \quad \bv_R|_{\pd\Omega} = 0. 
\end{aligned}\end{equation}
Here, 
\begin{align*}
G_R &= S_{R,\infty}(\lambda)\tilde\psi_R\bg - S_R(\lambda)S_{R, \infty}(\lambda)\tilde\psi_R\bg
 = \tilde\psi_R\bg+\sum_{\ell=1}^\infty S_R(\lambda)^\ell\tilde\psi_R\bg
-S_R(\lambda)\sum_{\ell=0}^\infty S_R(\lambda)^\ell\tilde\psi_R\bg
= \tilde\psi_R\bg.
\end{align*}
Since $\CS$ has $(s, \sigma, q,r)$ properties, 
as follows from Theorem \ref{thm:L.1}, we have
\begin{equation}\label{est.7.2}\begin{aligned}
\|(\lambda, \lambda^{1/2}\bar\nabla, \bar\nabla^2)\CS(\rho_*\lambda)S_{R, \infty}(\lambda)
\bh\|_{B^\nu_{q,r}(\BR^N)}
&\leq C\|\bh\|_{B^\nu_{q,r}(\BR^N)}, \\
\|(\lambda^{1/2}\bar\nabla, \bar\nabla^2)
\CS(\rho_*\lambda)S_{R, \infty}(\lambda)\bh\|_{B^s_{q,r}(\BR^N)}
&\leq C|\lambda|^{-\frac{\sigma}{2}}\|\bh\|_{B^{s+\sigma}_{q,r}(\BR^N)},\\
\|(1, \lambda^{-1/2}\bar\nabla)\CS(\rho_*\lambda)S_{R, \infty}(\lambda)\bh\|_{B^s_{q,r}(\BR^N)}
& \leq C|\lambda|^{-(1-\frac{\sigma}{2})}\|\bh\|_{B^{s-\sigma}_{q,r}(\BR^N)}.
\end{aligned}\end{equation}
\par

Let $\bu_R = \psi_R(x)\bv_R$. Setting 
\begin{equation*}\label{ur.1}\begin{aligned}
U_R(\lambda)\bg &= -(2\alpha (\nabla\tilde\psi_R)\nabla \CS(\rho_*\lambda)S_{R, \infty}(\lambda)\tilde\psi_R\bg + \alpha(\Delta\tilde\psi_R) \CS(\rho_*\lambda)S_{R, \infty}(\lambda)\tilde\psi_R\bg
\\
&+\beta\nabla((\nabla\tilde\psi_R) \CS(\rho_*\lambda)S_{R, \infty}(\lambda)\tilde\psi_R\bg) + \beta(\nabla\tilde\psi_R)\nabla \CS(\rho_*\lambda)S_{R, \infty}(\lambda)\tilde\psi_R\bg).
\end{aligned}\end{equation*}
and using the facts that $\tilde\psi_R\psi_R=\psi_R$ and 
$A_R\varphi_R = \eta_0$,   from \eqref{pert.8.1.2} we see that
$\bu_R$ satisfies equations:
\begin{equation}\label{pert.8.1.3}\begin{aligned}
\eta_0 \bu_R - \alpha\Delta\bu_R-\beta\nabla \dv\bu_R
&= \psi_R\bg - U_R(\lambda)\bg&\quad&\text{in $\Omega$}, \quad
\bu_R|_{\pd\Omega}=0.
\end{aligned}\end{equation}
\par
Let $x_0 \in \Omega$ and $x_1 \in \pd\Omega$. 
Let $\lambda_1 > 0$ and $\CS_p$ be respective the constant and the operator given in
Theorem \ref{thm:bent1}.  
Let $d_{x_0}$ and $d_{x_1}$ be 
two small posive numbers  such that   
$B_{4d_{x_0}}(x_0) \subset \Omega$ and $B_{4d_{x_1}}(x_1) \cap \Omega_+ 
\subset \Omega$. Below,  $i=0$ or $1$ and in Theorem \ref{thm:L.1} we choose 
$\lambda_0=\lambda_1$. 
Let $\CS_i(\lambda)$ ($i=0,1$) 
be defined by $\CS_0(\lambda) = \CS(\lambda)$ for $i=0$ and 
$\CS_1(\lambda) = \CS_p(\lambda)$ for $i=1$.  From the assumption \eqref{july.2}
$\rho_1 \leq \eta_0(x_i) \leq \rho_2$, and so
for  $\lambda \in \Sigma_{\epsilon, \lambda_1\rho_1^{-1}}$,
$\bw_{x_i}=\CS_i(\eta_0(x_i)\lambda)(\tilde\varphi_{x_i, d_{x_i}}\bg)
\in B^{\nu+2}_{q,1}(\Omega)^N$ satisfying the equations: 
\begin{equation*}\label{pert.8.2.1}\begin{aligned}
\eta_0(x_0)\lambda \bw_{x_0} - \alpha \Delta \bw_{x_0}-\beta\nabla\dv\bw_{x_0}
= \tilde\varphi_{x_0, d_{x_0}}\bg&&\quad&\text{in $\BR^N$}, \\
\eta_0(x_1)\lambda \bw_{x_1} - \alpha\Delta \bw_{x_1} - \beta\nabla\dv\bw_{x_1}
= \tilde\varphi_{x_1, d_{x_1}}\bg&& \quad &\text{in $\Omega$}, \quad
\bw_{x_1}|_{\pd\Omega} = 0.
\end{aligned}\end{equation*}
 Let
$$A_{x_i} =\eta_0(x_i)+ \tilde\varphi_{x_i, d_{x_i}}(x)(\eta_0(x)-\eta_0(x_i)). $$
We have
\begin{equation*}\label{2.25.2}\begin{aligned}
A_{x_0}\lambda \bw_{x_0} - \alpha\Delta\bw_{x_0} -\beta\nabla\dv \bw_{x_0}
&= \tilde\varphi_{x_0, d_{x_0}}\bg - 
S_{x_0}(\lambda)\tilde\varphi_{x_0, d_{x_0}}\bg &\quad&\text{in $\BR^N$}, \\
A_{x_1}\lambda \bw_{x_1} - \alpha\Delta\bw_{x_1} -\beta\nabla\dv \bw_{x_1}
&= \tilde\varphi_{x_1, d_{x_1}}\bg - 
S_{x_1}(\lambda)\tilde\varphi_{x_1, d_{x_1}}\bg &\quad&\text{in $\Omega$}, \quad
\bw_{x_1}|_{\pd\Omega}=0, 
\end{aligned}\end{equation*}
where we have set
$$S_{x_i}(\lambda)\bh 
 = -\tilde\varphi_{x_i, d_{x_i}}(x)(\eta_0(x)-\eta_0(x_i))\lambda\CS_i(\eta_0(x_i)\lambda)\bh.
$$
By Lemma \ref{lem:prod} we have
\begin{align*}
\|S_{x_i}(\lambda)\bh\|_{B^\nu_{q,1}(D_i)} 
&\leq C\|\tilde\varphi_{x_i, d_{x_i}}(\eta_0(\cdot)-\eta_0(x_i))\|_{B^{N/q}_{q,1}(D_i)}
\|\lambda\CS_i(\eta_0(x_0)\lambda)\bh\|_{B^s_{q,1}(D_i)}.
\end{align*}
Here and in the following, $D_0=\BR^N$ and $D_1=\Omega_+$. 
By Appendix in \cite{DT22},  for any $\delta > 0$ there exists a $d_0$ 
uniformly with respect to $x_i$ such that 
\begin{align*}
\|\tilde\varphi_{x_i, d_{x_i}}(\eta_0(\cdot)-\eta_0(x_i))\|_{B^{N/q}_{q,1}(D_i)} &< \delta
\end{align*}
provided $0 < d_{x_i} \leq d_0$. By  Theorems \ref{thm:L.1} and \ref{thm:bent1},  we have
\begin{align*}
\|\lambda \CS_i(\eta_0(x_i)\lambda)\bh \|_{B^\nu_{q,r}(D_i)} 
\leq C\delta\|\bh\|_{B^\nu_{q,1}(D_i)}, \\
\end{align*}
for $\nu \in \{s-\sigma, s, s+\sigma\}$ and $0 < d_{x_0} \leq d_0$. 
We choose $\delta > 0$ in such a way that 
$C\delta \leq 1/2$, 
\begin{align*}
\|S_{x_i}(\lambda)\bh\|_{B^\nu_{q,r}(D_i)} 
&\leq (1/2)\|\bh\|_{B^s_{q,r}(D_i)}
\end{align*}
for $\nu \in \{s-\sigma, s, s+\sigma\}$ and $0 < d_{x_i} \leq d_0$. \par
We define $\CS_{x_i, \infty}$ by 
$\CS_{x_i, \infty}(\lambda) = \sum_{\ell=0}^\infty S_{x_i}(\lambda)^\ell$,
and then we have 
\begin{align*}
\|\CS_{x_i, \infty}(\lambda)\bh\|_{B^\nu_{q,r}(D_i)} \leq 2\|\bh\|_{B^\nu_{q,r}(D_i)}.
\end{align*}
Let $\bv_{x_i} = \CS_i(\eta_0(x_i)\lambda) \CS_{x_i, \infty}(\lambda)\tilde
\varphi_{x_i, d_{x_i}}\bg$.  Then, $\bv_{x_i}$
satisfies the equations:
\begin{equation}\label{pert.8.2.2}\begin{aligned}
A_{x_0}\lambda \bv_{x_i} - \alpha\Delta \bv_{x_i}- \beta\nabla\dv\bv_{x_i}
&= G_{x_i}&\quad&\text{in $\Omega$}, \quad v_{x_i}|_{\pd\Omega}=0.
\end{aligned}\end{equation}
Here, 
\begin{align*}
G_{x_i} &= \CS_{x_i,\infty}(\lambda)\tilde\varphi_{x_i, d_{x_i}}
\bg - S_{x_i}(\lambda)\CS_{x_i, \infty}(\lambda)\tilde\varphi_{x_i, d_{x_i}}\bg
\\
& = \tilde\varphi_{x_i, d_{x_i}}\bg+\sum_{\ell=1}^\infty 
S_{x_i}(\lambda)^\ell\tilde\varphi_{x_i, d_{x_i}}\bg
-S_{x_i}(\lambda)\sum_{\ell=0}^\infty S_{x_i}(\lambda)^\ell\tilde\varphi_{x_i, d_{x_i}}\bg
 = \tilde\varphi_{x_i, d_{x_i}}\bg.
\end{align*}
Noting that $\|S_{x_i, \infty}\bh\|_{B^\nu_{q,r}(D_i)}
\leq 2\|\bh\|_{B^\nu_{q,1}(D_i)}$, 
by Theorems \ref{thm:L.1} and \ref{thm:bent1}, we have
\begin{equation}\label{est.7.3}\begin{aligned}
\|(\lambda, \lambda^{1/2}\bar\nabla, \bar\nabla^2)
\CS_i(\eta_0(x_i)\lambda)S_{R, \infty}(\lambda)\bh\|_{B^\nu_{q,1}(D_i)}
&\leq C\|\bh\|_{B^\nu_{q,r}(D_i)}, \\
\|(\lambda^{1/2}\bar\nabla, \bar\nabla^2)\CS_i(\eta_0(x_i)\lambda)S_{R, \infty}(\lambda)
\bh\|_{B^s_{q,r}(D_i)}
&\leq C|\lambda|^{-\frac{\sigma}{2}}\|\bh\|_{B^{s+\sigma}_{q,1}(D_i)},\\
\|(1, \lambda^{-1/2}\bar\nabla)\CS_i(\eta_0(x_i)\lambda)S_{R, \infty}(\lambda)\bh
\|_{B^s_{q,1}(D_i)}
& \leq C|\lambda|^{-(1-\frac{\sigma}{2})}\|\bh\|_{B^{s-\sigma}_{q,1}(D_i)}.
\end{aligned}\end{equation}

Let $\bu_{x_i} = \varphi_{x_i, d_{x_i}}(x)\bw_{x_i}$.
Using the fact that $\tilde\varphi_{x_i, d_{x_i}}\varphi_{x_i, d_{x_i}} = \varphi_{x_i, d_{x_i}}$
and that 
$A_{x_i}\varphi_{x_0} = \eta_0(x)$,  and setting 
\begin{equation*}\label{ur.2}\begin{aligned}
U_{x_i}(\lambda)\bg& = -(2\alpha (\nabla\tilde\varphi_{x_i, d_{x_i}})\nabla
\CS_i(\eta_0(x_i)\lambda) \CS_{x_i, \infty}(\lambda)\tilde
\varphi_{x_i, d_{x_i}}\bg
+ \alpha(\Delta\tilde\varphi_{x_i, d_{x_i}})\CS_i(\eta_0(x_i)\lambda) \CS_{x_i, \infty}(\lambda)\tilde
\varphi_{x_i, d_{x_i}}\bg \\
&+\beta\nabla((\nabla\tilde\varphi_{x_i, d_{x_i}})\CS_i(\eta_0(x_i)\lambda) \CS_{x_i, \infty}(\lambda)\tilde
\varphi_{x_i, d_{x_i}}\bg) + 
\beta(\nabla\tilde\varphi_{x_i, d_{x_i}})\nabla\CS_i(\eta_0(x_i)\lambda) \CS_{x_i, \infty}(\lambda)\tilde
\varphi_{x_i, d_{x_i}}\bg)
\end{aligned}\end{equation*}
 from \eqref{pert.8.2.2} we see that
$\bu_{x_i}$ satisfies equations:
\begin{equation}\label{pert.8.2.3}\begin{aligned}
\eta_0\lambda \bu_{x_i} - \alpha \Delta\bu_{x_i}- \beta\nabla\dv\bu_{x_i}
&= \varphi_{x_i}\bg - U_{x_i}(\lambda)\bg&\quad&\text{in $\Omega$}, \quad 
\bw_{x_i}|_{\pd\Omega}=0.
\end{aligned}\end{equation}
\par

Now, we shall prove the theorem. 
Notice that $\Omega \cup \pd\Omega = (B_{2R})^c \cup \overline{\Omega \cap B_{2R}}.$
Since $\overline{\Omega \cap B_{2R}}$ is a compact set, 
there exist a finte set $\{x^0_j\}_{j=1}^{m_0}$ of 
points of $\Omega$ and a finite set $\{x^1_j\}_{j=1}^{m_1}$ of points of  $\pd\Omega$ such that 
$\overline{\Omega} \subset (B_{2R})^c \cup (\bigcup_{j=1}^{m_0} B_{d_{x^0_j}/2}(x^0_j)) 
\cup (\bigcup_{j=1}^{m_1}B_{d_{x^1_j}/2}(x^1_j))$. 
Let $\Phi(x) = \varphi_R(x) + (\sum_{j=1}^{m_0} \varphi_{x^0_j}(x)) + (\sum_{j=1}^{m_1}
\varphi_{x^1_j}(x))$.  Obviously, $\Phi(x) \in C^\infty(\overline{\Omega})$ 
and $\Phi(x) \geq 1$
for $x \in \overline{\Omega}$.  Thus, set 
$\omega_0(x) = \varphi_{R}(x)/\Phi(x)$, $\omega_j(x) =  \varphi_{x^0_j}(x)/\Phi(x)$ 
($j=1, \ldots, m_0$), 
and $\omega_{m_0+j}(x) = \varphi_{x^1_j}(x)/\Phi(x)$ ($j=1, \ldots, m_1$).  Then, 
$\{\omega_j\}_{j=0}^{m_0+m_1}$ is a partition of unity on $\overline{\Omega}$. 
We  define an operator $T_\Omega(\lambda)$  and $U(\lambda)$ by 
\begin{align*}
T_\Omega(\lambda)\bg &= \omega_0\CS(\rho_*\lambda)S_{R, \infty}\tilde\psi_R\bg
 + \sum_{j=1}^{m_0}\omega_j \CS(\eta_0(x^0_j)\lambda)S_{x^0_j, \infty}(\lambda)
\tilde\varphi_{x^0_j, d_{x^0_j}}\bg \\
&+\sum_{j=1}^{m_1}\omega_{m_0+j} 
\CS_p(\eta_0(x^1_{m_0+j})\lambda)S_{x^1_{m_0+j}, \infty}(\lambda)
\tilde\varphi_{x^1_{m_0+j}, d_{x^1_{m_0+j}}}\bg, \\
U(\lambda)\bg &= U_R(\lambda)\bg + \sum_{j=1}^{m_0} U_{x^0_j}(\lambda)\bg
+ \sum_{j=1}^{m_1}U_{x^1_j}(\lambda)\bg. 
\end{align*}
Then, from \eqref{pert.8.1.3} and \eqref{pert.8.2.3} we see that
$\bu= T_\Omega(\lambda)\bg$ satisfies the equations 
\begin{equation*}\label{eq:8.1*}\begin{aligned}
\lambda \bu- \alpha \Delta \bu -\beta\nabla\dv\bu = \bg - U(\lambda)\bg
&&\quad&\text{in $\Omega$},  \quad \bu|_{\pd\Omega}=0.
\end{aligned}\end{equation*}
Since the summation is finite, by \eqref{est.7.2} and 
 \eqref{est.7.3},
we see that $T_\Omega(\lambda)$ satisfies the estimates:
\begin{equation}\label{est.7.5*}\begin{aligned}
\|(\lambda, \lambda^{1/2}\bar\nabla, \bar\nabla^2)
\CT_\Omega(\lambda)\bg \|_{B^\nu_{q,1}(\Omega)}
&\leq C\|\bg\|_{B^\nu_{q,r}(\Omega)}, \\
\|(\lambda^{1/2}\bar\nabla, \bar\nabla^2)\CT_\Omega(\lambda)\bg\|_{B^s_{q,r}(\Omega)}
&\leq C|\lambda|^{-\frac{\sigma}{2}}\|\bh\|_{B^{s+\sigma}_{q,1}(\Omega)},\\
\|(1, \lambda^{-1/2}\bar\nabla)\CT_\Omega(\lambda)\bg
\|_{B^s_{q,1}(\Omega)}
& \leq C|\lambda|^{-(1-\frac{\sigma}{2})}\|\bh\|_{B^{s-\sigma}_{q,1}(\Omega)}.
\end{aligned}\end{equation}

For $U(\lambda)$, we have 
$$\|U(\lambda)\bg\|_{B^s_{q,1}(\Omega)}
\leq C(\|\bv_R\|_{B^{s+1}_{q,1}(\Omega)} + \sum_{j=1}^{m_0}
\|\bv_{x^0_j}\|_{B^{s+1}_{q,1}(\Omega)} + \sum_{j=1}^{m_1}\|\bv_{x^1_j}\|_{B^{s+1}_{q,1}(\Omega)})
\leq C|\lambda|^{-1/2}\|\bg\|_{B^s_{q,1}(\Omega)}.
$$
Choosing $\lambda_2 \geq \lambda_1\rho_1^{-1}$ in such a way that 
$C\lambda_2^{-1} \leq 1/2$, we have
$\|U(\lambda)\bg\|_{B^s_{q,1}(\Omega)} \leq (1/2)\|\bg\|_{B^s_{q,1}(\Omega)}$,
and so $(\bI - U(\lambda))^{-1}$ exists. Thus, we define an operator $\CU_\Omega(\lambda)$
by $\CU_\Omega(\lambda)= T_\Omega(\lambda)(\bI-U(\lambda))^{-1}$.  Then, 
for $\bg \in B^\nu_{q,1}(\Omega)$, $\bz = \CU_\Omega(\lambda)\bg$
 is a solution of equations \eqref{eq.7.1}.  From \eqref{est.7.5*} we see that 
$\CU_\Omega(\lambda)\bg$ satisfies estimates
\begin{align}
\|(\lambda, \lambda^{1/2}\bar\nabla, \bar\nabla^2) \CU_\Omega(\lambda)\bg
\|_{B^s_{q,1}(\Omega)} 
&\leq C\|\bg\|_{B^s_{q,r}(\Omega)}\nonumber \\
\|(\lambda, \lambda^{1/2}\bar\nabla, \bar\nabla^2) \CU_\Omega(\lambda)\bg
\|_{B^s_{q,r}(\Omega)}
& \leq C|\lambda|^{-\frac{\sigma}{2}}\|\bg\|_{B^{s+\sigma}_{q,1}(\Omega)}, \nonumber\\
\|(1, \lambda^{-1/2}\bar\nabla) \CU_\Omega(\lambda)\bg\|_{B^s_{q,1}(\Omega)} & \leq 
C|\lambda|^{-(1-\frac{\sigma}{2})}\|\bg\|_{B^{s-\sigma}_{q,1}(\Omega)}. \label{est.7.6}
\end{align}
The uniqueness of solutions follows from the existence of solutions to the dual problem. 
Differentiating equations \eqref{eq.7.1},
$$\eta_0\lambda \pd_\lambda\bz - \alpha\Delta\pd_\lambda\pd_\lambda\bz
-\beta\nabla\dv\pd_\lambda \bz = -\eta_0\bz\quad\text{in $\Omega$},
\quad \pd_\lambda\bz|_{\pd\Omega}=0.
$$
By the uniqueness of solutions, we have $\pd_\lambda \bz = -\CU_\Omega(\lambda)(
\eta_0\CU_\Omega(\lambda)\bg)$.  By \eqref{est.7.6}  and Lemma \ref{lem:prod}, we have
\begin{align*}
\|(\lambda, \lambda^{1/2}\bar\nabla, \bar\nabla^2)\pd_\lambda \CU_\Omega(\lambda)\bg
\|_{B^\nu_{q,1}(\Omega)}
& \leq C\|\eta_0\CU_\Omega(\lambda)\bg\|_{B^\nu_{q,1}(\Omega)}
\leq C(\rho_* + \|\tilde\eta_0\|_{B^{N/q}_{q,1}(\Omega)})
\|\CU_\Omega(\lambda)\bg\|_{B^\nu_{q,1}(\Omega)} \\
&\leq C(\rho_* + \|\tilde\eta_0\|_{B^{N/q}_{q,1}(\Omega)})|\lambda|^{-1}
\|\bg\|_{B^\nu_{q,1}(\Omega)} \\
\|(\lambda, \lambda^{1/2}\bar\nabla, \bar\nabla^2)\pd_\lambda \CU_\Omega(\lambda)\bg
\|_{B^\nu_{q,1}(\Omega)}
& \leq C\|\eta_0\bz\|_{B^\nu_{q,1}(\Omega)}
\leq C(\rho_* + \|\tilde\eta_0\|_{B^{N/q}_{q,1}(\Omega)})
\|\CU_\Omega(\lambda)\bg\|_{B^\nu_{q,1}(\Omega)} \\
&\leq C(\rho_* + \|\tilde\eta_0\|_{B^{N/q}_{q,1}(\Omega)})|\lambda|^{-(1-\frac{\sigma}{s})}
\|\bg\|_{B^{s-\sigma}_{q,1}(\Omega)} 
\end{align*}
Thus, we have proved that $\CU_\Omega$ has $(s, \sigma, q, r)$ properties. 
This completes the proof of Theorem \ref{thm:7.1}.
\end{proof}

\section{On the sepctral analysis of the Stokes equations in $\Omega$}
 \label{sec.5}

In view of Propositions \ref{prop:L1} and \ref{prop:L2}, to prove the $L_1$ properties
of solutions to equations \eqref{Eq:Linear}, we have to show the spectram properties
of the the resolvent problem of the Stokes equations, which read
as
\begin{equation}\label{eq.8.1}\left\{\begin{aligned}
\lambda \rho + \eta_0 \dv \bu = f&&\quad&\text{in $\Omega$},  \\
\eta_0\lambda \bu - \alpha\Delta \bu - \beta\nabla \dv\bu + \nabla(P'(\eta_0)\rho) = \bg
&&\quad&\text{in $\Omega\times(0, T)$}, \\
\bu|_{\pd\Omega}=0&.
\end{aligned}\right.\end{equation}  
 Let $\eta_0=\rho_*+\tilde\eta_0$ and we assume that  
the assumption \eqref{july.2} holds.
We shall prove the following theorem.
\begin{thm}\label{thm:8.1} Let $0 < \epsilon < \pi/2$. 
\thetag1 If $\eta_0 = \rho_*$, then  $1 < q < \infty$, and $-1+1/q < s < 1/q$.
\\
\thetag2 If $\tilde\eta_0 \not\equiv0$ and $\tilde\eta_0 \in B^{N/q+1}_{q,1}(\Omega)$,
then  $N-1 < q < 2N$, $1 \leq r \leq \infty-$, 
$-1+N/q \leq s<1/q$.   \\ Let
\begin{equation*}\label{space.8.1}\begin{aligned}
&\CH^s_{q,1}(\Omega) = B^{s+1}_{q,1}(\Omega)\times B^s_{q,1}(\Omega)^N, 
\quad
\|(f, \bg)\|_{\CH^s_{q,1}(\Omega)} = \|f\|_{B^{s+1}_{q,1}(\Omega)} + \|\bg\|_{B^s_{q,1}(\Omega)},
\\
&\CD^s_{q,1}(\Omega) = \{(\rho, \bu) \in B^{s+1}_{q,1}(\Omega)\times B^{s+2}_{q,1}(\Omega)^N
\mid \bu|_{\pd\Omega}=0\}, \quad 
\|(\rho, \bu)\|_{\CD^s_{q,1}(\Omega)} = \|\rho\|_{B^{s+1}_{q,1}(\Omega)}
+ \|\bg\|_{B^{s+2}_{q,1}(\Omega)}.
\end{aligned}\end{equation*}
Then, there exists a large positive number 
$\lambda_3$ and an operator $\CA_\Omega(\lambda) \in 
{\rm Hol}\,(\Sigma_{\epsilon, \lambda_3}, 
\CL(\CH^s_{q,1}(\Omega), \CD^s_{q,1}(\Omega))$ such that 
$(\rho, \bu) = \CA_\Omega(\lambda)(f, \bg)$ is a unique solution of equations
\eqref{eq.8.1} for any $\lambda \in \Sigma_{\epsilon, \lambda_3}$ and 
$(f, \bg) \in \CH^s_{q,1}(\Omega)$, which satisfies the estimate:
$$|\lambda|\|\CA_\Omega(\lambda)(f, \bg)\|_{\CH^s_{q,1}(\Omega)}
+ \|\CA_\Omega(\lambda)(f, \bg)\|_{\CD^s_{q,1}(\Omega)}
\leq C\|(f, \bg)\|_{\CH^s_{q,1}(\Omega)}.
$$
\par
Moreover, there exist three operators  $\CB_v(\lambda)$, $\CC_m(\lambda)$ 
and $\CC_v(\lambda)$
such that 
\begin{enumerate}
\item $\CB_v(\lambda) \in {\rm Hol}\, (\Sigma_{\epsilon, \lambda_3}, 
\CL(B^\nu_{q,r}(\Omega)^N, B^{\nu+2}_{q,r}(\Omega)^N))$, 
\quad
$\CC_m(\lambda) \in {\rm Hol}\, (\Sigma_{\epsilon, \lambda_3}, 
\CL(\CH^s_{q,r}, B^{s+1}_{q,r}(\Omega))$, \quad 
$\CC_v(\lambda) \in {\rm Hol}\, (\Sigma_{\epsilon, \lambda_3}, 
\CL(B^\nu_{q,r}(\Omega)^N, B^{\nu+2}_{q,r}(\Omega)^N))$.  And,
$\CA_\Omega(\lambda)(f, \bg) = (\CC_m(\lambda)(f, \bg), \CB_v(\lambda)\bg
+ \CC_v(\lambda)(f, \bg))$ for any $\lambda \in \Sigma_{\epsilon, \lambda_3}$
and $(f, \bg) \in \CH^s_{q,r}$.
\item $\CB_v(\lambda)$ has a $(s,\sigma,q,1)$ property in $\Omega$,
$\CC_m(\lambda)$ has generalized resolvent properties for $X = \CH^s_{q,1}(\Omega)$
and $Y= B^{s+1}_{q,1}(\Omega)$, 
and $(\lambda, \lambda^{1/2}\bar\nabla, \bar\nabla^2)\CC_v(\lambda)$ 
has  generalized resolvent properties 
for $X = \CH^s_{q,1}(\Omega)$ and $Y=B^s_{q,1}(\Omega)$, respectively. 
\end{enumerate}
\end{thm}
\begin{proof}
In what follows, we shall show the theorem only in the case \thetag2, because 
the case \thetag1 can be proved  in the same argument. 
In \eqref{eq.8.1}, setting $\rho= \lambda^{-1}(f- \eta_0\dv\bu)$
and inserting this formula into the second equations, we have
\begin{equation}\label{eq.8.3} 
\eta_0\lambda \bu - \alpha\Delta \bu - \beta\nabla\dv\bu
-\lambda^{-1}\nabla(P'(\eta_0)\eta_0\dv\bu) = \bg
-\lambda^{-1}\nabla(P'(\eta_0)f)
\quad\text{in $\Omega$}, \quad \bu|_{\pd\Omega}=0.
\end{equation}
For a while, setting $\bh = \bg - \lambda^{-1}\nabla(P'(\eta_0)f)$, we shall consider
equations:
\begin{equation}\label{eq.8.4}
\eta_0\lambda \bu - \alpha\Delta \bu - \beta\nabla\dv\bu
-\lambda^{-1}\nabla(P'(\eta_0)\eta_0\dv\bu) = \bh \quad\text{in $\Omega$}, 
\bu|_{\pd\Omega}=0.
\end{equation}
Let $\lambda_2$ and $\CU_\Omega(\lambda)$ be the constant and the operator
 given in Theorem \ref{thm:7.1}.
Set $\bu = \CU_\Omega(\lambda)\bh$ and insert this formula into \eqref{eq.8.4}
to obtain
\begin{equation*}\label{eq.8.5}
\eta_0\lambda \bu - \alpha\Delta \bu - \beta\nabla\dv\bu
= (\bI  - \lambda^{-1}\CP(\lambda))\bh
\quad\text{in $\Omega$}, \quad \bu|_{\pd\Omega}=0
\end{equation*}
where we have set
$$\CP(\lambda)\bh = -\nabla(P'(\eta_0)\eta_0\dv\CU_\Omega(\lambda)\bh).$$
We will show that 
\begin{equation}\label{est.p1}
\|\CP(\lambda)\bh\|_{B^s_{q,1}(\Omega)}
\leq C(\rho_*, \|\tilde\eta_0\|_{B^{s+1}_{q,1}(\Omega)})
\|\CU_\Omega(\lambda)\bh\|_{B^{\nu+2}_{q,1}(\Omega)}.
\end{equation}
Here and in what follows, $C(\rho_*, 
\|\tilde\eta_0\|_{B^{s+1}_{q,1}(\Omega)})$ denotes some constant 
depending on $\rho_*$
and $\|\eta_0\|_{B^{s+1}_{q,1}(\Omega)}$. \par
To this end, we shall use Lemma \ref{lem:Hasp} and the fact that
$B^{s+1}_{q,1}(\Omega)$ is a Banach algebra. 
In fact, noting that $N/q \leq s+1$, by Lemma \ref{lem:prod}, we have
\begin{align*}
\|uv\|_{B^{s+1}_{q,1}(\Omega)} &\leq \|(\nabla u)v\|_{B^s_{q,1}(\Omega)}
+\|u(\nabla v)\|_{B^s_{q,1}(\Omega)} +\|uv\|_{B^s_{q,1}(\Omega)} \nonumber \\
&\leq C(\|u\|_{B^{s+1}_{q,1}(\Omega)}\|v\|_{B^{N/q}_{q,1}(\Omega)}
+ \|u\|_{B^{N/q}_{q,1}(\Omega)}\|v\|_{B^{s+1}_{q,1}(\Omega)}
+ \|u\|_{B^s_{q,1}(\Omega)}\|v\|_{B^{N/q}_{q,1}(\Omega)} \nonumber \\
&\leq C\|u\|_{B^{s+1}_{q,1}(\Omega)}\|v\|_{B^{s+1}_{q,1}(\Omega)}.
\label{banach.1}
\end{align*}
\par
To prove \eqref{est.p1},  recalling that $\eta_0 = \rho_* +\tilde\eta_0$, 
we write  $P'(\eta_0)\eta_0 = P'(\rho_*)\rho_* + \CP_1(\lambda)$, 
where we have set
$$\CP_1(r)= P'(\rho_*)r + \int^1_0P''(\rho_* + \theta r)\,d\theta r(\rho_*+r)
$$
with $r = \tilde\eta_0$.  Note that $\CP_1(0)=0$ and 
$\rho_1-\rho_* \leq \tilde \eta_0(x) \leq \rho_2-\rho_*$ as follows 
from \eqref{july.2}.  By Lemma \ref{lem:Hasp}, we have
$$\|\CP_1(\tilde\eta_0)\|_{B^{s+1}_{q,1}(\Omega)}
\leq C\|\tilde\eta_0\|_{B^{s+1}_{q,1}(\Omega)}.
$$
Thus, 
\begin{equation}\label{est.7.1}\begin{aligned}
\|\nabla(P'(\eta_0)\eta_0\dv\bu\|_{B^s_{q,1}(\Omega)}
&\leq |P'(\rho_*)\rho_*|\|\nabla\dv\bu\|_{B^s_{q,1}(\Omega)}
+ \|\CP_1(\tilde\eta_0)\dv\bu\|_{B^{s+1}_{q,1}(\Omega)}\\
&\leq |P'(\rho_*)\rho_*|\|\bu\|_{B^{s+2}_{q,1}(\Omega)}
+ C\|\tilde \eta_0\|_{B^{s+1}_{q,1}(\Omega)}\|\dv\bu\|_{B^{s+1}_{q,1}(\Omega)}.
\end{aligned}\end{equation}
This proves \eqref{est.p1}.  \par

Since $\|\CU_\Omega(\lambda)\bh\|_{B^{s+2}_{q,1}(\Omega)} \leq 
C\|\bh\|_{B^s_{q,1}(\Omega)}$ for any $\lambda \in 
\Sigma_{\epsilon, \lambda_2}$ as follows from Theorem \ref{thm:7.1},
it follows from \eqref{est.p1}  that 
\begin{equation}\label{est.p2}\|\lambda^{-1}\CP(\lambda)\bh\|_{B^s_{q,1}(\Omega)}
\leq |\lambda|^{-1} C(\rho_*, \|\tilde\eta_0\|_{B^{s+1}_{q,1}(\Omega)})
\|\bh\|_{B^{s}_{q,1}(\Omega)}
\end{equation}
Choosing $\lambda_3 \geq \lambda_2$ so large that 
$$\lambda_3^{-1}C(\rho_*, \|\tilde\eta_0\|_{B^{s+1}_{q,1}(\Omega)})\leq 1/2,$$
we see that $\|\lambda^{-1}\CP(\lambda)\bh\|_{B^s_{q,1}(\Omega)} 
\leq (1/2)\|\bh\|_{B^s_{q,r}(\Omega)}$ for any $\lambda \in 
\Sigma_{\epsilon, \lambda_3}$.  Thus,  $(\bI -\lambda^{-1}\CP(\lambda))^{-1}$
exists as an element of $\CL(B^s_{q,1}(\Omega)^N)$ and its operator norm
does not exceed $2$.  
Obvisouly,  $\bu = \CU_\Omega(\lambda)(\bI-\lambda^{-1}
\CP(\lambda))^{-1}\bh$ solves equations \eqref{eq.8.4}
uniquely.  In fact, the uniqueness follows from the existence theorem
of the dual problem. \par
We define an operator $\CB(\lambda)$ by
$$\CB(\lambda)(f, \bg) = 
\CU_\Omega(\lambda)(\bI-\lambda^{-1}
\CP(\lambda))^{-1}
(\bg- \lambda^{-1}\nabla(P'(\eta_0)f). $$
Obvisouly, $\bu = \CB(\lambda)(f, \bg)$
is a solution of equations \eqref{eq.8.3}. 
Let $\CC_m(\lambda)$ be an operator defined by
$$\CC_m(\lambda)(f, \bg) = \lambda^{-1}(f- \eta_0\dv \bu) 
= \lambda^{-1}(f- \eta_0\dv \CB(\lambda)(f, \bg)),
$$
 then $\rho = \CC_m(\lambda)(f, \bg)$ and 
$\bu = \CB(\lambda)(f, \bg)$ are  solutions of equations \eqref{eq.8.1}
for $\lambda \in \Sigma_{\epsilon, \lambda_3}$.  The uniqueness of 
equations \eqref{eq.8.1} follows from the uniqueness of solutions of
equations \eqref{eq.8.4}. In particular, we define 
$\CA_\Omega(\lambda)$ by 
$$\CA_\Omega(\lambda)(f, \bg) = ( \CC_m(\lambda)(f, \bg), 
\CB(\lambda)(f, \bg)).
$$
Obvisously, $\CA_\Omega(\lambda) \in {\rm Hol}\, (\Sigma_{\epsilon, \lambda_5},
\CL(\CH^s_{q,1}(\Omega), \CD^s_{q,1}(\Omega))$ and $(\rho, \bu) = 
\CA_\Omega(\lambda)(f, \bg)$ is a unique solution of equations \eqref{eq.8.1}.
The uniqueness follows from the uniqueness of solutions to \eqref{eq.8.4}.
\par
We now estimate $\CA_\Omega(\lambda)$. 
Employing the similar argument as in the proof of \eqref{est.p1}, 
we have
\begin{equation}\label{est.p4}
\|\nabla(P'(\eta_0)f)\|_{B^s_{q,1}(\Omega)} 
\leq C(\rho_*, \|\tilde\eta_0\|_{B^{s+1}_{q,1}(\Omega)})\|f\|_{B^{s+1}_{q,1}(\Omega)}.
\end{equation}
Using Theorem \ref{thm:7.1} and \eqref{est.p4}, we have
\begin{equation}\label{est.b.1}\begin{aligned}
&|\lambda|\|\CB(\lambda)(f, \bg)\|_{B^s_{q,1}(\Omega)}
+ \|\CB(\lambda)(f, \bg)\|_{B^{s+2}_{q,1}(\Omega)} \\
&\leq C\|(\bI -\lambda^{-1}\CP(\lambda))^{-1}
(\bg - \lambda^{-1}\nabla(P'(\eta_0)f))\|_{B^s_{q,1}(\Omega)} \\
& \leq C(\|\bg\|_{B^s_{q,1}(\Omega)} 
+|\lambda|^{-1}\|P'(\eta_0)f\|_{B^{s+1}_{q,1}(\Omega)})
\\
& \leq C(\rho_*, \|\tilde\eta_0\|_{B^{s+1}_{q,1}(\Omega)})\|(f, \bg)\|_{\CH^s_{q,1}}.
\end{aligned}\end{equation}
Moreover, we have
\begin{equation}\label{est.p5*}\begin{aligned}
\|\CC_m(\lambda)(f, \bg)\|_{B^{s+1}_{q,1}(\Omega)}
&\leq |\lambda|^{-1}(\|f\|_{B^{s+1}_{q,1}(\Omega)} + \|\eta_0\dv\CB(\lambda)(f, 
\bg)\|_{B^{s+1}_{q,1}(\Omega)})\\
&\leq |\lambda|^{-1}C(\rho_*, \|\tilde\eta_0\|_{B^{s+1}_{q,1}(\Omega)})
\|(f, \bg)\|_{\CH^s_{q,1}}.
\end{aligned}\end{equation}
Thus, we have 
$$|\lambda|\|\CA_\Omega(\lambda)(f, \bg)\|_{\CH^s_{q,1}(\Omega)}
+\|\CA_\Omega(\lambda)(f, \bg)\|_{\CD^s_{q,1}(\Omega)}
\leq C(\rho_*, \|\tilde\eta_0\|_{B^{s+1}_{q,1}(\Omega)})\|(f, \bg)\|_{\CH^s_{q,1}}.
$$
\par
We now consider the second assertions of Theorem \ref{thm:8.1}.  
 By the Neumann series expansion, 
we have
$$(\bI - \lambda^{-1}\CP(\lambda))^{-1}
= \bI - \lambda^{-1}\CP(\lambda)(\bI - \lambda^{-1}\CP(\lambda))^{-1}.$$
In view of this formula, we define operators $\CB_v(\lambda)$ and $\CC_v(\lambda)$ by
\begin{align*}
\CB_v(\lambda)\bg &= \CU_\Omega(\lambda)\bg, \\
\CC_v(\lambda)(f, \bg) & = -\CU_\Omega(\lambda)(\lambda^{-1}\nabla (P'(\eta_0)f))
-\CU_\Omega(\lambda)(\lambda^{-1}\CP(\lambda)(\bI - \lambda^{-1}\CP(\lambda))^{-1})
(\bg - \lambda^{-1}\nabla(P'(\eta_0)f)).
\end{align*}
Then, we have $\CB(\lambda)(f, \bg) = \CB_v(\lambda)\bg + \CC_v(\lambda)(f, \bg)$.
By Theorem \ref{thm:7.1}, we see that $\CB_v(\lambda)$ has $(s,\sigma, q, 1)$ properties.
  By Theorem \ref{thm:7.1}, \eqref{est.p1},  \eqref{est.p2},
and the fact that $\|(\bI - \lambda^{-1}\CP(\lambda))^{-1}\|_{\CL(B^s_{q,1}(\Omega))} \leq 2$, 
we have
\begin{equation}\label{est.p5}
\|(\lambda, \lambda^{1/2}\bar\nabla, \bar\nabla^2)\CC_v(\lambda)(f, \bg)\|_{B^s_{q,1}(\Omega)}
\leq C(\rho_*, \|\bar\eta_0\|_{B^{s+1}_{q,1}(\Omega)})|\lambda|^{-1}\|(f, \bg)\|_{B^s_{q,1}}.
\end{equation}
Since $\pd_\lambda \CP(\lambda)\bh = \nabla(P'(\eta_0)\eta_0\dv \pd_\lambda\CU_\Omega(\lambda)
\bh)$, using the similar argument to \eqref{est.7.1}, we have
\begin{align*}
\|\pd_\lambda\CP(\lambda)\bh\|_{B^s_{q,1}(\Omega)}& \leq 
C(\rho_*, \|\tilde\eta_0)\|_{B^{s+1}_{q,1}(\Omega)}|\lambda|^{-1}\|\bh\|_{B^s_{q,1}(\Omega)}, \\
\|\pd_\lambda(\bI - \lambda^{-1}\CP(\lambda))^{-1}\bh\|_{B^s_{q,1}(\Omega)}
& \leq \|(\bI-\lambda^{-1}\CP(\lambda))^{-2}(-\lambda^{-2}\CP(\lambda)
+ \lambda^{-1}\pd_\lambda\CP(\lambda))\bh\|_{B^s{q,1}(\Omega)} \\
& \leq |\lambda|^{-2} C(\rho_*, \|\tilde\eta_0\|_{B^{s+1}_{q,1}(\Omega)})\|\bh\|_{B^s_{q,1}(\Omega)}.
\end{align*}
Since $\CU_\Omega(\lambda)$ has $(s, \sigma, q, 1)$ properties, 
and since we may asuume that 
$\lambda_3 \geq 1$,  we have 
\begin{equation}\label{est.p6}
\|(\lambda, \lambda^{1/2}\bar\nabla, \bar\nabla^2)\pd_\lambda
\CC_v(\lambda)(f, \bg)\|_{B^s_{q,1}(\Omega)}
\leq C(\rho_*, \|\tilde\eta_0\|_{B^{s+1}_{q,1}(\Omega)})
|\lambda|^{-2}\|(f, \bg)\|_{\CH^s_{q,1}}
\end{equation}
for $\lambda \in \Sigma_{\epsilon, \lambda_3}$. Combining  \eqref{est.p5} and \eqref{est.p6}, 
we see that  $(\lambda, \lambda^{1/2}\bar\nabla, \bar\nabla^2)
\CC_v(\lambda)$ has generalized resolvent properties for 
$X=\CH^s_{q,1}(\Omega)$ and 
and $Y =B^{s}_{q,1}(\Omega)$. \par
Since
$$\pd_\lambda\CC_m(\lambda)(f, \bg) = -\lambda^{-2}(f- \eta_0\dv\CB(\lambda)(f, \eta_0\bg))
-\lambda^{-1}\eta_0\dv(\pd_\lambda\CB(\lambda)(f, \eta_0\bg)),
$$
we have
\begin{equation}\label{est.b.2}\begin{aligned}
\|\pd_\lambda\CC_m(\lambda)(f, \bg)\|_{B^{s+1}_{q,1}(\Omega)} &
 \leq |\lambda|^{-2}(\|f\|_{B^{s+1}_{q,1}(\Omega)} 
+ C(\rho_*, \|\tilde\eta_0\|_{B^{s+1}_{q,1}(\Omega)})
\|\CB(\lambda)(f, \bg)\|_{B^{s+2}_{q,1}(\Omega)}) \\
&+ |\lambda|^{-1}C(\rho_*, \|\tilde\eta_0\|_{B^{s+1}_{q,1}(\Omega)})
\|\pd_\lambda \CB(\lambda)(f, \bg)\|_{B^{s+2}_{q,1}(\Omega)}).
\end{aligned}\end{equation}
Recalling that $\CB(\lambda)(f, \bg) = \CU_\Omega(\lambda)\bg 
+ \CC_v(\lambda)(f, \bg)$, by Theorem \ref{thm:7.1}, \eqref{est.p5}, and 
\eqref{est.p6}, we have
\begin{equation}\label{est.b.3}
\|\pd_\lambda \CB(\lambda)(f, \bg)\|_{B^{s+2}_{q,1}(\Omega)} \leq 
C(\rho_*, \|\tilde\eta_0\|_{B^{s+1}_{q,1}(\Omega)})|\lambda|^{-1}
\|(f, \bg)\|_{\CH^s_{q,1}(\Omega)}.
\end{equation}
Putting  \eqref{est.b.1}, \eqref{est.b.2}, and \eqref{est.b.3} gives 
$$\|\pd_\lambda \CC_m(\lambda)(f, \bg)\|_{B^{s+1}_{q,1}(\Omega)}
\leq C(\rho_*, 
\|\tilde\eta_0\|_{B^{s+1}_{q,1}(\Omega)})|\lambda|^{-2}\|(f, \bg)\|_{\CH^s_{q,1}(\Omega)}.
$$
Combining this estimate with \eqref{est.p5*}, we see that $\CC_m(\lambda)$ has
a generalized resolvent properties for $X=Y = B^{s+1}_{q,1}(\Omega)$.
This completes the proof of Theorem \ref{thm:8.1}.
\end{proof}
\section{On the $L_1$ maximal regularity of the Stokes seqmigroup in $\Omega$,
A proof of Theorem \ref{thm:main.1}}
\label{sec:6}
In this section, we consider equations \eqref{Eq:Linear}. 
We first consider equations \eqref{eq.8.1}. For $\nu \in \{s-\sigma, s, s+\sigma\}$, 
let $\CH^\nu_{q,r}(\Omega)$ and $\CD^\nu_{q,r}(\Omega)$ 
be the spaces defined in Theorem \ref{thm:8.1}. 
Let $\CA$ be an operator defined by 
\begin{equation*}\label{semi.1}
\CA(\rho, \bu) = (\eta_0\dv\bu, -\eta_0^{-1}(\alpha\Delta \bu + 
\beta\nabla\dv\bu - \nabla (P'(\eta_0)\rho))
\end{equation*}
for $(\rho, \bu) \in \CD^\nu_{q,r}$.  Then, problem \eqref{eq.8.1} is written as 
\begin{equation*}\label{eq.8.2}
(\lambda \bI + \CA)(\rho, \bu) = (f, \eta_0(x)^{-1}\bg).
\end{equation*}
When $\tilde\eta_0\not\equiv0$, the operation $\eta_0(x)^{-1}\cdot$
is guaranteed by the following lemma. 
\begin{lem}\label{lem:operation}  Assume that $\tilde\eta_0\not\equiv0$.
Let $N-1 < q < 2N$ and $-1+N/q\leq s , 1/q$. Then, for any $\bu \in B^s_{q,1}(\Omega)$, 
there holds 
\begin{equation}\label{op.1}
\|u\eta_0^{-1}\|_{B^s_{q,r}(\Omega)} \leq \rho_*^{-1}\|u\|_{B^s_{q,r}(\Omega)}
+ C\|\tilde\eta_0\|_{B^{N/q}_{q,1}(\Omega)}\|u\|_{B^s_{q,r}(\Omega)}
\end{equation}
for some constant $C>0$ depending on $\rho_*$, $\rho_1$ and $\rho_2$.
\end{lem}
\begin{proof}Note that 
$\eta_0(x)^{-1} = \rho_*^{-1}-\tilde\eta_0(x)(\rho_*\eta_0(x))^{-1}$. 
If $q_1 > N$, then 
\begin{equation}\label{op.2}
\|\tilde\eta_0(x)(\rho_*\eta_0(x))^{-1}\|_{B^{N/q_1}_{q_1, \infty}(\Omega)} 
\leq C\|\tilde\eta_0\|_{B^{N/q_1}_{q_1, \infty}(\Omega)}.
\end{equation}
In fact,to prove \eqref{op.2},
we use the relation $B^{N/q_1}_{q_1,\infty}(\Omega) = (L_{q_1}(\Omega), W^1_{q_1}(\Omega))_{N/q_1, \infty}$.
Since $\rho_1 < \eta_0(x) < \rho_2$ as follows from \eqref{july.2}, 
we have
$$
\|\tilde\eta_0(x)(\rho_*\eta_0(x))^{-1}\|_{L_{q_1}(\Omega)} \leq (\rho_*\rho_1)^{-1}
\|\tilde\eta_0\|_{L_{q_1}(\Omega)}.
$$
And also, 
$$\|\nabla (\tilde\eta_0(\rho_*\eta_0)^{-1})\|_{L_{q_1}(\Omega)}
\leq \|(\nabla \tilde\eta_0) (\rho_*\eta_0)^{-1}\|_{L_{q_1}(\Omega)}
+ \rho_*^{-1}\|\tilde\eta_0(\nabla \eta_0)\eta_0^{-2}\|_{L_{q_1}(\Omega)}.
$$
Noticing that $\nabla \eta_0=\nabla\tilde\eta_0$ and that 
$|\tilde\eta_0(x)| \leq |\eta_0(x)| + \rho_* \leq \rho_2+\rho_*$, we have
$$\|\nabla (\tilde\eta_0(\rho_*\eta_0)^{-1})\|_{L_{q_1}(\Omega)}
\leq ((\rho_*\rho_1)^{-1}+ \rho_*^{-1}(\rho_2+\rho_*)
\rho_1^{-2})\|\nabla\tilde\eta_0\|_{L_{q_1}(\Omega)}.
$$
Thus, there exists a constant $C$ depending 
on $\rho_*$, $\rho_1$ and $\rho_2$ such that 
\eqref{op.2} holds. \par
Now, we shall prove \eqref{op.1}. 
First, we consider the case where $N/q < 1$. Then, using Abidi-Paicu-Haspot estimate (
\cite[Cor.2.5]{AP07} and \cite[Corollary 1]{H11}), we have
\begin{align*}
\|u\eta_0^{-1}\|_{B^s_{q,r}(\Omega)} &\leq 
(\rho_*^{-1}\|u\|_{B^s_{q,r}(\Omega)} 
+ \|\tilde\eta_0(\rho_*\eta_0)^{-1}\|_{B^{N/q}_{q, \infty}(\Omega)
\cap L_\infty(\Omega)}\|u\|_{B^s_{q,r}(\Omega)})\\
&\leq (\rho_*^{-1}+ C\|\tilde\eta_0\|_{B^{N/q}_{q,1}(\Omega)})\|u\|_{B^s_{q,r}(\Omega)}.
\end{align*}
Next, we consider the case where $N/q \geq 1$.  Since $-1 + N/q \leq s < 1/q$, 
  if we choose $q_1$ in such a way that $N < q_1 < qN$, then 
$s \in (-N/q_1, N/q_1)$ and $s \in (-N/q', N/q_1)$. Thus, since $N/q_1 < 1$, 
using Abidi-Paicu-Haspot estimate
and \eqref{op.2}
we have
\begin{align*}
\|u\eta_0^{-1}\|_{B^s_{q,r}(\Omega)} &\leq 
(\rho_*^{-1}\|u\|_{B^s_{q,r}(\Omega)} + \|\tilde\eta_0(\rho_*\eta_0)^{-1}\|_{B^{N/q_1}_{q_1, \infty}(\Omega)
\cap L_\infty(\Omega)}\|u\|_{B^s_{q,r}(\Omega)}) \\
& \leq (\rho_*^{-1}\|u\|_{B^s_{q,r}(\Omega)} + C\|\tilde\eta_0\|_{B^{N/q_1}_{q_1, \infty}(\Omega)}
+ (\rho_*\rho_1)^{-1}(\rho_*+\rho_2))\|u\|_{B^s_{q,r}(\Omega)}.
\end{align*}
Notice that $1 < q \leq  N <  q_1$. 
By the embedding theorem of the Besov spaces, we have
\begin{align*}
\|\tilde\eta_0\|_{B^{N/q_1}_{q_1, \infty}(\Omega)}
\leq C\|\tilde\eta_0\|_{B^{\frac{N}{q_1}+N\left(\frac{1}{q}-\frac{1}{q_1}\right)}_{q,1}(\Omega)}
=  C\|\tilde\eta_0\|_{B^{N/q}_{q,1}(\Omega)}.
\end{align*}
This completes the proof of Lemma \ref{lem:operation}.
\end{proof}
If we consider the resolvent equation: $(\lambda\bI+\CA)(\rho, \bu) = (f, \bg)$,
then by Thereoms \ref{thm:8.1}, we see that  the resolvent set 
$\rho(\CA) \supset \Sigma_{\epsilon, \lambda_3}$ and the resolvent is written as 
$$(\lambda\bI+\CA)^{-1}(\rho, \bg) = \CA_\Omega(\lambda)(f, \eta_0\bg)$$
for any $\lambda \in \Sigma_{\epsilon, \lambda_3}$ and $(f, \bg)\in \CH^s_{q,1}(\Omega)$. 
Thus,  in view of Theorem \ref{thm:8.1} and   the standard semigroup theorem (cf. Yosida \cite{YK}), 
$\CA$ generates a $C_0$ analytic semigroup $\{T(t)\}_{t\geq 0}$, and 
for any $(\rho_0, \bu_0) \in \CH^s_{q,1}(\Omega)$, $(\rho, \bu) =T(t)(\rho_0, \bu_0)$
is a unique solution of equations \eqref{Eq:Linear} in the case where $F=0$ and $\bG=0$. 
\\
{\bf A proof of Theorem \ref{thm:main.1}.}
Let $(\theta, \bv) = (\lambda+\CA)^{-1}(f, \bg) = \CA_\Omega(\lambda)(f, \eta_0\bg)$. 
By the standard analytic semigroup theory, we see that $T(t)(f, \bg)
 = \CL^{-1}[(\lambda\bI + \CA)^{-1}]
= \CL^{-1}[\CA_\Omega(\lambda)(f, \eta_0\bg)]$.   Let $\CC_m(\lambda)$, $\CB_v(\lambda)$, 
and $\CC_v(\lambda)$ be the operators given in Theorem \ref{thm:8.1}.  Let 
$T_v(t)(f, \bg) = \CL^{-1}[\CC_m(\lambda)(f, \eta_0\bg)]$, 
$T_m^1(t)\bg = \CL^{-1}[\CB_v(\lambda)\eta_0\bg]$, and 
$T^2_v(t)(f, \bg) = \CL^{-1}[\CC_v(\lambda)(f, \eta_0\bg)]$.  By Theorem \ref{thm:8.1},
we have $T(t)(f, \bg) = T_m(t)(f, \bg) + T_v^1(t)\bg + T_v^2(t)(f, \bg)$.
Since $\CB_v(\lambda)$ has  $(s, \sigma, q,1)$ properties in $\Omega$, 
$\CC_m(\lambda)$ has  generalized resolvent properties for $X=\CH^s_{q,1}(\Omega)$ and 
$Y = B^{s+1}_{q,1}(\Omega)$, and $\bar\nabla^2\CC_v(\lambda)$ has generalized 
resolvent properties for $X=\CH^s_{q,1}(\Omega)$ and $Y = B^s_{q,1}(\Omega)^N$. 
Thus, by Propositions \ref{prop:L1} and \ref{prop:L2}, we see that 
$$\int^\infty_0 e^{-\gamma t}\|T(t)(f, \bg)\|_{\CD^s_{q,1}(\Omega)}\,dt
\leq C(\rho_*, \|\tilde\eta_0\|_{B^{N/q}_{q,1}(\Omega)})\|(f, \bg)\|_{\CH^s_q(\Omega)}.
$$
for any $\gamma > \lambda_2$. 
By the Duhamel principle, solutions $(\rho, \bv)$ 
to equations \eqref{Eq:Linear} can be written as
$$(\rho, \bv) = T(t)(\rho_0, \bv_0) + \int^t_0 T(t-\tau)(F(\cdot, \tau), \eta_0\bG(\cdot, \tau))\,d\tau.$$
Thus, by Fubini's theorem we see that
$$\int^\infty_0 e^{-\gamma t}\|\rho(t), \bv(t)\|_{\CD^s_{q,1}(\Omega)}\,dt
\leq C(\rho_*, \|\tilde\eta_0\|_{B^{N/q}_{q,1}(\Omega)})
(\|(\rho_0, \bv_0\|_{\CH^s_{q,1}(\Omega)}
+\int^\infty_0e^{-\gamma t}\|F(t), \bG(t)\|_{\CH^s_{q,1}(\Omega)}\,dt).
$$

Concerning the estimates of the time derivative, we use the equtions:
$\pd_t\rho = -\eta_0\dv\bv + F$ and $\pd_t\bv = \eta_0^{-1}(\alpha \Delta \bv
+ \beta\nabla\dv\bv - \nabla(P'(\eta_0)\rho) + \bG$, and then we have 
\begin{align*}
&\int^\infty_0e^{-\gamma t}\|(\pd_t\rho(t), \pd_t\bv(t))\|_{\CH^s_{q,1}(\Omega)}\,dt \\
&\quad \leq C(\rho_*, \|\tilde\eta_0\|_{B^{N/q}_{q,1}(\Omega)})
(\|(\rho_0, \bv_0\|_{\CH^s_{q,1}(\Omega)}
+\int^\infty_0e^{-\gamma t}\|F(t), \bG(t)\|_{\CH^s_{q,1}(\Omega)}\,dt).
\end{align*}
This completes the proof of Theorem \ref{thm:main.1}.
\\
\\
\\
{\bf Data Availability Statement:} Not applicable 
\\ \\
{\bf Conflicts of Interest:} The authors declare no conflict of interest.


\end{document}